\documentclass[reqno]{amsart}
\usepackage[utf8]{inputenc}
\usepackage{amscd,amsfonts,amsmath,amssymb,amsthm}
\usepackage{bbm,dsfont,centernot,mathtools}
\usepackage{color,mathrsfs,stackrel}
\usepackage[margin=1in]{geometry}
\usepackage[colorlinks=true,linkcolor=blue,citecolor=red]{hyperref}

\newtheorem{corollary}{Corollary}[section]
\newtheorem{lemma}[corollary]{Lemma}
\newtheorem{proposition}[corollary]{Proposition}
\newtheorem{theorem}[corollary]{Theorem}

\theoremstyle{definition}
\newtheorem{definition}[corollary]{Definition}
\newtheorem{remark}[corollary]{Remark}

\newtheorem*{acknowledgements}{\sc Acknowledgements}

\numberwithin{equation}{section}

\definecolor{ddorange}{rgb}{1,0.5,0}
\definecolor{ddcyan}{rgb}{0,0.2,1.0}

\allowdisplaybreaks

\def\Lin {\mathop {\rm Lin}\nolimits}

\def\div {\mathop {\rm div}\nolimits}

\def\tr {\mathop {\rm tr}\nolimits}
\def\de {\mathrm{d}}
\def\R {\mathbb R}
\def\N {\mathbb N}
\def\M {\mathbb M^{n\times n}_{\rm sym}}
\def\C {\mathbb C}
\def\B {\mathbb B}
\def\EE {\mathrm{E}}

\begin{document}

\title[A rate-independent model for geomaterials under compression]{A rate-independent model for geomaterials under compression coupling strain gradient plasticity and damage
% : time regularity of energetic solutions and vanishing-viscosity approximation
} 
% BV solutions for geomaterials-complete damage and strain gradient plasticity}

\author[M. Caponi]{Maicol Caponi}
\address[Maicol Caponi]{Dipartimento di Matematica e Applicazioni ``R. Caccioppoli''
\newline\indent Università degli Studi di Napoli ``Federico II'' 
\newline\indent Via Cintia, Monte S. Angelo, 80126 Naples, Italy}
\email{maicol.caponi@unina.it}

\author[V. Crismale]{Vito Crismale}
\address[Vito Crismale]{Dipartimento di Matematica ``Guido Castelnuovo''
\newline\indent Sapienza Università di Roma 
\newline\indent Piazzale Aldo Moro 2, I-00185 Roma, Italy}
\email{vitocrismale@uniroma1.it}

\begin{abstract}
We study a strain gradient-enhanced version of a model for geomaterials under compression by Marigo and Kazymyrenko (2019) coupling damage and small-strain associative plasticity. We prove that the jumps in time of the plastic variable may happen only along jumps of the damage variable. Moreover, we perform a vanishing-viscosity analysis showing existence of Balanced Viscosity quasistatic solutions \emph{à la} Mielke-Rossi-Savaré.
\end{abstract}

\maketitle

\noindent
{\bf Keywords}: variational models, quasistatic evolution, energetic solutions, vanishing-viscosity,
Drucker-Prager elasto-plasticity, gradient damage models, complete damage, strain gradient plasticity.

\medskip

\noindent
{\bf MSC 2020}: 
74C05, % Small-strain, rate-independent theories (including rigid-plastic and elasto-plastic materials)
74R05, % Brittle damage
74G65, % Energy minimization in equilibrium problems in solid mechanics
35Q74, % PDEs in connection with mechanics of deformable solids
49J45. % Methods involving semicontinuity and convergence; relaxation

%----------
%   Main   
%----------

\maketitle

%--------------
%   Notation   
%--------------

\section{Introduction}

In this paper we study a strain gradient-enhanced version of a model for geomaterials under compression by Marigo and Kazymyrenko~\cite{KazMar18} coupling damage and small-strain associative plasticity, proving a better regularity in time for the corresponding evolutions. 

The model of~\cite{KazMar18} is based on a micromechanical analysis, accounting for microcracks inside the material as the main responsible of the progressive loss of rigidity and then for the macroscopic elasto-plastic coupled behavior.

% More precisely, 
Given a \emph{reference domain} $\Omega\subset\R^n$ and denoting by $u\colon \Omega\to \R^n$ the \emph{displacement field}, in the general framework of small-strain elasto-plasticity the total strain $\EE u = (\nabla u)^{\mathrm{sym}} \in \M$ is additively decomposed into \emph{elastic strain} and \emph{plastic strain}:
$$\EE u=e+p$$
and the \emph{stress tensor} $\sigma\colon \Omega\to \M$ depends linearly on $e$ according to Hooke's law 
$$\sigma=\C e,$$
where $\C$ is the fourth order positive definite \emph{Hooke's tensor}, with $\sigma$ lying in a fixed closed and convex set $K \subset \M$ (the \emph{constraint set}) and satisfying equilibrium conditions involving the external loads.

In order to describe the model in~\cite{KazMar18}, consider the reference case of a triaxial test - a shear compression plus a compression normal to the shear plane - and the observed phenomenon of \emph{dilatance}, i.e., the irreversible increase of volume in compression. A micromechanical justification is that in the shear compression the volume may increase with that of the internal voids between the granuli, interpreted as microcracks, due to a less efficient particle organization; conversely, in the normal compression a part of these microcracks may be closed, preventing a free sliding and thus the free relaxation to the initial configuration in unloading.

Then a \emph{damage variable} $\alpha\colon [0,T]\times\Omega\to [0,1]$ is considered in~\cite{KazMar18}, related to the density
% related with the presence 
of closed microcracks and reflecting into an internal blocked energy depending on plastic strain, 
% since it is 
being related
% associated
to irreversible deformations.  
Assuming Coulomb law for the sliding with friction between the lips of closed microcracks, plasticity follows a standard Drucker-Prager criterion with an associative flow rule, that is the plastic flow is normal to the constraint set $K$.
% $$K=\{\sigma\in \M \colon \tau \sigma_m + |\sigma_D| -k\leq 0\},$$ 
% for $\tau$, $k>0$ and $\sigma= \sigma_m \mathrm{Id}+ \sigma_D$, $\sigma_m$ being the \emph{mean stress} $\frac{\mathrm{tr}\,\sigma}{n}$. 
Moreover, damage is assumed to be an irreversible process (consequently $\alpha$ is nonincreasing in time) and the growth of the microcracks is modelled by the terms usually present in gradient damage models
$$D(\alpha(t))+ l_\alpha \|\nabla \alpha(t)\|_{L^2}^2,$$
% $\varepsilon\in(0,1)$, 
$l_\alpha>0$, cf.~\cite{CRASII},~\cite[Section~4]{AMV},~\cite[Remark~3.1]{KazMar18}.

In~\cite{Cr3} the existence of quasistatic evolutions for such a model is proven, fulfilling the notion of \emph{energetic solutions} by Mielke and Theil~\cite{MieThe99MMRI}: this is based on a \emph{global stability condition}, which prescribes at each time the minimization, among the admissible configurations, of the total internal energy plus the dissipation potential, and on an \emph{energy-dissipation balance} between the total variation in time of the internal energy, the total dissipated energy, and the work of the external loadings. 

In~\cite{UllWamAleSamDegFra21} the model has been enriched by including, still employing an irreversible degradation function $\alpha$, also the tensile response. This accounts for the fact that the evolution of opening microcracks in tension leads to brittle behavior and mode I fracture. Conversely, the response in compression shares many features with~\cite{KazMar18}, but plasticity is required to satisfy a non-associative flow rule.

The present contribution focuses on a different modification of the model in~\cite{KazMar18}, obtained by formally including in the internal energy the $L^2$ norm of the plastic strain gradient, times a dimensional (length) constant. Terms involving higher order derivatives are considered on the one hand by their capability to capture the correct scale size of shear bands, usually observed in failure in soils induced by shearing, and on the other hand to avoid mesh dependence in numerical simulations.
Inspired by the work of Aifantis~\cite{Aif84}, several plastic strain gradient models have been proposed, for both crystalline and granular materials (e.g.~\cite{FleHut01, EGB03, Peerlings07, Forest09, PohPeeGeeSwa11, AnaAslChe12, MAM13, ZreKal16}).

The framework of our strain gradient enhancement is that of the \emph{explicit gradient theories}, in which the gradient terms typically evolve according to ordinary differential equations determined by constitutive relations; in particular we follow the lines of~\cite{FleHut01, Forest09, PohPeeGeeSwa11}. 

Therefore, the stored energy at time $t$ is the sum of the elastic energy of the sound material, of the kinematical hardening term (depending on $\alpha$), and of the strain gradient term
\begin{equation*}
\frac{1}{2}\int_{\Omega} \C e(t,x) : e(t,x) \,\de x + \int_\Omega \B(\alpha(t,x)) p(t,x) \colon p(t,x) \,\de x + l_p \|\nabla p(t)\|_{L^2}^2,
\end{equation*}
for $\B$ a fourth-order tensor positive semidefinite (usually positive definite except for $\alpha=0$). The gradient term, which enforces the plasticity to be in $H^1$ (in space), introduces a characteristic internal length related to the scale of the observed shear bands.  In this setting, the plastic dissipation from a plastic strain $q$ is the functional
$$\int_\Omega H(p(t,x)-q(x))\, \de x,\qquad H(\xi)=\sup_{\sigma\in K} \sigma \colon \xi.$$
We observe that in the original model $p(t)$ is in general not in $L^2$ in the set $\{\alpha=0\}$, so shear bands with infinitesimal thickness may therein display. Therefore, in~\cite{Cr3} the dissipation from a plastic strain $q$ is the relaxation of the functional above for $p(t)$ a  bounded Radon measure. 

In the present setting, the additional regularity of the plastic strain with respect to~\cite{Cr3} makes the existence of energetic solutions much easier to prove; the core of the paper is the study of the regularity in time of the evolutions.
In fact, with general nonconvex total energies, evolutions may display jumps in time. Basing on standard techniques we show in Lemma~\ref{lem:continuity} that energetic solutions, whose existence is proven in  Theorems~\ref{thm:glob_qe}, are strongly continuous outside an at most countable set of time instants.

Further, we develop a more precise analysis exploiting the particular structure of the model: we prove that the variations in time of the whole set of variables with respect to the target norms are controlled in terms of the variation in time of the damage in the $L^1(\Omega)$-norm, plus that of the external loading (cf.\ Theorem~\ref{thm:cont_time} for the precise statement). This ensures that, as soon as damage does not jump, the whole evolution, and in particular the plastic strain, has no time discontinuities.
Our analysis is in the spirit of~\cite{DMDSMo}, which proves that small strain perfect plasticity is absolutely continuous in time.

The interplay between damage and plastic strain is a key point in the analysis of coupled damage-plasticity models. We consider, for instance, the uniaxial responses in~\cite{AMV}, a reference work for the variational analysis of such coupling, and in~\cite{KazMar18}: in~\cite{AMV} it is shown that strain concentrates into  `cohesive crack points' where the space derivative of the damage is discontinuous; in~\cite{KazMar18}, assuming solutions smooth in time and a linear $D$, it is observed that damage evolves only with plasticity.

By inspecting our proof it is apparent that the strain gradient regularization is crucial for the improved regularity. This could suggest, in analogous cases with gradient plasticity, to carefully verify if jumps in time of plasticity but not of damage are analytically ruled out and if this is in accordance with the expected experimental results.  

In the final part of the paper we study the existence of quasistatic evolution obtained through a vanishing viscosity approximation, satisfying the notion of Balanced Viscosity (BV) solutions, see Definition~\ref{def:BV_sol}.

BV solutions, which have been described in an ‘abstract’ context in~\cite{EfMi} and the subsequent~\cite{MieRosSav12, MieRosSav13}, are obtained trough the approximation of a rate-independent system by viscously perturbed system and provide a selection criterion for solutions with a mechanically feasible behavior at jumps, as well as a description of trajectories joining states before and after jumps. In fact,
the notion of energetic solutions is not satisfactory from the physical point of view, whenever jumps develop in time: these jumps may appear ‘too long and too early’, being forced to overtake energy barriers between energy wells (see the characterization in~\cite{RosSav12} and the references therein).

% , motivated by the observation that Energetic solutions jump ‘too long and too early’, cf.\ the characterization proved in [29] and the references therein. 
% We refer to the pioneering , for the definition and properties of such solutions.

The vanishing-viscosity technique has also  been adopted in various concrete applications, ranging from plasticity  (cf., e.g.,~\cite{DMDMM08, DMSo, DMDSSo, BabFraMor12, FrSt2013, Sol14}), to damage,  fracture, and fatigue  (see for instance~\cite{KnMiZa08ILMC, Lazzaroni-Toader, KRZ, CrLa, AlmLazLuc19, ACO2018, CrRo19}).

We show existence of BV solutions (cf.\ Theorem~\ref{thm:BV_evol}), mostly taking advantage of the analysis carried out for the regularity result. Arguing similarly, we show that as soon as the damage variable is regularized in time with respect to the $L^2$-norm (in space), then all the evolution gains the same regularity in time, in the energy space.
Moreover, using an argument from~\cite{KRZ, CrLa}, we get that the evolution is absolutely continuous in time uniformly respect to the viscosity parameter.
Eventually, such regularity permits follow the lines of~\cite{CrLa} to conclude existence of Balanced Viscosity solutions. 

The paper is organised as follows: In Section~\ref{sec:prel} we fix the notation used throughout the paper and we introduced the mathematical framework of our model. In Section~\ref{sec:gl_qe} we first prove in Theorem~\ref{thm:glob_qe} the existence of global quasistatic evolutions, and then in Theorem~\ref{thm:cont_time} we investigate their regularity in time. We continue with Section~\ref{sec:ap_ve}, where we introduced a $\varepsilon$-regularization in time with respect to the $L^2$-norm in space of the damage variable, and in Theorem~\ref{thm:eap_vis_ev} we prove the existence of $\varepsilon$-approximate viscous evolution. Finally, in Section~\ref{sec:res_qvs} we prove the existence of BV solutions (see Theorem~\ref{thm:BV_evol}). 
% We conclude the paper with the Appendix~\ref{sec:app} which contains some technical results used throughout the paper.

\section{Preliminaries}\label{sec:prel}

\subsection{Notation} The space of $n\times n$ symmetric matrices with real entries is denoted by $\M$. Given two matrices $A_1,A_2\in \M$, their scalar product is denoted by $A_1\colon A_2$ and the induced norm of $A\in\M$ by $|A|$. The identity matrix is denoted by $I \in \M$. 

For a measurable set $B \subset \R^n$ we use $\mathcal L^n(B)$ to denote the $n$-dimensional Lebesgue measure of $B$. By $\chi_B\colon \R^n \to \{ 0 ,1 \}$ we denote the corresponding characteristic function.  Given an open subset $\Omega$ of $\R^n$, the set of all distributions on $\Omega$, namely the continuous dual space of $C_c^\infty (\Omega)$, endowed with the strong dual topology, is denoted by $\mathcal D'(\Omega)$. We adopt standard notation for Lebesgue spaces on measurable subsets $B\subset\mathbb R^n$ and Sobolev spaces on open subsets $\Omega\subset\mathbb R^n$. According to the context, we use $\|\cdot\|_{L^p(B)}$ to denote the norm in $L^p(B)$ for all $1\le p\le\infty$. A similar convention is also used to denote the norms in Sobolev spaces. The boundary values of a Sobolev function are always intended in the sense of traces.

The partial derivatives with respect to the variable $x_i$ are denoted by $\partial_i$. Given an open subset $\Omega\subset\mathbb R^n$ and a function $u\colon\Omega\to\R^d$, we denote its Jacobian matrix by $\nabla u$, whose components are $(\nabla u)_{ij}\coloneqq\partial_j u_i$ for $i=1,\dots,d$ and $j=1,\dots,n$. For a function $u\colon \Omega \to \R^n$ we use $\EE u\coloneqq(\nabla u)^{\mathrm{sym}}=\frac{1}{2}(\nabla u+\nabla u^T)\in\M$ to denote the symmetric part of the gradient. Given a tensor field $F\colon \Omega\to\M$, by $\div F$ we mean its divergence with respect to lines, namely $(\div F)_i\coloneqq\sum_{j=1}^n\partial_jF_{ij}$ for $i=1,\dots,n$. 

\subsection{Mathematical framework}
We fix the mathematical framework of the model of complete damage and strain gradient plasticity presented in the introduction. For simplicity, all the physical constants are fixed to 1.

Throughout the paper the {\it reference configuration} $\Omega$ is a bounded domain with Lipschitz boundary in $\R^n$, where $n=2,3$. The {\it displacement field} is represented by a function $u\in H^1(\Omega;\R^n)$, and we decompose its {\it total strain} as $\EE u=e+p$, where $e$ denotes the {\it elastic strain} and $p$ the {\it plastic strain}. To simplify the notation, we set
$$\mathcal A\coloneqq H^1(\Omega;\R^n)\times L^2(\Omega;\M)\times H^1(\Omega;\M).$$ Given a function $w\in H^1(\Omega;\R^n)$, the set of admissible displacements and strains for the boundary datum $w$ on $\partial\Omega$ is defined as
\begin{align*}
\mathcal A(w)\coloneqq \{(u,e,p)\in \mathcal A\,:\,\text{$\EE u=e+p$ in $\Omega$, $u=w$ on $\partial\Omega$}\}.
\end{align*}

\begin{remark}
In this model we require a higher regularity in the plastic strain $p$, which usually is only a Radon measure. This property is ensured by the addition in the total energy~\eqref{eq:tot_en} of the $L^2$ norm of the plastic strain gradient term.
\end{remark}

We fix $T>0$. Let us consider a time-dependent boundary displacement $w$ satisfying
\begin{equation}\label{eq:w}
w\in H^1((0,T);H^1(\Omega;\R^n)).
\end{equation}
The damage variable is a function $\alpha\in H^1(\Omega;[0,1])$. Given $\alpha^0\in H^1(\Omega;[0,1])$ we set
\begin{equation*}
\mathcal D(\alpha^0)\coloneqq \{\alpha\in H^1(\Omega;[0,1])\,:\,\alpha\le \alpha^0\},
\end{equation*}
so that 
\begin{equation*}
\mathcal D(\alpha^2)\subset \mathcal D(\alpha^1)\quad\text{for all $\alpha^2\in\mathcal D(\alpha^1)$}.
\end{equation*}
The {\it energy dissipated in damage growth} is the functional $D\colon H^1(\Omega;[0,1])\to [0,\infty)$ given by
\begin{equation*}
D(\alpha)\coloneqq \int_\Omega d(\alpha(x)) \,\de x\quad\text{for all $\alpha\in H^1(\Omega;[0,1])$},
\end{equation*}
where $d\colon [0,1]\to\R$ satisfies
\begin{align}
&d\in C^{1,1}([0,1];[0,\infty)),\qquad \dot d(0)\le 0\label{eq:d}
\end{align}
(notice that the assumption $\dot d(0)\le 0$ is compatible with the ones proposed in~\cite{KazMar18}).
The {\it elastic energy} is the functional $\mathcal Q\colon L^2(\Omega;\M)\to [0,\infty)$ given by
\begin{equation*}
\mathcal Q(e)\coloneqq \frac{1}{2}\int_\Omega\C e(x):e(x) \,\de x\quad\text{for all $e\in L^2(\Omega;\M)$},
\end{equation*}
where $\C\in \Lin(\M;\M)$ satisfies
\begin{align}
&\C_{ijkl}=\C_{klij}=\C_{jikl}\quad\text{for all $i,j,k,l=1,\dots,n$},\label{eq:C1}\\
&\gamma_1|\xi|^2\le \C\xi:\xi\le\gamma_2|\xi|^2\quad\text{for all $\xi\in\M$},\label{eq:C2}
\end{align}
with constants $\gamma_1,\gamma_2>0$. In particular, it holds
\begin{equation*}
|\C\xi|\le \gamma_2 |\xi|\quad\text{for all $\xi\in\M$}.
\end{equation*}
The {\it kinematical hardening term} is the functional $\tilde Q\colon H^1(\Omega;[0,1])\times H^1(\Omega;\M)\to [0,\infty)$ given by
\begin{equation*}
\tilde Q(\alpha,p)\coloneqq \int_\Omega \B(\alpha(x))p(x):p(x) \,\de x\quad\text{for all $\alpha\in H^1(\Omega;[0,1])$ and $p\in H^1(\Omega;\M)$},
\end{equation*}
where $\B\colon [0,1]\to \Lin(\M;\M)$ satisfies
\begin{align}
&\B\in C^{1,1}([0,1];\Lin(\M;\M)),\qquad \dot{\B}(0)\xi:\xi\le 0\quad\text{for all $\xi\in\M$},\label{eq:B1}\\
&\B_{ijkl}(\alpha)=\B_{klij}(\alpha)=\B_{jikl}(\alpha)\quad\text{for all $i,j,k,l=1,\dots,n$ and $\alpha\in [0,1]$},\label{eq:B2}\\
&\B(\alpha)\xi:\xi\ge 0\quad\text{for all $\alpha\in[0,1]$ and $\xi\in\M$}.\label{eq:B3}
\end{align}

\begin{remark}\label{rem:BCS}
Thanks to~\eqref{eq:B2} and~\eqref{eq:B3}, we derive the Cauchy-Schwartz Inequality \begin{equation*}
|\B(\alpha)\xi_1:\xi_2|^2\le (\B(\alpha)\xi_1:\xi_1)(\B(\alpha)\xi_2:\xi_2)\quad\text{for all $\alpha\in[0,1]$ and $\xi_1,\xi_2\in\M$}.
\end{equation*}
\end{remark}
Define the {\it total energy} $\mathcal E\colon H^1(\Omega;[0,1])\times L^2(\Omega;\M)\times H^1(\Omega;\M)\to [0,\infty)$ as
\begin{equation}\label{eq:tot_en}
\mathcal E(\alpha,e,p)\coloneqq \mathcal Q(e)+D(\alpha)+\|\nabla\alpha\|_2^2+\tilde Q(\alpha,p)+\|\nabla p\|_2^2.
\end{equation}

\begin{remark}\label{rem:emb}
By the Sobolev Embedding Theorem, for $n=2,3$ and $d\in\N$ we have that the embedding $H^1(\Omega;\R^d)\hookrightarrow L^\theta(\Omega;\R^d)$ is continuous for all $\theta\in[1,6]$ and compact for all $\theta\in[1,6)$. In particular, by~\eqref{eq:B1} for all $\alpha\in H^1(\Omega;[0,1])$, $\beta\in H^1(\Omega)$, and $p,q\in H^1(\Omega;\M)$ we have:
\begin{align*}
&\left|\int_\Omega \B(\alpha(x))p(x):q(x) \,\de x\right|\le \|\B\|_\infty\|p\|_2\|q\|_2,\\
&\left|\int_\Omega \dot{\B}(\alpha(x))\beta(x)p(x):q(x) \,\de x\right|\le \|\dot{\B}\|_\infty\|\beta\|_2\|p\|_4\|q\|_4,\\
&\left|\int_\Omega [\B(\alpha_1(x))-\B(\alpha_2(x))]p(x):q(x) \,\de x\right|\le \|\dot{\B}\|_\infty\|\alpha_1-\alpha_2\|_2\|p\|_4\|q\|_4,\\
&\left|\int_\Omega [\dot{\B}(\alpha_1(x))-\dot{\B}(\alpha_2(x))]\beta(x)p(x):q(x) \,\de x\right|\le \|\ddot{\B}\|_\infty\|\alpha_1-\alpha_2\|_4\|\beta\|_4\|p\|_4\|q\|_4.
\end{align*}
\end{remark}

The {\it constraint set} for the stress tensor $\sigma$ is identified by a closed and convex set  $K\subset\M$ satisfying
\begin{equation}\label{eq:K}
\{\sigma\in\M\,:\,|\sigma|\le r_H\}\subset K
\end{equation}
for a radius $r_H> 0$. 

\begin{remark}
Notice that $K$ can be unbounded. In particular, as in~\cite{Cr3}, we can consider constraint set of the form 
$$K\coloneqq\{\sigma\in\M\,:\,\tau\sigma_m+|\sigma_D|-\kappa\ge 0\}$$
for $\tau,\kappa>0$, where $\sigma_m\coloneqq\frac{\tr\sigma}{n}\in\R$ is the {\it mean stress} and $\sigma_D\coloneqq\sigma-\sigma_mI$ is the {\it deviatoric part} of $\sigma$.
\end{remark}

Let us consider the support function $H\colon\M\to [0,\infty]$ of $K$, which is defined as
\begin{equation*}
H(\xi)\coloneqq \sup_{\sigma\in K}\sigma:\xi\quad\text{for all $\xi\in\M$}.
\end{equation*}
Clearly $H$ is convex, lower semicontinuous, and positively 1-homogeneous. Moreover
\begin{equation}\label{eq:rH}
r_H|\xi|\le H(\xi)\quad\text{for all $\xi\in\M$}.
\end{equation}
The {\it plastic potential} $\mathcal H\colon L^1(\Omega;\M)\to [0,\infty]$ is defined as 
\begin{equation*}
\mathcal H(p)\coloneqq \int_\Omega H(p(x)) \,\de x\quad\text{for all $p\in L^1(\Omega;\M)$}.
\end{equation*}

\begin{lemma}\label{lem:lsc}
The functional $\mathcal H$ is lower semicontinuous in $L^1(\Omega;\M)$.
\end{lemma}

\begin{proof}
Let $(p_k)_k\subset L^1(\Omega;\M)$ and $p\in L^1(\Omega;\M)$ be such that $p_k\to p$ in $L^1(\Omega;\M)$. We consider a subsequence $(p_{k_j})_j\subset (p_k)_k$ satisfying
$$\lim_{j\to\infty}\mathcal H(p_{k_j})=\liminf_{k\to\infty}\mathcal H(p_k),\qquad p_{k_j}(x)\to p(x)\quad\text{for a.e.\ $x\in\Omega$}.$$
Since $H\colon\M\to [0,\infty]$ is lower semicontinuous, by Fatou Lemma we have
$$\mathcal H(p)=\int_\Omega H(p(x))\,\de x\le \int_\Omega \liminf_{j\to\infty}H(p_{k_j}(x))\,\de x\le \liminf_{j\to\infty}\int_\Omega H(p_{k_j}(x))\,\de x=\lim_{j\to\infty}\mathcal H(p_{k_j})=\liminf_{k\to\infty}\mathcal H(p_k),$$
which implies the lower semicontinuous of $\mathcal H$ in $L^1(\Omega;\M)$.
\end{proof}

\begin{remark}\label{rem:lsc}
Since $\mathcal H\colon L^1(\Omega;\M)\to [0,\infty]$ is convex, from Lemma~\ref{lem:lsc} we derive that $\mathcal H$ weakly (sequentially) lower semicontinuous in $L^1(\Omega;\M)$. By arguing in the same way, we derive that also the functional $p\mapsto \int_{t_1}^{t_2}\mathcal H(p(t))\,\de t$ is lower semicontinuous in $L^1([t_1,t_2];L^1(\Omega;\M))$ for all $[t_1,t_2]\subset[0,T]$. \end{remark}

Given $[t_1,t_2]\subset[0,T]$ and $p\colon [t_1,t_2]\to L^1(\Omega;\M)$, the {\it $\mathcal H$-dissipation of $p$} on $[t_1,t_2]$ is defined as
$$\mathcal V_{\mathcal H}(p;t_1,t_2)\coloneqq \sup\left\{\sum_{j=1}^N\mathcal H(p(s_j)-p(s_{j-1}))\,:\,N\in\N,\,t_1=s_0<s_1<\dots<s_{N-1}<s_N=t_2\right\}.$$
We recall that this notion has been introduced in~\cite[Appendix]{DMDSMo}. By~\eqref{eq:rH} we have
\begin{equation}\label{eq:rHvar}
r_H\mathcal V(p;t_1,t_2)\le \mathcal V_{\mathcal H}(p;t_1,t_2),
\end{equation}
where $\mathcal V(p;t_1,t_2)$ is the total variation of $p\colon [t_1,t_2]\to L^1(\Omega;\M)$, which is defined as
$$\mathcal V(p;t_1,t_2)\coloneqq \sup\left\{\sum_{j=1}^N\mathcal \|p(s_j)-p(s_{j-1})\|_1\,:\,N\in\N,\,t_1=s_0<s_1<\dots<s_{N-1}<s_N=t_2\right\}.$$

\begin{lemma}
Let $[t_1,t_1]\subset[0,T]$ and $p_k\colon [t_1,t_2]\to L^1(\Omega;\M)$, $k\in\N$, satisfy
$$p_k(r)\rightharpoonup p(r)\quad\text{in $L^1(\Omega;\M)$ as $k\to\infty$ for all $r\in[t_1,t_2]$}.$$
Then
\begin{equation}\label{eq:Hdis-lsc}
\mathcal V_{\mathcal H}(p;t_1,t_2)\le\liminf_{k\to\infty}\mathcal V_{\mathcal H}(p_k;t_1,t_2).    
\end{equation}
\end{lemma}

\begin{proof}
We fix $n\in\N$ and a collection of points $(s_j)_{j=0}^N$ such that
\begin{equation}\label{eq:partition}
t_1=s_0<s_1<\dots<s_{N-1}<s_N=t_2.   
\end{equation}
For all $j=1,\dots,n$ we have 
$$p_k(s_j)-p_k(s_{j-1})\rightharpoonup p(s_j)-p(s_{j-1})\quad\text{in $L^1(\Omega;\M)$ as $k\to\infty$}.$$
Therefore, by Remark~\ref{rem:lsc} \begin{align*}
\sum_{j=1}^N\mathcal H(p(s_j)-p(s_{j-1}))&\le \sum_{j=1}^N\liminf_{k\to\infty}\mathcal H(p_k(s_j)-p_k(s_{j-1}))\\
&\le \liminf_{k\to\infty}\sum_{j=1}^N\mathcal H(p_k(s_j)-p_k(s_{j-1}))\le \liminf_{k\to\infty}\mathcal V_{\mathcal H}(p_k;t_1,t_2).
\end{align*}
By taking the supremum over $N\in\N$ and all collections of points $(s_j)_{j=0}^N$ satisfying~\eqref{eq:partition} we get~\eqref{eq:Hdis-lsc}.
\end{proof}

We recall the following result.

\begin{lemma}\label{lem:Cdelta}
For all $\theta\in[2,6)$ and $\delta>0$ there exists a constant $C_{\theta,\delta}>0$ such that
\begin{equation*}
\|u\|_\theta^2\le C_{\theta,\delta}\|u\|_1^2+\delta\|\nabla u\|_2^2\quad\text{for all $u\in H^1(\Omega;\R^d)$}.
\end{equation*}
\end{lemma}

\begin{proof}
We fix $\theta\in[2,6)$ and $\delta>0$ and we argue by contradiction. We assume that for all $k\in \N$ there exists $u_k\in H^1(\Omega;\R^d)$ satisfying
\begin{equation*}
\|u_k\|_\theta^2> k\|u_k\|_1^2+\delta \|\nabla u_k\|_2^2.
\end{equation*}
We set $\hat u_k\coloneqq \frac{u_k}{\|u_k\|_\theta}$ and we have
\begin{equation*}
\|\hat u_k\|_\theta=1,\quad 1>k\|\hat u_k\|_1^2+\delta\|\hat \nabla u_k\|_2^2\quad\text{for all $k\in\N$}.
\end{equation*}
In particular, by the continuity of the embedding $L^\theta(\Omega;\R^d)\hookrightarrow L^2(\Omega;\R^d)$ we deduce the existence of a constant $C_\theta>0$ such that
$$\|\hat u_k\|_{H^1}^2=\|\hat u_k\|_2^2+\|\nabla \hat u_k\|_2^2\le C_\theta+\frac{1}{\delta}\quad\text{for all $k\in\N$}.$$
By the Sobolev Embedding Theorem there exists a function $\hat u\in H^1(\Omega;\R^d)$ with $\|\hat u\|_\theta=1$ and a subsequence $(\hat u_{k_j})_j\subset (\hat u_k)_k$ such that 
\begin{equation*}
\hat u_{k_j}\to \hat u\quad\text{in $L^\theta(\Omega;\R^d)$ as $j\to\infty$}.
\end{equation*}
On the other hand,
\begin{equation*}
    \|\hat u_k\|_1^2\le \frac{1}{k}\quad\text{for all $k\in\N$},
\end{equation*}
which gives
\begin{equation*}
\hat u_k\to 0\quad\text{in $L^1(\Omega;\R^d)$ as $k\to\infty$}.
\end{equation*}
This implies that $\hat u=0$, which contradicts $\|\hat u\|_\theta=1$.
\end{proof}

The following is an easy consequence of Lemma~\ref{lem:Cdelta}.

\begin{corollary}\label{coro:H1norm}
There exists a constant $C>0$ such that
\begin{equation}\label{eq:H1norm}
\|u\|_{H^1}^2\le C\left(\|u\|_1^2+\|\nabla u\|_2^2\right)\quad\text{for all $u\in H^1(\Omega;\R^d)$}.
\end{equation}
\end{corollary}

\begin{proof}
By Lemma~\ref{lem:Cdelta} with $\theta=2$ and $\delta=1$ there exists a constant $C_1>0$ such that 
$$\|u\|_{H^1}^2=\|u\|_2^2+\|\nabla u\|_2^2\le C_1\|u\|_1^2+2\|\nabla u\|_2^2\le \max\{2,C_1\}\left(\|u\|_1^2+\|\nabla u\|_2^2\right),$$
which gives~\eqref{eq:H1norm}
\end{proof}

%--------------------------------
% Global quasistatic evolutions
%--------------------------------

\section{Global quasistatic evolutions}\label{sec:gl_qe}

In this section we prove Theorem~\ref{thm:glob_qe}, i.e., the existence of global quasistatic evolutions (also referred to as energetic solutions) for the model of geomaterials coupling damage and small-strain associative plasticity introduced before. Moreover, in Theorem~\ref{thm:cont_time} we perform a detailed analysis of the regularity in time of these evolutions.

\subsection{Definition and preliminary results}

We start by introducing the mathematical formulation of our model of geomaterials. Following~\cite{Cr3} and the definition of energetic solution of~\cite{MieThe99MMRI}, we give the following notion of global quasistatic evolutions.

\begin{definition}\label{def:QSE}
Assume~\eqref{eq:w}--\eqref{eq:B3} and~\eqref{eq:K}. A quadruple $(\alpha,u,e,p)$ from $[0,T]$ into $H^1(\Omega;[0,1])\times\mathcal A$ is a {\it global quasistatic evolution} with datum $w$ if:
\begin{itemize}
    \item[(qs0)] for a.e.\ $x\in\Omega$ the map
    \begin{equation*}
t\mapsto \alpha(t,x)\quad\text{is non increasing on $[0,T]$}; 
\end{equation*}
    \item[(qs1)] $(u(t),e(t),p(t))\in\mathcal A(w(t))$ for all $t\in[0,T]$ and
    \begin{align*}
     &\mathcal E(\alpha(t),e(t),p(t))\le \mathcal E(\beta,\eta,q)+\mathcal H(q-p(t))\quad\text{for all $(\beta,v,\eta,q)\in \mathcal D(\alpha(t))\times \mathcal A(w(t))$;}
    \end{align*}
    \item[(qs2)] $p\colon [0,T]\to H^1(\Omega;\M)$ has bounded $\mathcal H$-variation and for all $t\in[0,T]$
  \begin{align*}
      &\mathcal E(\alpha(t),e(t),p(t))+\mathcal V_{\mathcal H}(p;0,t)= \mathcal E(\alpha(0),e(0),p(0))+\int_0^t(\C e(s),\EE \dot w(s))\,\de s.
  \end{align*}
\end{itemize}
\end{definition}

\begin{remark}
The first condition (qs0) models the {\it irreversibility} of the damage process, (qs1) is the {\it global stability condition}, and (qs2) is the {\it energy-dissipation balance}.
\end{remark}

Let us check that the integral in (qs2) is well defined. To this aim, we first prove the following stability result.

\begin{lemma}\label{lem:stab}
Assume~\eqref{eq:d}--\eqref{eq:B3}. Let $w_k\in H^1(\Omega;\R^n)$, and $(\alpha_k,u_k,e_k,p_k)\in H^1(\Omega;[0,1])\times\mathcal A(w_k)$ for all $k\in\N$. Assume that as $k\to\infty$
\begin{align*}
   &u_k\rightharpoonup u_\infty\quad\text{in $H^1(\Omega;\R^n)$}, & &\alpha_k\rightharpoonup \alpha_\infty\quad\text{in $H^1(\Omega)$},\\
   &e_k\rightharpoonup e_\infty\quad\text{in $L^2(\Omega;\M)$}, & &p_k\rightharpoonup p_\infty\quad\text{in $H^1(\Omega;\M)$},\\
   &w_k\rightharpoonup w_\infty\quad\text{in $H^1(\Omega;\R^n)$}. & &
\end{align*}
Then $(\alpha_\infty,e_\infty,p_\infty)\in\mathcal A(w_\infty)$. If in addition
\begin{equation}\label{eq:gl_kst}
\mathcal E(\alpha_k,e_k,p_k)\le \mathcal E(\beta,\eta,q)+\mathcal H(q-p_k) \quad\text{for all $(\beta,v,\eta,q)\in\mathcal D(\alpha_k)\times \mathcal A(w_k)$},
\end{equation}
then
\begin{equation}\label{eq:gl_infst}
\mathcal E(\alpha_\infty,e_\infty,p_\infty)\le \mathcal E(\beta,\eta,q)+\mathcal H(q-p_\infty) \quad\text{for all $(\beta,v,\eta,q)\in\mathcal D(\alpha_\infty)\times \mathcal A(w_\infty)$}.
\end{equation}
\end{lemma}

\begin{proof}
We argue as in~\cite[Theorem 3.6]{Cr1}. Since
$$\EE u_k=e_k+p_k\quad\text{in $\Omega$ for all $k\in\N$},\qquad u_k=w_k\quad\text{on $\partial\Omega$ for all $k\in\N$},$$
by passing to the limit as $k\to\infty$ and using the continuity of the trace operator, we derive that 
$$\EE u_\infty=e_\infty+p_\infty\quad\text{in $\Omega$},\qquad u_\infty=w_\infty\quad\text{on $\partial\Omega$}.$$
Therefore, $(u_\infty,e_\infty,p_\infty)\in\mathcal A(w_\infty)$. 

We fix $(\beta,v,\eta,q)\in\mathcal D(\alpha_\infty)\times \mathcal A(w_\infty)$ and for all $k\in\N$ we set
\begin{align*}
\beta_k\coloneqq \beta\wedge\alpha_k,\qquad v_k\coloneqq v-u_\infty+u_k,\qquad\eta_k\coloneqq \eta-e_\infty+e_k,\qquad q_k\coloneqq q-p_\infty+p_k.
\end{align*}
By construction $(\beta_k,v_k,\eta_k,q_k)\in\mathcal D(\alpha_k)\times \mathcal A(w_k)$ for all $k\in\mathbb N$, and as $k\to\infty$
\begin{align}
   &v_k\rightharpoonup v\quad\text{in $H^1(\Omega;\R^n)$}, & &\beta_k\rightharpoonup \beta \quad\text{in $H^1(\Omega)$},\label{eq:vqk1}\\
   &\eta_k\rightharpoonup \eta\quad\text{in $L^2(\Omega;\M)$}, & &q_k\rightharpoonup q\quad\text{in $H^1(\Omega;\M)$}.\label{eq:vqk2}
\end{align}
We use $(\beta_k,v_k,\eta_k,q_k)$ as test function in~\eqref{eq:gl_kst}, and from the identity
$$\|\nabla \beta\|_2^2+\|\nabla \alpha_k\|_2^2=\|\nabla (\beta\wedge\alpha_k)\|_2^2+\|\nabla (\beta\vee\alpha_k)\|_2^2,$$
we derive
\begin{align}
&D(\alpha_k)+\|\nabla (\beta\vee \alpha_k)\|_2^2-\|\nabla \beta\|_2^2+\tilde Q(\alpha_k,p_k)\nonumber\\
&=D(\alpha_k)+\|\nabla \alpha_k\|_2^2-\|\nabla \beta_k\|_2^2+\tilde Q(\alpha_k,p_k)\nonumber\\
&\le \mathcal Q(\eta_k)-\mathcal Q(e_k)+D(\beta_k)+\|\nabla q_k\|_2^2-\|\nabla p_k\|_2^2+\tilde Q(\beta_k,q_k)+\mathcal H(q_k-p_k)\nonumber\\
&=\frac{1}{2}(\C(\eta-e_\infty+2e_k),\eta-e_\infty)_2+D(\beta_k)+\tilde Q(\beta_k,q_k)\nonumber\\
&\quad+(\nabla q-\nabla p_\infty+2\nabla p_k,\nabla q-\nabla p_\infty)_2+\mathcal H(q-p_\infty).\label{eq:minvqk}
\end{align}
By~\eqref{eq:vqk1}--\eqref{eq:vqk2} and Remark~\ref{rem:emb}, for all $\theta\in[1,6)$ we have as $k\to\infty$
\begin{align*}
&\alpha_k\to\alpha_\infty\quad\text{in $L^\theta(\Omega)$},& &\beta_k\to \beta \quad\text{in $L^\theta(\Omega)$},\\
&p_k\to p_\infty\quad\text{in $L^\theta(\Omega;\M)$},& &q_k\to q \quad\text{in $L^\theta(\Omega;\M)$}.
\end{align*}
In particular, by using~\eqref{eq:d} and~\eqref{eq:B1} and arguing as in Remark~\ref{rem:emb} we deduce
\begin{align*}
&\lim_{k\to\infty}D(\alpha_k)=D(\alpha_\infty),& & \lim_{k\to\infty} D(\beta_k)=D(\beta),\\
&\lim_{k\to\infty}\tilde Q(\alpha_k,p_k)=\tilde Q(\alpha_\infty,p_\infty),& & \lim_{k\to\infty}\tilde Q(\beta_k,q_k)=\tilde Q(\beta,q).
\end{align*}
Moreover, $\beta\vee \alpha_k\rightharpoonup \alpha_\infty$ in $H^1(\Omega)$ as $k\to\infty$, which implies
$$\|\nabla\alpha_\infty\|_2^2\le \liminf_{k\to\infty}\|\nabla (\beta\vee\alpha_k)\|_2^2.$$
Hence, by passing to the limit as $k\to\infty$ in~\eqref{eq:minvqk} we get
\begin{align*}
&D(\alpha_\infty)+\|\nabla \alpha_\infty\|_2^2-\|\nabla \beta\|_2^2+\tilde Q(\alpha_\infty,p_\infty)\\
&\le\frac{1}{2}(\C(\eta+e_\infty),\eta-e_\infty)_2+D(\beta)+\tilde Q(\beta,q)+(\nabla q+\nabla p_\infty,\nabla q-\nabla p_\infty)_2+\mathcal H(q-p_\infty)\\
&=\mathcal Q(\eta)-\mathcal Q(e_\infty)+D(\beta)+\tilde Q(\beta,q)+\|\nabla q\|_2^2-\|\nabla p_\infty\|_2^2+\mathcal H(q-p_\infty),
\end{align*}
which gives~\eqref{eq:gl_infst}.
\end{proof}

Thanks to Lemma~\ref{lem:stab} we deduce the following regularity result.

\begin{lemma}\label{lem:measurability}
Assume~\eqref{eq:w}--\eqref{eq:B3} and~\eqref{eq:K}. Let $(\alpha,u,e,p)$ from $[0,T]$ into $H^1(\Omega;[0,1])\times\mathcal A$ be a quadruple satisfying {\em (qs0)} and {\em (qs1)}, and such that $p\colon [0,T]\to H^1(\Omega;\M)$ has bounded $\mathcal H$-variation. Then there exists a constant $C>0$, independent on $t\in[0,T]$, such that
\begin{equation}\label{eq:unif_bound}
\|\alpha(t)\|_{H^1}+\|u(t)\|_{H^1}+\|e(t)\|_2+\|p(t)\|_{H^1}\le C\quad\text{for all $t\in[0,T]$}.
\end{equation}
Moreover, there exists a countable set $N\subset [0,T]$ such that the quadruple $(\alpha,u,e,p)$ is weakly continuous from $[0,T]\setminus N$ into $H^1(\Omega)\times \mathcal A$. In particular, the quadruple $(\alpha,u,e,p)$ is strongly measurable from $[0,T]$ into $H^1(\Omega)\times \mathcal A$.
\end{lemma}

\begin{proof}
Since $p\colon [0,T]\to L^1(\Omega;\M)$ has bounded $\mathcal H$-variation, by~\eqref{eq:rHvar} we deduce that $p$ has bounded total variation. In particular, $\|p(t)\|_1$ is uniformly bounded in $[0,T]$, being
\begin{equation}\label{eq:pbound}
\|p(t)\|_1\le \|p(t)-p(0)\|_1+\|p(0)\|_1\le \mathcal V(p;0,T)+\|p(0)\|_1\quad\text{for all $t\in[0,T]$}.    
\end{equation}
Since $(0,u(T)-w(T)+w(t),e(T)-\EE w(T)+\EE w(t),p(T))\in\mathcal D(\alpha(t))\times\mathcal A(w(t))$, by the global stability condition (qs1) we deduce
\begin{align*}
&\mathcal E(\alpha(t),e(t),p(t))\\
&\le \mathcal Q(e(T))+Q(\EE w(T)-\EE w(t))+(\C e(T),\EE w(T)-\EE w(t))_2+\tilde Q(0,p(T))+\|\nabla p(T)\|_2^2+\mathcal H(p(T)-p(t))\\
&\le \frac{\gamma_2}{2}\|e(T)\|_2^2+ \frac{\gamma_2}{2}\|\EE w(T)-\EE w(t)\|_2^2+ \gamma_2\|e(T)\|_2\|\EE w(T)-\EE w(t)\|_2\\
&\quad+\|\B\|_\infty\|p(T)\|_2^2+\|\nabla p(T)\|_2^2+\mathcal V_{\mathcal H}(p,t,T)\\
&\le\gamma_2\|e(T)\|_2^2+\gamma_2\left(\int_0^T\|\EE \dot w(r)\|_2\,\de r\right)^2+\|\B\|_\infty\|p(T)\|_2^2+ \|\nabla p(T)\|_2^2+\mathcal V_{\mathcal H}(p;0,T).
\end{align*}
Notice that the quantity on the right-hand side is independent on $t\in[0,T]$. Since
$$\mathcal E(\alpha(t),e(t),p(t))\ge \frac{\gamma_1}{2}\|e(t)\|_2^2+\|\nabla\alpha(t)\|_2^2+\|\nabla p(t)\|_2^2,$$
we deduce that there exists a constant $C_1>0$, independent on $t\in[0,T]$, such that
$$\|\nabla \alpha(t)\|_2^2+\|e(t)\|_2^2+\|\nabla p(t)\|_2^2\le C_1\quad\text{for all $t\in[0,T]$}.$$
By using~\eqref{eq:pbound}, Corollary~\ref{coro:H1norm}, and the fact that $0\le \alpha(t)\le 1$ a.e.\ in $\Omega$ for all $t\in[0,T]$ we obtain a constant $C_2>0$, independent on $t\in[0,T]$, such that
\begin{equation}\label{eq:aepbound}
\|\alpha(t)\|_{H^1}^2+\|e(t)\|_2^2+\|p(t)\|_{H^1}^2\le C_2\quad\text{for all $t\in[0,T]$}.   
\end{equation}
Finally, since $u(t)-w(t)\in H^1_0(\Omega;\R^n)$, by Korn Inequality there exists a constant $C_3>0$, independent on $t\in[0,T]$, such that 
\begin{align*}
\|u(t)\|_2&\le \|u(t)-w(t)\|_2+\|w(t)\|_2\\
&\le C\|\EE u(t)-\EE w(t)\|_2+\|w(t)\|_2\le C\|e(t)\|_2+C\|p(t)\|_2+C\|\EE w(t)\|_2+\|w(t)\|_2.    
\end{align*}
Therefore, by~\eqref{eq:aepbound} and Korn Inequality there exists a constant $C_5>0$, independent on $t\in[0,T]$, such that 
$$\|u(t)\|_{H^1}\le C.$$
This implies~\eqref{eq:unif_bound}. 

It remains to prove that there exists a countable set $N\subset [0,T]$ such that the quadruple $(\alpha,u,e,p)$ is weakly continuous from $[0,T]\setminus N$ into $H^1(\Omega)\times \mathcal A$. Indeed, if this holds, then the quadruple $(\alpha,u,e,p)$ is weakly measurable from $[0,T]$ into $H^1(\Omega)\times \mathcal A$, and therefore strongly measurable, since $H^1(\Omega)\times \mathcal A$ is a separable Banach space.

We start by observing that $\alpha\colon [0,T]\to H^1(\Omega;[0,1])$ satisfies 
$$\|\alpha(t)\|_\infty\le 1\quad\text{for all $t\in[0,T]$},\qquad \alpha(t_2)\le \alpha(t_1)\quad\text{a.e.\ in $\Omega$ for all $0\le t_1\le t_2\le T$}.$$
By~\cite[Lemma A.2]{Cr1} there exists a countable set $N_1\subset [0,T]$ such that the function $\alpha\colon [0,T]\setminus N_1\to L^2(\Omega)$ is (strongly) continuous. In particular, thanks to the uniform estimate~\eqref{eq:unif_bound} we derive that $\alpha\colon [0,T]\setminus N_1\to H^1(\Omega)$ is weakly continuous. Similarly, since $p$ has bounded total variation in $L^1(\Omega;\M)$, there exists a countable set $N_2\subset [0,T]$ such that the function $p\colon [0,T]\setminus N_2\to L^1(\Omega;\M)$ is strongly continuous.
% ($p$ is continuous on every continuity point of the map $t\mapsto\mathcal V(p;0,t)$ from $[0,T]$ into $\R$, which is an increasing function on $[0,T]$). 
Hence, the uniform estimate~\eqref{eq:unif_bound} yields that $p\colon [0,T]\setminus N_2\to H^1(\Omega;\M)$ is weakly continuous. 

Let us define $N\coloneqq N_1\cup N_2$. We claim that the pair $(u,e)$ is weakly continuous from $[0,T]\setminus N$ into $H^1(\Omega;\R^n)\times L^2(\Omega;\M)$. Indeed, let us fix $t_\infty\in [0,T]\setminus N$ and let us consider a sequence $(t_k)_k\subset [0,T]\setminus N$ such that $t_k\to t_\infty$ as $k\to\infty$. By the uniform estimate~\eqref{eq:unif_bound} there exists a constant $C>0$ independent on $k\in\N$ such that 
$$\|u(t_k)\|_{H^1}+\|e(t_k)\|_2\le C\quad\text{for all $k\in\N$}.$$
Therefore, there exists a subsequence $(t_{k_j})_j\subset(t_k)_k$ and a pair $(u_\infty,e_\infty)\in H^1(\Omega;\R^n)\times L^2(\Omega;\M)$ such that as $j\to\infty$
\begin{align*}
u(t_{k_j})\rightharpoonup u_\infty\quad\text{in $H^1(\Omega;\R^n)$},\qquad e(t_{k_j})\rightharpoonup e_\infty\quad\text{in $L^2(\Omega;\M)$}.
\end{align*}
Moreover, as $j\to\infty$ we have
\begin{align*}
\alpha(t_{k_j})\rightharpoonup \alpha(t_\infty)\quad\text{in $H^1(\Omega)$},\qquad p(t_{k_j})\rightharpoonup p(t_\infty)\quad\text{in $H^1(\Omega;\M)$},\qquad
w(t_{k_j})\to w(t_\infty)\quad\text{in $H^1(\Omega;\R^n)$}.
\end{align*}
Thus, we can apply Lemma~\ref{lem:stab} to deduce that $(\alpha(t_\infty),u_\infty,e_\infty,p(t_\infty))\in\mathcal A(w(t_\infty))$ and satisfies 
\begin{align*}
&\mathcal E(\alpha(t_\infty),e_\infty,p(t_\infty))\le \mathcal E(\beta,\eta,q)+\mathcal H(q-p(t_\infty))\quad\text{for all $(\beta,v,\eta,q)\in \mathcal D(\alpha(t_\infty))\times \mathcal A(w(t_\infty))$.}
\end{align*}
In particular, by choosing $\beta=\alpha(t_\infty)$ and $q=p(t_\infty)$, we derive that the pair $(u_\infty,e_\infty)$ minimizes the functional $F\colon H^1(\Omega;\R^n)\times L^2(\Omega;\M)\to [0,\infty)$, defined as 
$$F(v,\eta)\coloneqq \mathcal Q(\eta)\quad\text{for all $(v,\eta)\in H^1(\Omega;\R^n)\times L^2(\Omega;\M)$},$$
over the convex set 
$$G\coloneqq \{(v,\eta)\in H^1(\Omega;\R^n)\times L^2(\Omega;\M)\,:\, (v,\eta,p(t_\infty))\in\mathcal A(w(t_\infty))\}.$$
This implies that $(u_\infty,e_\infty)$ is uniquely determined for all $t\in[0,T]$ by the strict convexity of $\mathcal Q$. Therefore, $u(t_\infty)=u_\infty$ and $e(t_\infty)=e_\infty$. In particular, we get that as $k\to\infty$
\begin{align*}
u(t_k)\rightharpoonup u(t_\infty)\quad\text{in $H^1(\Omega;\R^n)$},\qquad e(t_k)\rightharpoonup e(t_\infty)\quad\text{in $L^2(\Omega;\M)$},
\end{align*}
i.e., the pair $(u,e)$ is weakly continuous from $[0,T]\setminus N$ into $H^1(\Omega;\R^n)\times L^2(\Omega;\M)$. 
\end{proof}

\begin{remark}
By Lemma~\ref{lem:measurability} the integral in (qs2) is well defined for all quadruple $(\alpha,u,e,p)$ from $[0,T]$ into $H^1(\Omega;[0,1])\times\mathcal A$ satisfying (qs0) and (qs1) and such that $p\colon [0,T]\to H^1(\Omega;\M)$ has bounded $\mathcal H$-variation. 
\end{remark}

\begin{remark}\label{rem:unif_bound}
If $(\alpha,u,e,p)$ from $[0,T]$ into $H^1(\Omega;[0,1])\times \mathcal A$ is a global quasistatic evolution, then we can improve the estimate in~\eqref{eq:unif_bound}. Indeed, by Lemma~\ref{lem:measurability} we have that the function $t\mapsto \|e(t)\|_2$ is bounded and measurable on $[0,T]$. In particular, if we define 
$$M\coloneqq\sup_{t\in[0,T]}\|e(t)\|_2<\infty,$$
by (qs2) we derive
$$\frac{\gamma_1}{2}M^2\le \mathcal E(\alpha(0),e(0),p(0))+\gamma_2 M\int_0^T\|\EE \dot w(s)\|_2\,\de s.$$
This implies the existence of a constant $C_1$, which depends only on $w$, $d$, $\C$, $\B$, and on the initial datum $(\alpha(0),e(0),p(0))$, satisfying
$$M=\sup_{t\in[0,T]}\|e(t)\|_2\le C_1.$$
Therefore, by using again (qs2) we derive the existence of a further constant $C_2$, depending only on $w$, $d$, $\C$, $\B$, $K$, and the initial datum $(\alpha(0),u(0),e(0),p(0))$, such that
\begin{equation*}
\|\alpha(t)\|_{H^1}+\|u(t)\|_{H^1}+\|e(t)\|_2+\|p(t)\|_{H^1}\le C_2\quad\text{for all $t\in[0,T]$}.
\end{equation*}     
\end{remark}

\subsection{Existence of a global quasistatic evoltion}

In this subsection we prove the existence of a global quasistatic evolution satisfying a fixed initial configuration. As in~\cite{Cr3}, the proof relies on the De Giorgi’s Minimizing Movement approach to quasistatic evolutions, where time-continuous evolutions are approximated by discrete-time ones, constructed by solving incremental minimization problems.

We consider an initial configuration $$(\alpha^0,u^0,e^0,p^0)\in H^1(\Omega;[0,1])\times\mathcal A(w(0)),$$ 
which satisfies
\begin{align}\label{eq:ic}
\mathcal E(\alpha^0,e^0,p^0)\le \mathcal E(\beta,\eta,q)+\mathcal H(q-p^0)\quad\text{for all $(\beta,v,\eta,q)\in \mathcal D(\alpha^0)\times \mathcal A(w(0))$}.
\end{align}
For all $k\in\N$ we define
\begin{equation*}
\tau_k\coloneqq \frac{T}{k},\quad t_k^i\coloneqq i\tau_k,\quad w_k^i\coloneqq w(t_k^i)\quad\text{for $i=0,\dots,k$}.
\end{equation*}
Starting from
\begin{equation*}
(\alpha_k^0,u_k^0,e_k^0,p_k^0)\coloneqq (\alpha^0,u^0,e^0,p^0)\in H^1(\Omega;[0,1])\times \mathcal A(w(0)),
\end{equation*}
for all $i=1,\dots,k$ we define
\begin{equation*}
(\alpha_k^i,u_k^i,e_k^i,p_k^i)\in \mathcal D(\alpha_k^{i-1})\times \mathcal A(w_k^i)
\end{equation*}
as the solution of
\begin{equation}\label{eq:gl_minki}
\min_{(\beta,v,\eta,q)\in \mathcal D(\alpha_k^{i-1})\times \mathcal A(w_k^i)}\left\{\mathcal E(\beta,\eta,q)+\mathcal H(q-p_k^{i-1})\right\}.
\end{equation}
Notice that 
\begin{equation*}
(\alpha_k^{i-1},w_k^i,\EE w_k^i,0)\in \mathcal D(\alpha_k^{i-1})\times \mathcal A(w_k^i)\neq\emptyset\quad\text{for all $i=1,\dots,k$}.
\end{equation*}
Moreover, for all $i=1,\dots,k$ and $(\alpha,u,e,p)\in\mathcal D(\alpha_k^{i-1})\times \mathcal A(w_k^i)$ we have
\begin{align*}
\mathcal E(\alpha,e,p)+\mathcal H(p-p_k^i)\ge \frac{\gamma_1}{2}\|e\|_2^2+\|\nabla\alpha\|_2^2+\|\nabla p\|_2^2+r_H\|p\|_1-r_H\|p_k^i\|_1
\end{align*}
and the functional
\begin{equation*}
(\alpha,u,e,p)\mapsto \mathcal E(\alpha,e,p)+\mathcal H(p-p_k^i)
\end{equation*}
is sequentially lower semicontinuous with respect to the weak convergences in the spaces $H^1(\Omega)$, $H^1(\Omega;\R^n)$, $L^2(\Omega;\M)$, and $H^1(\Omega;\M)$. Hence, by the Direct Method of Calculus of Variations, for all $k\in\N$ and $i=1\dots,k$ there exists a solution $(\alpha_k^i,u_k^i,e_k^i,p_k^i)\in\mathcal D(\alpha_k^i)\times\mathcal A(w_k^i)$ to the minimum problem~\eqref{eq:gl_minki}. In particular, since 
$$\mathcal H(q-p_k^{i-1})-\mathcal H(p_k^i-p_k^{i-1})\le \mathcal H(q-p_k^i),$$
the quadruple $(\alpha_k^i,u_k^i,e_k^i,p_k^i)$ satisfies
\begin{align}\label{eq:gl_stki}
\mathcal E(\alpha_k^i,e_k^i,p_k^i)\le \mathcal E(\beta,\eta,q)+\mathcal H(q-p_k^i)\quad\text{for all $(\beta,v,\eta,q)\in \mathcal D(\alpha_k^{i-1})\times \mathcal A(w_k^i)$.}
\end{align}

In the following lemma, we  derive an energy estimates for the discrete evolution $\{(\alpha_k^i,u_k^i,e_k^i,p_k^i)\}_{i=1}^k$. First, for all $k\in\N$ we define
\begin{equation*}
\dot w_k^i\coloneqq \frac{w_k^i-w_k^{i-1}}{\tau_k}\quad\text{for $i=1,\dots,k$}.
\end{equation*}

\begin{lemma}\label{lem:glslkey}
Assume~\eqref{eq:d}--\eqref{eq:B3},~\eqref{eq:K}, and~\eqref{eq:ic}. For all $k\in \N$ and $i=1,\dots,k$ we have
\begin{equation}\label{eq:gl_denin}
\mathcal E(\alpha_k^i,e_k^i,p_k^i)+\sum_{j=1}^i\mathcal H(p_k^j-p_k^{j-1})\le \mathcal E(\alpha^0,e^0,p^0)+\sum_{j=1}^i\tau_k(\C e_k^{j-1},\EE \dot w_k^j)_2+\delta_k,
\end{equation}
where
\begin{equation}\label{eq:deltak}
\delta_k\coloneqq \tau_k\frac{\gamma_2}{2}\int_0^T\|\EE \dot w(r)\|_2^2\,\de r\to 0\quad\text{as $k\to\infty$}.
\end{equation}
In particular, there exists a constant $C>0$ independent on $k$ and $i$, such that
\begin{align*}
&\max_{i=0,\dots,k}\|e_k^i\|_2+\max_{i=0,\dots,k}\|p_k^i\|_{H^1}+\max_{i=0,\dots,k}\|\alpha_k^i\|_{H^1}+\sum_{i=1}^k\mathcal H(p_k^j-p_k^{j-1})\le C.
\end{align*}
\end{lemma}

\begin{proof}
We fix $i=1,\dots,k$ and for $j=1,\dots,i$ we consider
$$(\alpha_k^{j-1},u_k^{j-1}-w_k^{j-1}+w_k^j,e_k^{j-1}-\EE w_k^{j-1}+\EE w_k^j,p_k^{j-1})\in \mathcal D(\alpha_k^{j-1})\times \mathcal A(w_k^j).$$
By the minimality condition~\eqref{eq:gl_minki}
we get
\begin{align*}
\mathcal E(\alpha_k^j,e_k^j,p_k^j)+\mathcal H(p_k^j-p_k^{j-1}) &\le \mathcal E(\alpha_k^{j-1},e_k^{j-1}-\EE w_k^{j-1}+\EE w_k^j,p_k^{j-1})\\
&\le\mathcal E(\alpha_k^{j-1},e_k^{j-1},p_k^{j-1})+\tau_k(\C e_k^{j-1},\EE \dot w_k^j)_2+\frac{\gamma_2}{2}\|\EE w_k^j-\EE w_k^{j-1}\|_2^2.
\end{align*}
We sum over $j=1,\dots,i$ and we use the inequality
\begin{equation*}
\sum_{j=1}^k\|\EE w_k^j-\EE w_k^{j-1}\|_2^2\le\tau_k\int_0^T\|\EE \dot w(r)\|_2^2\,\de r
\end{equation*}
to obtain~\eqref{eq:gl_denin}.

By the discrete energy inequality~\eqref{eq:gl_denin} we get
$$\frac{\gamma_1}{2}\|e_k^i\|_2^2\le \mathcal E(\alpha^0,e^0,p^0)+\gamma_2\sum_{j=1}^i\tau_k\|\EE \dot w_k^j\|_2\|e_k^{j-1}\|_2+\delta_k.$$
Since $\delta_k\to 0$ as $k\to\infty$, there exists a constant $C_1>0$, independent on $k$ and $i$, such that
$$\max_{i=0,\dots,k}\|e_k^i\|_2\le C_1.$$
Therefore, by~\eqref{eq:gl_denin} we can  find another constant $C_2>0$, independent on $k$ and $i$, such that
\begin{align*}
&\max_{i=0,\dots,k}\|\nabla\alpha_k^i\|_2^2+ \max_{i=0,\dots,k}\|\nabla p_k^i\|_2^2+\sum_{i=1}^k\mathcal H(p_k^j-p_k^{j-1})\le C_2.
\end{align*}
By Corollary~\ref{coro:H1norm} and the inequality 
\begin{align*}
&\|p_k^i\|_1\le \|p^0\|_1+\sum_{j=1}^k\|p_k^j-p_k^{j-1}\|_1\le \|p^0\|_1+\frac{1}{r_H}\sum_{j=1}^k\mathcal H(p_k^j-p_k^{j-1}),
\end{align*}
we deduce the existence a constant $C_3>0$, independent on $k$ and $i$, such that
\begin{equation*}\max_{i=0,\dots,k}\|p_k^i\|_{H^1}\le C_3.
\end{equation*}
Since $0\le \alpha_k^i\le 1$ in $\Omega$, there exists a constant $C_4>0$ independent on $k$ and $i$ such that
\begin{equation*}
\max_{i=0,\dots,k}\|\alpha_k^i\|_{H^1}\le\sqrt{|\Omega|}+\max_{i=0,\dots,k}\|\nabla\alpha_k^i\|_2\le C_4.
\end{equation*}
This concludes the proof.
\end{proof}

We introduce the {\it piecewise constant interpolant} $\overline \alpha_k\colon [0,T]\to H^1(\Omega;[0,1])$ as
\begin{align}\label{eq:overline_ak}
\overline \alpha_k(t)\coloneqq\alpha_k^i\quad\text{for $t\in (t_k^{i-1},t_k^i]$, $i=1,\dots,k$},\qquad \overline \alpha_k(0)=\alpha_k^0,
\end{align}
and similarly we consider $\overline u_k\colon [0,T]\to H^1(\Omega;\R^n)$, $\overline e_k\colon [0,T]\to L^2(\Omega;\M)$, $\overline p_k\colon [0,T]\to H^1(\Omega;\M)$, and $\overline w_k\to H^1(\Omega;\R^n)$. Moreover, we define $\overline \tau_k\colon[0,T]\to [0,T]$ as 
\begin{equation}\label{eq:overline_tauk}
\overline\tau_k(t)\coloneqq t_k^i\quad\text{for $t\in (t_k^{i-1},t_k^i]$, $i=1,\dots,k$},\qquad \overline\tau_k(0)=0.
\end{equation}
By definition, we have
$$\sum_{j=1}^k\mathcal H(p_k^j-p_k^{j-1})=\mathcal V_{\mathcal H}(\overline p_k;0,T),$$
and thanks to Lemma~\ref{lem:glslkey} we deduce that the piecewise constant interpolants satisfy the following estimate: there is a constant $C>0$ independent on $k\in\N$ and $t\in[0,T]$ such that 
\begin{equation}\label{eq:k_unif_bound_gl}
\|\overline e_k(t)\|_2+\|\overline p_k(t)\|_{H^1}+\|\overline \alpha_k(t)\|_{H^1}+\mathcal V_{\mathcal H}(\overline p_k;0,T)\le C\quad\text{for all $t\in[0,T]$}.
\end{equation}
We can now prove our existence result for global quasistatic evolution.

\begin{theorem}\label{thm:glob_qe}
Assume~\eqref{eq:w}--\eqref{eq:B3},~\eqref{eq:K}, and~\eqref{eq:ic}. Then there exists a global quasistatic evolution $(\alpha,u, e, p)$ from $[0,T]$ into $H^1(\Omega;[0,1])\times\mathcal A$ satisfying the initial condition 
\begin{equation}\label{eq:initial}
(\alpha(0),u(0), e(0), p(0))=(\alpha^0,u^0,e^0,p^0).    
\end{equation}
\end{theorem}

\begin{proof}
{\bf Step 1}. Starting from $(\alpha^0,u^0,e^0,p^0)$ we consider the incremental problem~\eqref{eq:gl_minki} for all $k\in\N$ and $i=1,\dots,k$, and we obtain a sequence of approximate solutions $(\overline\alpha_k,\overline u_k,\overline e_k,\overline p_k)$ from $[0,T]$ into $H^1(\Omega;[0,1])\times\mathcal A$. Since all functions $\overline\alpha_k$ are non increasing in time, we can apply a generalized version of the classical Helly Theorem to conclude that there
exist a subsequence, still denoted by $\overline\alpha_k$, and a function $\alpha\colon[0,T]\to H^1(\Omega;[0,1])$ non increasing in time, such that for all $t\in[0,T]$ as $k\to\infty$
$$\overline\alpha_k(t)\rightharpoonup \alpha(t)\quad\text{in $H^1(\Omega)$}.$$ 
In particular, thanks to Remark~\ref{rem:emb}, for all $t\in[0,T]$ and $\theta\in[1,6)$ we have as $k\to\infty$
$$\overline\alpha_k(t)\to \alpha(t)\quad\text{in $L^\theta(\Omega)$}.$$ 
Moreover, by~\eqref{eq:k_unif_bound_gl} there exists a constant $C>0$, independent on $k$, such that
$$r_H\mathcal V(\overline p_k;0,T)\le\mathcal V_{\mathcal H}(\overline p_k;0,T)=\sum_{j=1}^k\mathcal H(p_k^j-p_k^{j-1})\le C,\qquad\|\overline p_k(0)\|_{H^1}=\|p^0\|_{H^1}\le C.$$ 
Hence, there exist a further subsequence, still denoted by $\overline p_k$, and a function $p\colon[0,T]\to H^1(\Omega;\M)$ with bounded variation in $L^1(\Omega;\M)$ such that  for all $t\in[0,T]$ as $k\to\infty$
$$\overline p_k(t)\rightharpoonup p(t)\quad\text{in $H^1(\Omega;\M)$}.$$ 
In particular, by~\eqref{eq:Hdis-lsc} we get
$$\mathcal V_{\mathcal H}(p;0,T)\le \liminf_{k\to\infty}\mathcal V_{\mathcal H}(\overline p_k;0,T)\le C,$$
which yields that $p$ has bounded $\mathcal H$-variation. Furthermore, for all $t\in[0,T]$ and $\theta\in[1,6)$ we have as $k\to\infty$
$$\overline p_k(t)\to p(t)\quad\text{in $L^\theta(\Omega;\M)$}.$$ 

{\bf Step 2}. For all fixed $t\in[0,T]$ there exist a subsequence $k_j$ and functions $e^*\in L^2(\Omega;\M)$ and $u^*\in H^1(\Omega;\R^n)$, depending on $t\in[0,T]$, such that as $j\to\infty$
$$\overline e_{k_j}(t)\rightharpoonup e^*\quad\text{in $L^2(\Omega;\M)$},\qquad \overline u_{k_j}(t)\to u^*\quad\text{in $H^1(\Omega;\R^n)$}.$$
By construction $(\overline \alpha_k(t),\overline u_k(t),\overline e_k(t),\overline p_k(t))\in H^1(\Omega;[0,1])\times \mathcal A(\overline w_k(t))$ and in view of~\eqref{eq:gl_stki}
\begin{align*}
\mathcal E(\overline\alpha_k(t),\overline e_k(t),\overline p_k(t))\le \mathcal E(\beta,\eta,q)+\mathcal H(q-\overline p_k(t))\quad\text{for all $(\beta,v,\eta,q)\in \mathcal D(\overline \alpha_k(t))\times \mathcal A(\overline w_k(t))$}.
\end{align*}
By writing the inequality for $k_j$ and using Lemma~\ref{lem:stab} we deduce
\begin{align*}
\mathcal E(\alpha(t),e^*,p(t))\le \mathcal E(\beta,\eta,q)+\mathcal H(q-p(t))\quad\text{for all $(\beta,v,\eta,q)\in \mathcal D(\alpha(t))\times \mathcal A(w(t))$}.
\end{align*}
Therefore, we can argue as in Lemma~\ref{lem:measurability} to derive that $(u^*,e^*)$ is uniquely determined for all $t\in[0,T]$. In particular, by defining $u(t)\coloneqq u^*$ and $e(t)\coloneqq e^*$ we deduce that as $k\to\infty$
$$\overline e_k(t)\rightharpoonup e(t)\quad\text{in $L^2(\Omega;\M)$},\qquad \overline u_k(t)\to u(t)\quad\text{in $H^1(\Omega;\R^n)$},$$
and the quadruple $(\alpha,u,e,p)$ satisfies (qs0), (qs1), and~\eqref{eq:initial}.

{\bf Step 3}. By Lemma~\ref{lem:measurability}, we deduce that $(\alpha,u,e,p)$ are strongly measurable from $[0,T]$ into $H^1(\Omega)\times\mathcal A$. We define $\underline e_k\colon [0,T]\to L^2(\Omega;\M)$ as 
\begin{align*}
\underline e_k(t)\coloneqq e_k^{i-1}\quad\text{for $t\in [t_k^{i-1},t_k^i)$, $i=1,\dots,k$},\qquad \underline e_k(T)=e_k^k.
\end{align*}
By the discrete energy inequality~\eqref{eq:gl_denin}, for all $k\in\N$, $t\in (t_k^{i-1},t_i]$, and $i=1,\dots,k$ we have
\begin{equation}\label{eq:dis-qs-en-in}
\mathcal E(\overline\alpha_k(t),\overline e_k(t),\overline p_k(t))+\mathcal V_{\mathcal H}(\overline p_k;0,t_k^i)\le \mathcal E(\alpha^0,e^0,p^0)+\int_0^{\overline\tau_k(t)}(\C \underline e_k(r),\EE \dot w(r))_2\,\de r+\delta_k.
\end{equation}
Since $\underline e_k(t)=\overline e_k(t-\tau_k)$ for a.e.\ $t\in [\tau_k,T]$, we derive that as $k\to\infty$
$$\underline e_k \rightharpoonup e\quad\text{in $L^2((0,T);L^2(\Omega;\M))$}.$$
Hence, we can use the convergences of Steps 1 and 2 to pass to the limit as $k\to\infty$ in~\eqref{eq:dis-qs-en-in}, and to obtain that the quadruple $(\alpha,u,e,p)$ satisfies for all $t\in[0,T]$
$$\mathcal E(\alpha(t),e(t),p(t))+\mathcal V_{\mathcal H}(p;0,t)\le \mathcal E(\alpha^0,e^0,p^0)+\int_0^t(\C e(r),\EE \dot w(r))_2\,\de r,$$
being all the terms on the left-hand side are sequentially lower semicontinuous with respect to the convergences of $\overline \alpha_k,\overline e_k,\overline p_k$. It remains to prove the opposite inequality.

We fix $t\in(0,T]$ and for all $k\in\N$ we consider a collection of points $\{s_k^i\}_{i=0}^k$ in $[0,t]$ satisfying
$$0=s_k^0<s_k^1<\cdots<s_k^{k-1}<s_k^k=t,\quad\lim_{k\to\infty}\max_{i=1,\dots,k}|s_k^i-s_k^{i-1}|\to 0.$$
For all $i=1,\dots,k$ we have $$(\alpha(s_k^i),u(s_k^i)-w(s_k^i)+w(s_k^{i-1}),e(s_k^i)-\EE w(s_k^i)+\EE w(s_k^{i-1}),p(s_k^i))\in \mathcal D(\alpha(s_k^{i-1}))\times \mathcal A(w(s_k^{i-1})).$$
By using (qs1) in $t=s_k^{i-1}$ for all $k\in\N$ and $i=1,\dots,k$ we deduce
 \begin{align*}
\mathcal E(\alpha(s_k^{i-1}),e(s_k^{i-1}),p(s_k^{i-1}))&\le \mathcal E(\alpha(s_k^i),e(s_k^i)-\EE w(s_k^i)+\EE w(s_k^{i-1}),p(s_k^i))+\mathcal H(p(s_k^i)-p(s_k^{i-1}))\\
&=\mathcal E(\alpha(s_k^i),e(s_k^i),p(s_k^i))+\mathcal H(p(s_k^i)-p(s_k^{i-1}))\\
&\quad-(\C e(s_k^i),\EE w(s_k^i)-\EE w(s_k^{i-1}))_2+\mathcal Q(\EE w(s_k^i)-\EE w(s_k^{i-i}))\\
&\le \mathcal E(\alpha(s_k^i),e(s_k^i),p(s_k^i))+\mathcal H(p(s_k^i)-p(s_k^{i-1}))-\int_{s_k^{i-1}}^{s_k^i}(\C e(s_k^i),\EE \dot w(r))_2\,\de r\\
&\quad +\max_{i=1,\dots,k}|s_k^i-s_k^{i-1}|\frac{\gamma_2}{2}\int_0^t\|\EE \dot w(r)\|_2^2\,\de r.
\end{align*}
By summing over $i=1,\dots,k$ we get
\begin{align}
\mathcal E(\alpha^0,e^0,p^0)&\le \mathcal E(\alpha(t),e(t),p(t))+\mathcal V_{\mathcal H}(p;0,t)-\int_0^t(\C \tilde e_k(r),\EE \dot w(r))_2\,\de r\nonumber\\
&\quad +\max_{i=1,\dots,k}|s_k^i-s_k^{i-1}|\frac{\gamma_2}{2}\int_0^t\|\EE \dot w(r)\|_2^2\,\de r,\label{eq:other_in}
\end{align}
where $\tilde e_k\colon [0,T]\to L^2(\Omega;\M)$ is defined as
$$\tilde e_k(t)\coloneqq e(s_k^i)\quad\text{for all $t\in[s_k^{i-1},s_k^i)$},\qquad \tilde e_k(T)\coloneqq e(T).$$
Since $\tilde e_k(r)\to e(r)$ in $L^2(\Omega;\M)$ for a.e.\ $r\in[0,t]$ as $k\to\infty$, we can use the Dominated Convergence Theorem to derive
\begin{equation*}
\lim_{k\to\infty}\int_0^t(\C \tilde e_k(r),\EE \dot w(r))_2\,\de r=\int_0^t(\C e(r),\EE \dot w(r))_2\,\de r.
\end{equation*}
By passing to the limit as $k\to\infty$ in~\eqref{eq:other_in} we get the other opposite energy inequality. Hence, $(\alpha,u,e,p)$ satisfies (qs2) and it is a global quasistatic evolution from $[0,T]$ into $H^1(\Omega;[0,1])\times\mathcal A$. 
\end{proof}

\subsection{Continuity in time}

In this subsection we investigate the regularity in time of global quasistatic evolutions $(\alpha,u,e,p)$ in the sense of Definition~\ref{def:QSE}. More precisely, we show that every evolution is continuous at every continuity point $t$ of $\alpha\colon [0,T]\to L^1(\Omega)$. We point out that such a property improves the continuity result of Lemma~\ref{lem:measurability}.

We start by proving several preliminary results. In the first one we show that every global quasistatic evolutions is strongly continuous up to a countable set.

\begin{lemma}\label{lem:continuity}
Assume~\eqref{eq:d}--\eqref{eq:B3} and~\eqref{eq:K}. Let $(\alpha,u,e,p)$ be a global quasistatic evolution from $[0,T]$ into $H^1(\Omega;[0,1])\times\mathcal A$. There exists a countable set $N\subset [0,T]$ such that the quadruple $(\alpha,u,e,p)$ is strongly continuous from $[0,T]\setminus N$ into $H^1(\Omega)\times \mathcal A$. 
\end{lemma}

\begin{proof}
From Lemma~\ref{lem:measurability} we know that there exists a countable set $N_1\subset [0,T]$ such that every global quasistatic evolution $(\alpha,u,e,p)$ is weakly continuous from $[0,T]\setminus N$ into $H^1(\Omega)\times \mathcal A$. Moreover, since the map $t\mapsto \mathcal V_{\mathcal H}(p;0,t)$ from $[0,T]$ into $\R$ is increasing, there exists a countable set $N_2\subset [0,T]$ such that $t\mapsto \mathcal V_{\mathcal H}(p;0,t)$ is continuous from $[0,T]\setminus N_2$ into $\R$. In particular, the  energy equality (qs2) yields that the map $t\mapsto \mathcal E(\alpha(t),e(t),p(t))$ is continuous from $[0,T]\setminus N_2$ into $\R$.

Let us define $N\coloneqq N_1\cup N_2$. We fix $t_\infty\in[0,T]\setminus N$ and a sequence $(t_k)_k\subset [0,T]\setminus N$ such that $t_k\to t_\infty$ as $k\to\infty$. Thanks to Lemma~\ref{lem:measurability} we have
\begin{align*}
&\alpha(t_k)\rightharpoonup \alpha(t_\infty)\quad\text{in $H^1(\Omega)$ as $k\to\infty$},& &u(t_k)\rightharpoonup u(t_\infty)\quad\text{in $H^1(\Omega;\R^n)$ as $k\to\infty$},\\
& e(t_k)\rightharpoonup e(t_\infty)\quad\text{in $L^2(\Omega;\M)$ as $k\to\infty$},& & p(t_k)\rightharpoonup p(t_\infty)\quad\text{in $H^1(\Omega;\M)$ as $k\to\infty$},
\end{align*}
and by the Sobolev Embedding Theorem (see also Remark~\ref{rem:emb}), for all $\theta\in [1,6)$ we get
\begin{align*}
&\alpha(t_k)\to \alpha(t_\infty)\quad\text{in $L^\theta(\Omega)$ as $k\to\infty$},& &u(t_k)\to u(t_\infty)\quad\text{in $L^\theta(\Omega;\R^n)$ as $k\to\infty$},\\
&p(t_k)\to p(t_\infty)\quad\text{in $L^\theta(\Omega;\M)$ as $k\to\infty$} & &w(t_k)\to w(t_\infty)\quad\text{in $H^1(\Omega;\R^n)$ as $k\to\infty$}.
\end{align*}  
In particular, we derive that
\begin{align*}
&\mathcal E(\alpha(t_\infty),e(t_\infty),p(t_\infty))\\
%&\le\liminf_{k\to\infty}\mathcal Q(e(t_k))+D(\alpha(t_\infty))+\liminf_{k\to\infty}\|\nabla\alpha(t_k)\|_2^2+\tilde Q(\alpha(t_\infty),p(t_\infty))+\liminf_{k\to\infty}\|\nabla p(t_k)\|_2^2\\
&\le\liminf_{k\to\infty}\mathcal Q(e(t_k))+\lim_{k\to\infty}D(\alpha(t_k))+\liminf_{k\to\infty}\|\nabla\alpha(t_k)\|_2^2+\lim_{k\to\infty}\tilde Q(\alpha(t_k),p(t_k))+\liminf_{k\to\infty}\|\nabla p(t_k)\|_2^2\\
&\le \lim_{k\to\infty}\mathcal E(\alpha(t_k),e(t_k),p(t_k))=\mathcal E(\alpha(t_\infty),e(t_\infty),p(t_\infty)).
\end{align*}
This implies that
\begin{align*}
Q(e(t_\infty))=\liminf_{k\to\infty}\mathcal Q(e(t_k)),\qquad \|\nabla\alpha(t_\infty)\|_2^2=\liminf_{k\to\infty}\|\nabla\alpha(t_k)\|_2^2,\qquad\|\nabla p(t_k)\|_2^2=\liminf_{k\to\infty}\|\nabla p(t_k)\|_2^2.    
\end{align*}
If we use also that
$$\liminf_{k\to\infty}(a_k+b_k)\le\limsup_{k\to\infty}a_k+\liminf_{k\to\infty}b_k\le \limsup_{k\to\infty}(a_k+b_k)\quad\text{for all two sequences $(a_k)_k,(b_k)_k\subset\R$},$$
we conclude that
\begin{align*}
Q(e(t_\infty))=\lim_{k\to\infty}\mathcal Q(e(t_k)),\qquad \|\nabla\alpha(t_\infty)\|_2^2=\lim_{k\to\infty}\|\nabla\alpha(t_k)\|_2^2,\qquad\|\nabla p(t_k)\|_2^2=\lim_{k\to\infty}\|\nabla p(t_k)\|_2^2.    
\end{align*}
Hence, thanks to~\eqref{eq:C2} we obtain the following strong convergences:
\begin{align*}
&\alpha(t_k)\to \alpha(t_\infty)\quad\text{in $H^1(\Omega)$ as $k\to\infty$},&
& e(t_k)\to e(t_\infty)\quad\text{in $L^2(\Omega;\M)$ as $k\to\infty$},\\
& p(t_k)\to p(t_\infty)\quad\text{in $H^1(\Omega;\M)$ as $k\to\infty$}.
\end{align*}
Finally, by Korn Inequality, we derive
$$\|\nabla u(t_k)-\nabla u(t_\infty)\|_2\le C\left(\|u(t_k)-u(t_\infty)\|_2+\|e(t_k)-e(t_\infty)\|_2+\|p(t_k)-p(t_\infty)\|_2\right)\to 0\quad\text{as $k\to\infty$}.$$
Therefore, the quadruple $(\alpha,u,e,p)$ is strongly continuous from $[0,T]\setminus N$ into $H^1(\Omega)\times \mathcal A$.
\end{proof}

Moreover, we introduce the following definition.

\begin{definition}
For all $\alpha\in H^1(\Omega;[0,1])$, $e\in L^2(\Omega;\M)$, and $p\in H^1(\Omega;\M)$ we define 
\begin{align*}
\partial_\alpha \mathcal E(\alpha,e,p)[\beta]\coloneqq &\int_\Omega \dot{d}(\alpha(x))\beta(x) \,\de x+2\int_\Omega \nabla\alpha(x)\cdot\nabla\beta(x) \,\de x+\int_\Omega \dot{\B}(\alpha(x))\beta(x)p(x):p(x) \,\de x
\end{align*}
 for all $\beta\in H^1(\Omega)$.
 \end{definition}
 
 By Remark~\ref{rem:emb} we deduce that $\partial_\alpha \mathcal E(\alpha,e,p)[\beta]$ is well defined, and we can identify $\partial_\alpha\mathcal E(\alpha,e,p)$ as an element of $ (H^1(\Omega))'$ by setting
 \begin{align*}
 \langle \partial_\alpha\mathcal E(\alpha,e,p),\beta\rangle\coloneqq \partial_\alpha \mathcal E(\alpha,e,p)[\beta]\quad\text{for all $\beta\in H^1(\Omega)$}.
 \end{align*}
 
In the next lemma we show that every global quasistatic evolution satisfies the following variational inequalities.

\begin{lemma}\label{lem:qs_var_ineq}
Assume~\eqref{eq:d}--\eqref{eq:B3} and~\eqref{eq:K}. Let $(\alpha,u,e,p)$ be a global quasistatic evolution from $[0,T]$ into $H^1(\Omega;[0,1])\times\mathcal A$. 
\begin{itemize}
\item[(a)] For all $t\in[0,T]$ we have
\begin{equation}\label{eq:st_1}
\partial_\alpha \mathcal E(\alpha(t),e(t),p(t))[\beta]\ge 0\quad\text{for all $\beta\in H^1(\Omega)$ with $\beta\le 0$}.
\end{equation}
\item[(b)] For all $t\in[0,T]$ and $(v,\eta,q)\in \mathcal A(0)$ we have
\begin{equation}\label{eq:st_2}
-\mathcal H(q)\le (\C e(t),\eta)_2+2(\B(\alpha(t))p(t),q)_2+2(\nabla p(t),\nabla q)_2\le \mathcal H(-q).
\end{equation}
\item[(c)] For all $t\in[0,T]$ and $(v,\eta,q)\in \mathcal A(w(t))$
\begin{align}
&\mathcal Q(e(t))+\mathcal Q(\eta-e(t))+\tilde Q(\alpha(t),p(t))+\tilde Q(\alpha(t),q-p(t))+\|\nabla p(t)\|_2^2+\|\nabla q-\nabla p(t)\|_2^2\nonumber\\
&\le \mathcal Q(\eta)+\tilde Q(\alpha(t),q)+\|\nabla q\|_2^2+\mathcal H(q-p(t)).\label{eq:st_3}
\end{align}
\end{itemize}
\end{lemma}

\begin{proof}
(a) We first extend $d$ and $\mathbb B$ to $(-\infty,0)$ by setting
\begin{equation}\label{eq:dB_ext}
d(\alpha)\coloneqq d(0)+\alpha \dot{d}(0)\quad\text{and}\quad \mathbb B(\alpha)=\mathbb B(0)+\alpha\dot{\mathbb B}(0)\quad\text{for all $\alpha\in (-\infty,0)$}.    
\end{equation}
By construction, 
\begin{align*}
&d\in C^1((-\infty,1];[0,\infty))\cap C^{0,1}((-\infty,1];[0,\infty)),\\
 &\mathbb B\in C^1((-\infty,1];\Lin(\M;\M))\cap C^{0,1}((-\infty,1];\Lin(\M;\M)).
\end{align*}
In particular, by the Dominated Convergence Theorem for all $\alpha\in H^1(\Omega;[0,1])$, $e\in L^2(\Omega;\M)$, and $p\in H^1(\Omega;\M)$, and for all $\beta\in H^1(\Omega)$ with $\beta\le 0$ there exists
\begin{equation}\label{eq:par_der}
\lim_{\lambda\to 0^+}\frac{\mathcal E(\alpha+\lambda \beta,e,p)-\mathcal E(\alpha,e,p)}{\lambda}=\partial_\alpha \mathcal E(\alpha,e,p)[\beta].
\end{equation}

We fix $t\in[0,T]$. Let $\beta\in H^1(\Omega)$ with $\beta\le 0$  and let $\lambda>0$. Since $(\alpha(t)+\lambda\beta)^+\in \mathcal D(\alpha(t))$, by (qs1) we get
\begin{align*}
\mathcal E(\alpha(t),e(t),p(t))&\le \mathcal E((\alpha(t)+\lambda\beta)^+,e(t),p(t))
\end{align*}
Moreover, by~\eqref{eq:d},~\eqref{eq:B1},~\eqref{eq:dB_ext}, and the fact that
$\alpha(t)\ge 0$ we have
\begin{align*}
&D((\alpha(t)+\lambda\beta)^+)\le D(\alpha(t)+\lambda\beta),& &\|\nabla((\alpha(t)+\lambda\beta)^+)\|^2_2\le \|\nabla(\alpha(t)+\lambda\beta)\|_2^2,\\
&\tilde Q((\alpha(t)+\lambda \beta)^+,p(t))\le \tilde Q(\alpha(t)+\lambda \beta,p(t)).  & &
\end{align*}
Hence, we derive
\begin{align*}
\mathcal E(\alpha(t),e(t),p(t))\le \mathcal E(\alpha(t)+\lambda\beta,e(t),p(t)).
\end{align*}
By dividing for $\lambda>0$, sending $\lambda\to 0^+$, and using~\eqref{eq:par_der} we get~\eqref{eq:st_1}.

(b) For all $(v,\eta,q)\in \mathcal A(0)$ and $\lambda>0$ we test (qs1) with 
\begin{equation*}
(\alpha(t),u(t)+\lambda v,e(t)+\lambda \eta,p(t)+\lambda q)\in\mathcal D(\alpha(t))\times\mathcal A(w(t)),
\end{equation*}
and we get
\begin{align*}
\mathcal Q(e(t))+\tilde Q(\alpha(t),p(t))+\|\nabla p(t)\|_2^2\le \mathcal Q(e(t)+\lambda \eta)+\tilde Q(\alpha(t),p(t)+\lambda q)+\|\nabla (p(t)+\lambda q)\|_2^2+\lambda\mathcal H(q).
\end{align*}
By dividing for $\lambda>0$ and sending $\lambda\to 0^+$ we get one inequality of~\eqref{eq:st_2}. To get the other inequality, it is enough to replace $(v,\eta,q)\in\mathcal A(0)$ with $(-v,-\eta,-q)\in\mathcal A(0)$.

(c) By~\eqref{eq:st_2} for all $(v,\eta,q)\in \mathcal A(w(t))$ we get
\begin{align*}
&\mathcal Q(e(t))+\mathcal Q(\eta-e(t))-\mathcal Q(\eta)+\tilde Q(\alpha(t),p(t))+\tilde Q(\alpha(t),q-p(t))-\tilde Q(\alpha(t),q)\\
&\quad+\|\nabla p(t)\|_2^2+\|\nabla q-\nabla p(t)\|_2^2-\|\nabla q\|_2^2\\
&=(\C e(t),e(t)-\eta)_2+2(\B(\alpha(t))p(t),p(t)-q)_2+2(\nabla p(t),\nabla p(t)-\nabla q)_2\le \mathcal H(q-p(t)),
\end{align*}
which is~\eqref{eq:st_3}.
\end{proof}

\begin{remark}
The assumptions $\dot d(0)\le 0$ and $\dot{\B}(0)\xi:\xi\le 0$ for all $\xi\in\M$ in~\eqref{eq:d} and~\eqref{eq:B1} are needed to deduce~\eqref{eq:st_1} for all $\beta\in H^1(\Omega)$ with $\beta\le 0$.
\end{remark}

By~\eqref{eq:st_2} we deduce that every global quasistatic evolution $(\alpha,u,e,p)$ satisfies 
\begin{equation}\label{eq:div0}
\div \C e(t)=0\quad\text{in $\mathcal D'(\Omega;\R^n)$ for all $t\in[0,T]$}.
\end{equation}
Indeed, since $(\phi,\EE \phi,0)\in\mathcal A(0)$ for all $\phi\in C_c^\infty(\Omega;\R^n)$ and $\mathcal H(0)=0$, by~\eqref{eq:st_2} we deduce
$$(\C e(t),\EE \phi)_2=0\quad\text{for all $\phi\in C_c^\infty(\Omega;\R^n)$},$$
which is~\eqref{eq:div0}. Let us recall also the following integration by parts formula.

\begin{lemma}
Let $e\in L^2(\Omega;\M)$ satisfy
\begin{align}\label{eq:divCe}
\div\C e=0\quad\text{in $\mathcal D'(\Omega;\R^n)$.}
\end{align} 
Then, for all $w\in H^1(\Omega;\R^n)$ and $(v,\eta,q)\in \mathcal A(w)$ we have
\begin{align}\label{eq:int_for}
(\C e,q)_2=-(\C e, \eta-\EE w)_2.
\end{align}
\end{lemma}
\begin{proof}
Since $e\in L^2(\Omega;\M)$, by using a density argument in~\eqref{eq:divCe} we derive
$$ (\C e,\EE \psi)_2=0\quad\text{for all $\psi\in H^1_0(\Omega;\R^n)$}.$$
In particular, for all $w\in H^1(\Omega;\R^n)$ and $(v,\eta,q)\in \mathcal A(w)$, we get
$$ (\C e,q)_2=(\C e,\EE u-\eta)_2=-(\C e, \eta-\EE w)_2+(\C e, \EE (u-w))_2=-(\C e, \eta-\EE w)_2,$$
which is~\eqref{eq:int_for}.
\end{proof}

Moreover, we need the following version of Gronwall Inequality, whose proof is analogous to the one of~\cite[Lemma 5.3]{DMDSMo}.

\begin{lemma}\label{lem:gronw}
Let $a,b\in\R$ with $a<b$. Let $f\colon [a,b]\to [0,\infty)$ be a bounded measurable function, let $g\colon [a,b]\to [0,\infty)$ be a non decreasing function, and let $h\colon [a,b]\to [0,\infty)$ be an integrable function. Suppose that
\begin{align}\label{eq:gronw1}
f(t)^2\le g(t)^2+\int_a^tf(r)h(r)\,\de r+\left(\int_a^t h(r)\,\de r\right)^2\quad\text{for all $t\in[a,b]$}.
\end{align}
Then, 
\begin{align}\label{eq:gronw2}
f(t)\le g(t)+\frac{3}{2}\int_a^th(r)\,\de r\quad\text{for all $t\in[a,b]$}.    
\end{align}
\end{lemma}

\begin{proof}
Let $t_0\in[a,b]$ be fixed. Define
\begin{equation}\label{eq:gamma0}
\gamma_0\coloneqq g(t_0)^2+\left(\int_a^{t_0} h(r)\,\de r\right)^2.    
\end{equation}
If $\gamma_0=0$, then $g(t_0)=0$ and $h=0$ a.e.\ on $[0,t_0]$, which implies that $f(t_0)=0$ by~\eqref{eq:gronw1}, and so~\eqref{eq:gronw2} is satisfied. Therefore, we may assume that $\gamma_0>0$ and we define the function
$$V(t)\coloneqq\int_a^tf(r)h(r)\,\de r\quad\text{for all $r\in[a,t_0]$}.$$
By assumptions $V\in AC([a,t_0])$ and $\dot{V}(t)=f(t)h(t)$ for a.e.\ $t\in[a,t_0]$. Therefore, thanks to~\eqref{eq:gronw1} we derive that $\dot{V}(t)\le h(t)\sqrt{\gamma_0+V(t)}$ for a.e.\ $t\in[a,t_0]$. Integrating between $a$ and $t_0$, and using~\eqref{eq:gamma0} we get
$$2\sqrt{V(t_0)+\gamma_0}\le 2\sqrt{\gamma_0}+\int_a^{t_0}h(r)\,\de r\le 2g(t_0)+3\int_a^{t_0}h(r)\,\de r.$$
On the other hand, by~\eqref{eq:gronw1} we derive
$f(t_0)\le \sqrt{V(t_0)+\gamma_0}$, and so 
$$2f(t_0)\le 2g(t_0)+3\int_a^{t_0}h(r)\,\de r,$$
which gives~\eqref{eq:gronw2}.
\end{proof}

We can finally prove the strong continuity of a global quasistatic evolution $(\alpha,u,e,p)$ from $[0,T]$ into $H^1(\Omega)\times \mathcal A$ at every continuity point $t$ of $\alpha\colon [0,T]\to L^1(\Omega)$. 

\begin{theorem}\label{thm:cont_time}
Assume~\eqref{eq:d}--\eqref{eq:B3} and~\eqref{eq:K}. Let $(\alpha,u,e,p)$ be a global quasistatic evolution from $[0,T]$ into $H^1(\Omega;[0,1])\times \mathcal A$. Then, there exists a constant $C>0$ such that for all $t_1,t_2\in[0,T]$ with $t_1<t_2$ we have
\begin{align}
\|\alpha(t_2)-\alpha(t_1)\|_{H^1}&+\|e(t_2)-e(t_1)\|_2+\|p(t_2)-p(t_1)\|_{H^1}\nonumber\\
&\le C\left(\|\alpha(t_2)-\alpha(t_1)\|_1+\int_{t_1}^{t_2}\|\EE \dot w(r)\|_2\,\de r\right),\label{eq:cont1}\\
\|u(t_2)-u(t_1)\|_{H^1}&\le C\left(\|\alpha(t_2)-\alpha(t_1)\|_1+\int_{t_1}^{t_2}\|\dot w(r)\|_{H^1}\,\de r\right).\label{eq:cont2}
\end{align}
\end{theorem}

\begin{remark}
Therefore, every global quasistatic evolution $(\alpha,u,e,p)$ is (strongly) continuous from $[0,T]$ into $H^1(\Omega)\times\mathcal A$, except for a countable subset of $[0,T]$, which is the set of discontinuity points of $\alpha$ with respect to the convergence in $L^1(\Omega)$. Moreover, if $\alpha\in AC([0,T];L^1(\Omega))$, then $(\alpha,u,e,p)$ are absolutely continuous from $[0,T]$ into $H^1(\Omega)\times \mathcal A$.
\end{remark}

\begin{proof}
We fix $t_1,t_2\in[0,T]$ with $t_1<t_2$. We test~\eqref{eq:st_1} in $t_1$ with $\beta\coloneqq\alpha(t_2)-\alpha(t_1)$ and we use the identity
$$\|\nabla \alpha(t_2)\|_2^2-\|\nabla \alpha(t_2)-\nabla\alpha(t_1)\|_2^2-\|\nabla \alpha(t_1)\|_2^2=2(\nabla\alpha(t_1),\nabla\alpha(t_2)-\nabla\alpha(t_1))_2,$$
to derive that
\begin{align}
&\|\nabla \alpha(t_2)-\nabla\alpha(t_1)\|_2^2\nonumber\\
&\le \|\nabla \alpha(t_2)\|_2^2-\|\nabla \alpha(t_1)\|_2^2+(\dot{d}(\alpha(t_1)),\alpha(t_2)-\alpha(t_1))_2+(\dot{\B}(\alpha(t_1))(\alpha(t_2)-\alpha(t_1))p(t_1),p(t_1))_2.\label{eq:lowbou1}  
\end{align}
We now test~\eqref{eq:st_3} in $t_1$ with
\begin{equation*}
(v,\eta,q)\coloneqq(u(t_2)-w(t_2)+w(t_1),e(t_2)-\EE w(t_2)+\EE w(t_1),p(t_2))\in \mathcal A(w(t_1)),
\end{equation*}
and we get
\begin{align*}
&\mathcal Q(e(t_1))+\mathcal Q(e(t_2)-e(t_1)-\EE(w(t_2)-w(t_1)))\\
&\quad+\tilde Q(\alpha(t_1),p(t_1))+\tilde Q(\alpha(t_1),p(t_2)-p(t_1))+\|\nabla p(t_1)\|_2^2+\|\nabla p(t_2)-\nabla p(t_1)\|_2^2\\
&\le \mathcal Q(e(t_2)-\EE(w(t_2)-w(t_1)))+\tilde Q(\alpha(t_1),p(t_2))+\|\nabla p(t_2)\|_2^2+\mathcal H(p(t_2)-p(t_1)).
\end{align*}
In particular
\begin{align}
&\mathcal Q(e(t_2)-e(t_1))+\tilde Q(\alpha(t_1),p(t_2)-p(t_1))+\|\nabla p(t_2)-\nabla p(t_1)\|_2^2\nonumber\\
&\le \mathcal Q(e(t_2))-\mathcal Q(e(t_1))-(\C e(t_1),\EE(w(t_2)-w(t_1)))_2\nonumber\\
&\quad+\tilde Q(\alpha(t_1),p(t_2))-\tilde Q(\alpha(t_1),p(t_1))+\|\nabla p(t_2)\|_2^2-\|\nabla p(t_1)\|_2^2+\mathcal H(p(t_2)-p(t_1)).\label{eq:lowbou2}
\end{align}
Moreover, by (qs2) evaluated in $t_1$ and $t_2$, we have that 
\begin{align}
\mathcal H(p(t_2)-p(t_1))&\le \mathcal V_{\mathcal H}(p;t_1,t_2)\nonumber\\
&= \mathcal E(\alpha(t_1),e(t_1),p(t_1))-\mathcal E(\alpha(t_2),e(t_2),p(t_2))+\int_{t_1}^{t_2}(\C e(s),\EE \dot w(s))_2\,\de s.\label{eq:lowbou3}   
\end{align}
If we sum~\eqref{eq:lowbou1} and~\eqref{eq:lowbou2} and we use~\eqref{eq:lowbou3} we deduce that
\begin{align*}
&\mathcal Q(e(t_2)-e(t_1))+\tilde Q(\alpha(t_1),p(t_2)-p(t_1))+\|\nabla p(t_2)-\nabla p(t_1)\|_2^2+  \|\nabla \alpha(t_2)-\nabla \alpha(t_1)\|_2^2\\
&\le \int_{t_1}^{t_2}(\C (e(s)-e(t_1)),\EE \dot w(s))_2\,\de s+D(\alpha(t_1))-D(\alpha(t_2))+(\dot{d}(\alpha(t_1)),\alpha(t_2)-\alpha(t_1))_2\\
&\quad+\tilde Q(\alpha(t_1),p(t_2))-\tilde Q(\alpha(t_2),p(t_2))+(\dot{\B}(\alpha(t_1))(\alpha(t_2)-\alpha(t_1))p(t_1),p(t_1))_2.
\end{align*}
By~\eqref{eq:d} and~\eqref{eq:B1}, for all $\alpha_1,\alpha\in [0,1]$ and $\xi\in\M$ we have
\begin{align*}
|d(\alpha_1)-d(\alpha_2)-\dot{d}(\alpha_1)(\alpha_1-\alpha_2)| &\le \|\ddot{d}\|_\infty|\alpha_1-\alpha_2|^2,\\
|\B(\alpha_1)\xi:\xi-\B(\alpha_2)\xi:\xi-\dot{\B}(\alpha_1)(\alpha_1-\alpha_2)\xi:\xi| &\le \|\ddot{\B}\|_\infty|\alpha_1-\alpha_2|^2|\xi|^2.
\end{align*}
Therefore, thnkas to~\eqref{eq:C2} and the uniform estimates~\eqref{eq:unif_bound} (see also Remark~\ref{rem:unif_bound}), we have
\begin{align}
&\frac{\gamma_2}{2}\|e(t_2)-e(t_1)\|_2^2+\tilde Q(\alpha(t_1),p(t_2)-p(t_1))+\|\nabla p(t_2)-\nabla p(t_1)\|_2^2+  \|\nabla \alpha(t_2)-\nabla \alpha(t_1)\|_2^2\nonumber\\
&\le \int_{t_1}^{t_2}(\C (e(s)-e(t_1)),\EE \dot w(s))_2\,\de s+\|\ddot{d}\|_\infty\|\alpha(t_2)-\alpha(t_1)\|_2^2+\|\ddot{\B}\|_\infty\|\alpha(t_2)-\alpha(t_1)\|_4^2\|p(t_1)\|_4^2\nonumber\\
&\quad+\tilde Q(\alpha(t_1),p(t_2))-\tilde Q(\alpha(t_2),p(t_2))-\tilde Q(\alpha(t_1),p(t_1))+\tilde Q(\alpha(t_2),p(t_1)),\nonumber\\
&=\int_{t_1}^{t_2}(\C (e(s)-e(t_1)),\EE \dot w(s))_2\,\de s+\|\ddot{d}\|_\infty\|\alpha(t_2)-\alpha(t_1)\|_2^2+\|\ddot{\B}\|_\infty\|\alpha(t_2)-\alpha(t_1)\|_4^2\|p(t_1)\|_4^2\nonumber\\
&\quad+((\B(\alpha(t_1))-\B(\alpha(t_2))(p(t_2)-p(t_1)),p(t_2)+p(t_1))_2\nonumber\\
&\le C_1\|\alpha(t_2)-\alpha(t_1)\|_4^2+C_1\|\alpha(t_1)-\alpha(t_2)\|_4\|p(t_2)-p(t_1)\|_2+C_1\int_{t_1}^{t_2}\|e(s)-e(t_1)\|_2\|\EE \dot w(s)\|_2\,\de s\label{eq:part1}
\end{align}
for a constant $C_1>0$.

By Remark~\ref{rem:BCS}, the uniform estimates~\eqref{eq:unif_bound}, and~\eqref{eq:rH},~\eqref{eq:lowbou3}, there exists a constant $C_2>0$ such that
\begin{align*}
r_H\|p(t_2)-p(t_1)\|_1&\le \mathcal H(p(t_2)-p(t_1))\\
&\le \mathcal E(\alpha(t_1),e(t_1),p(t_1))-\mathcal E(\alpha(t_2),e(t_2),p(t_2))+\int_{t_1}^{t_2}(\C e(s),\EE \dot w(s))_2\,\de s\\
&\le \mathcal Q(e(t_1))-\mathcal Q(e(t_2))+D(\alpha(t_1))-D(\alpha(t_2))+\|\nabla \alpha(t_1)\|_2^2-\|\nabla \alpha(t_2)\|_2^2\\
&\quad+\tilde Q(\alpha(t_1),p(t_1))-\tilde Q(\alpha(t_1),p(t_2))+\tilde Q(\alpha(t_1),p(t_2))-\tilde Q(\alpha(t_2),p(t_2))\\
&\quad+\|\nabla p(t_1)\|_2^2-\|\nabla p(t_2)\|_2^2+\int_{t_1}^{t_2}(\C e(s),\EE \dot w(s))_2\,\de s\\
&\le C_2\left(\|e(t_2)-e(t_1)\|_2+\|\alpha(t_2)-\alpha(t_1)\|_4+\|\nabla\alpha(t_2)-\nabla\alpha(t_1)\|_2\right)\\
&\quad+C_2\left(\sqrt{\tilde Q(\alpha(t_1),p(t_2)-p(t_1))}+\|\nabla p(t_2)-\nabla p(t_1)\|_2+\int_{t_1}^{t_2}\|\EE \dot w(s)\|_2\,\de s\right).
\end{align*}
Therefore, if we use also~\eqref{eq:part1} we deduce the existence of a constant $C_3>0$ such that
\begin{align}
&r_H^2\|p(t_2)-p(t_1)\|_1^2\nonumber\\
&\le 3C_2^2\left(\|e(t_2)-e(t_1)\|_2^2+\|\alpha(t_2)-\alpha(t_1)\|_4^2+\|\nabla\alpha(t_2)-\nabla\alpha(t_1)\|_2^2\right)\nonumber\\
&\quad+3C_2^2\left(\tilde Q(\alpha(t_1),p(t_2)-p(t_1))+\|\nabla p(t_2)-\nabla p(t_1)\|_2^2+\left(\int_{t_1}^{t_2}\|\EE \dot w(s)\|_2\,\de s\right)^2\right)\nonumber\\
&\le C_3\|\alpha(t_2)-\alpha(t_1)\|_4^2+C_3\|\alpha(t_1)-\alpha(t_2)\|_4\|p(t_2)-p(t_1)\|_2\nonumber\\
&\quad+C_3\int_{t_1}^{t_2}\|e(s)-e(t_1)\|_2\|\EE \dot w(s)\|_2\,\de s+C_3\left(\int_{t_1}^{t_2}\|\EE \dot w(s)\|_2\,\de s\right)^2.\label{eq:part2}
\end{align}
By summing~\eqref{eq:part1} and~\eqref{eq:part2} and using also Corollary~\ref{coro:H1norm}, we can find a constant $C_4>0$ such that
\begin{align*}
&\|e(t_2)-e(t_1)\|_2^2+\|p(t_2)-p(t_1)\|_{H^1}^2+\|\alpha(t_2)-\alpha(t_1)\|_{H^1}^2\\
&\le C_4\|\alpha(t_2)-\alpha(t_1)\|_4^2+C_4\|\alpha(t_1)-\alpha(t_2)\|_4\|p(t_2)-p(t_1)\|_2\\
&\quad+C_4\int_{t_1}^{t_2}\|e(s)-e(t_1)\|_2\|\EE \dot w(s)\|_2\,\de s+C_4\left(\int_{t_1}^{t_2}\|\EE \dot w(s)\|_2\,\de s\right)^2.
\end{align*}
Firstly, we use Young Inequality
$$\|\alpha(t_1)-\alpha(t_2)\|_4\|p(t_2)-p(t_1)\|_2\le C_\epsilon\|\alpha(t_1)-\alpha(t_2)\|_4^2+\epsilon\|p(t_2)-p(t_1)\|_{H^1}^2\quad\text{for all $\epsilon>0$},$$
and then Lemma~\ref{lem:Cdelta} with $\theta=4$ to derive that
\begin{align*}
&\|e(t_2)-e(t_1)\|_2^2+\|p(t_2)-p(t_1)\|_{H^1}^2+\|\alpha(t_2)-\alpha(t_1)\|_{H^1}^2\\
&\le C_5\|\alpha(t_2)-\alpha(t_1)\|_1^2+C_5\int_{t_1}^{t_2}\|e(s)-e(t_1)\|_2\|\EE \dot w(s)\|_2\,\de s+C_5\left(\int_{t_1}^{t_2}\|\EE \dot w(s)\|_2\,\de s\right)^2
\end{align*}
for a constant $C_5>0$. In particular, there exist a constant $C>0$ such that
\begin{align}
&\left(\|e(t_2)-e(t_1)\|_2+\|p(t_2)-p(t_1)\|_{H^1}+\|\alpha(t_2)-\alpha(t_1)\|_{H^1}\right)^2\nonumber\\
&\le C^2\|\alpha(t_2)-\alpha(t_1)\|_1^2+C\int_{t_1}^{t_2}\|e(s)-e(t_1)\|_2\|\EE \dot w(s)\|_2\,\de s+\left(C\int_{t_1}^{t_2}\|\EE \dot w(s)\|_2\,\de s\right)^2.\label{eq:continuity}
\end{align}
Let us define the functions
\begin{align*}
&f(t)\coloneqq\|e(t)-e(t_1)\|_2+\|p(t)-p(t_1)\|_{H^1}+\|\alpha(t)-\alpha(t_1)\|_{H^1}\quad\text{for all $t\in[t_1,T]$},\\
&g(t)\coloneqq C\|\alpha(t)-\alpha(t_1)\|_1\quad\text{for all $t\in[t_1,T]$},\qquad h(t)\coloneqq C\|\EE \dot w(t)\|_2\quad\text{for a.e.\ $t\in[t_1,T]$}.
\end{align*}
Therefore, by~\eqref{eq:continuity} we deduce
\begin{align*}
f(t)^2\le g(t)^2+\int_{t_1}^tf(s)h(s)\,\de s+\left(\int_{t_1}^th(s)\,\de s\right)\quad\text{for all $t\in[t_1,T]$}.
\end{align*}
We can apply Lemma~\ref{lem:gronw} to deduce~\eqref{eq:cont1}. Finally, by Korn Inequality we have
$$\|u(t_2)-u(t_1)\|_2\le C_6\|e(t_2)-e(t_1)\|_2+C_6\|p(t_2)-p(t_1)\|_2+C_6\|w(t_2)-w(t_1)\|_{H^1}$$
for a constant $C_6>0$. Hence, by~\eqref{eq:cont1} we derive~\eqref{eq:cont2}.

\end{proof}

%--------------------------------
% Approximate viscous evolutions      
%--------------------------------

\section{Approximate viscous evolutions}\label{sec:ap_ve}

%We now pass to the second result of this paper, which is the proof of balance viscosity quasistatic solution. To this aim, we first need 
The present section is devoted to prove the existence of a quasistatic evolution for a viscous regularization of our elastoplastic-damage model, driven by a small parameter $\varepsilon\in(0,1)$ (see Definition~\ref{def:ep_QSE}). As in the previous section, this is done by using a time discretization scheme. 

\subsection{Discretization in time}

Let $\varepsilon\in(0,1)$ be fixed. As in Section~\ref{sec:gl_qe} we consider an initial configuration $(\alpha^0,u^0,e^0,p^0)\in H^1(\Omega;[0,1])\times\mathcal A(w(0))$ which satisfies~\eqref{eq:ic}.

For all $k\in\N$ we define
\begin{equation*}
\tau_k\coloneqq \frac{T}{k},\quad t_k^i\coloneqq i\tau_k,\quad w_k^i\coloneqq w(t_k^i)\quad\text{for $i=0,\dots,k$}.
\end{equation*}
Starting from
\begin{equation*}
(\alpha_k^0,u_k^0,e_k^0,p_k^0)\coloneqq (\alpha^0,u^0,e^0,p^0)\in H^1(\Omega;[0,1])\times \mathcal A(w(0)),
\end{equation*}
for all $i=1,\dots,k$ we define
\begin{equation*}
(\alpha_k^i,u_k^i,e_k^i,p_k^i)\in \mathcal D(\alpha_k^{i-1})\times \mathcal A(w_k^i)
\end{equation*}
as the solution of the minimum problem
\begin{equation}\label{eq:min_ki}
\min_{(\beta,v,\eta,q)\in \mathcal D(\alpha_k^{i-1})\times \mathcal A(w_k^i)}\left\{\mathcal E(\beta,\eta,q)+\mathcal H(q-p_k^{i-1})+\frac{\varepsilon}{2\tau_k}\|\beta-\alpha_k^{i-1}\|_2^2\right\}.
\end{equation}
As in Section~\ref{sec:gl_qe}, the minimum problem~\eqref{eq:min_ki} admits a solution $(\alpha_k^i,u_k^i,e_k^i,p_k^i)\in\mathcal D(\alpha_k^{i-1})\times \mathcal A(w_k^i)$ for all $i=1,\dots,k$. Moreover, it satisfies the following properties.

\begin{lemma}\label{lem:ki_stab}
Assume~\eqref{eq:d}--\eqref{eq:B3},~\eqref{eq:K}, and~\eqref{eq:ic}. The following hold for all $k\in\N$ and $i=0\dots,k$.
\begin{itemize}
\item[(a)] For all $(v,\eta,q)\in \mathcal A(w_k^i)$
\begin{align*}
\mathcal E(\alpha_k^i,e_k^i,p_k^i)\le \mathcal E(\alpha_k^i,\eta,q)+\mathcal H(q-p_k^i).
\end{align*}
\item[(b)] For all $(v,\eta,q)\in \mathcal A(0)$
\begin{equation*}
-\mathcal H(q)\le (\C e_k^i,\eta)_2+2(\B(\alpha_k^i)p_k^i,q)_2+2(\nabla p_k^i,\nabla q)_2\le \mathcal H(-q).
\end{equation*}
In particular
\begin{align}\label{eq:divCeki}
\div\C e_k^i=0\quad\text{in $\mathcal D'(\Omega;\R^n)$}.
\end{align}
\item[(c)] For all $(v,\eta,q)\in \mathcal A(w_k^i)$
\begin{align*}
&\mathcal Q(e_k^i)+\mathcal Q(\eta-e_k^i)+\tilde Q(\alpha_k^i,p_k^i)+\tilde Q(\alpha_k^i,q-p_k^i)+\|\nabla p_k^i\|_2^2+\|\nabla q-\nabla p_k^i\|_2^2\\
&\le \mathcal Q(\eta)+\tilde Q(\alpha_k^i,q)+\|\nabla q\|_2^2+\mathcal H(q-p_k^i).
\end{align*}
\end{itemize}
\end{lemma}

\begin{proof}
(a) For $i=0$ this is true by~\eqref{eq:ic}, while for $i=1,\dots,k$ and $(v,\eta,q)\in \mathcal A(w_k^i)$ it is enough to use~\eqref{eq:min_ki} together with
\begin{equation*}
\mathcal H(q-p_k^{i-1})-\mathcal H(p_k^i-p_k^{i-1})\le \mathcal H(q-p_k^i).
\end{equation*}

(b) For all $(v,\eta,q)\in \mathcal A(0)$ and $\lambda>0$ we text in (a) with  $(u_k^i+\lambda v,e_k^i+\lambda \eta,p_k^i+\lambda q)\in \mathcal A(w_k^i)$, and we get
\begin{align*}
\mathcal Q(e_k^i)+\tilde Q(\alpha_k^i,p_k^i)+\|\nabla p_k^i\|_2^2\le \mathcal Q(e_k^i-\lambda \eta)+\tilde Q(\alpha_k^i,p_k^i-\lambda q)+\|\nabla (p_k^i-\lambda q)\|_2^2+\lambda\mathcal H(-q).
\end{align*}
By dividing for $\lambda>0$ and sending $\lambda\to 0^+$ we get one inequality of (b). To get the other inequality, it is enough to replace $(v,\eta,q)\in\mathcal A(0)$ with $(-v,-\eta,-q)\in\mathcal A(0)$. 

Let now $\phi\in C_c^\infty(\Omega)$. Then $(\phi,\EE\phi,0)\in \mathcal A(0)$ and by using (b) we get
$$(\C e_k^i,\EE\phi)_2=0\quad\text{for all $\phi\in C_c^\infty(\Omega;\R^n)$},$$
which gives~\eqref{eq:divCeki}.

(c) Thanks to (b), for all $(v,\eta,q)\in \mathcal A(w_k^i)$ we get
\begin{align*}
&\mathcal Q(e_k^i)+\mathcal Q(\eta-e_k^i)-\mathcal Q(\eta)+\tilde Q(\alpha_k^i,p_k^i)+\tilde Q(\alpha_k^i,q-p_k^i)-\tilde Q(\alpha_k^i,q)\\
&\quad+\|\nabla p_k^i\|_2^2+\|\nabla q-\nabla p_k^i\|_2^2-\|\nabla q\|_2^2\\
&=(\C e_k^i,e_k^i-\eta)_2+2(\B(\alpha_k^i)p_k^i,p_k^i-q)_2+2(\nabla p_k^i,\nabla p_k^i-\nabla q)_2\le \mathcal H(q-p_k^i),
\end{align*}
which implies the desired inequality.
\end{proof}

For all $k\in\N$ and $i=1,\dots,k$ we define
\begin{equation*}\dot\alpha_k^i\coloneqq \frac{\alpha_k^i-\alpha_k^{i-1}}{\tau_k},\quad \dot u_k^i\coloneqq \frac{u_k^i-u_k^{i-1}}{\tau_k},\quad \dot e_k^i\coloneqq \frac{e_k^i-e_k^{i-1}}{\tau_k},\quad \dot p_k^i\coloneqq \frac{ p_k^i-p_k^{i-1}}{\tau_k},\quad \dot w_k^i\coloneqq \frac{w_k^i-w_k^{i-1}}{\tau_k}.\end{equation*}
By arguing as in Lemma~\ref{lem:glslkey}, we can derive the following discrete energy estimates.

\begin{lemma}
Assume~\eqref{eq:d}--\eqref{eq:B3},~\eqref{eq:K}, and~\eqref{eq:ic}. For all $k\in \N$ and $i=1,\dots,k$ we have
\begin{equation*}
\mathcal E(\alpha_k^i,e_k^i,p_k^i)+\sum_{j=1}^i\tau_k\mathcal H(\dot p_k^j)+\frac{\varepsilon}{2}\sum_{j=1}^i\tau_k\|\dot\alpha_k^j\|_2^2\le \mathcal E(\alpha^0,e^0,p^0)+\sum_{j=1}^i\tau_k(\C e_k^{j-1},\EE \dot w_k^j)_2+\delta_k,
\end{equation*}
where $\delta_k\to 0$ as $k\to\infty$ is defined as in~\eqref{eq:deltak}. In particular, there exists a constant $C>0$ independent on $\varepsilon,k,i$ such that
\begin{align}\label{eq:unif-est-eps}
\max_{i=0,\dots,k}\|e_k^i\|_2+\max_{i=0,\dots,k}\|p_k^i\|_{H^1}+\max_{i=0,\dots,k}\|\alpha_k^i\|_{H^1}+\sum_{i=1}^k\tau_k\mathcal H(\dot p_k^i)+\varepsilon \sum_{i=1}^k\tau_k\|\dot \alpha_k^i\|_2^2\le C.
\end{align}
\end{lemma}

\begin{proof}
It is enough to proceed as in the proof of Lemma~\ref{lem:glslkey}.
\end{proof}

We now show a variational inequality for the quadruple $(\alpha_k^i,u_k^i,e_k^i,p_k^i)$, similarly to Lemma~\ref{lem:qs_var_ineq}-(a).

\begin{lemma}
Assume~\eqref{eq:d}--\eqref{eq:B3},~\eqref{eq:K}, and~\eqref{eq:ic}. For all $i=1\dots,k$ we have
\begin{equation}\label{eq:st_aki}
\partial_\alpha \mathcal E(\alpha_k^i,e_k^i,p_k^i)[\beta]+\varepsilon(\dot\alpha_k^i,\beta)_2\ge 0\quad\text{for all $\beta\in H^1(\Omega)$ with $\beta\le 0$}.
\end{equation}
Moreover
\begin{equation}\label{eq:st_daki}
\partial_\alpha \mathcal E(\alpha_k^i,e_k^i,p_k^i)[\dot\alpha_k^i]+\varepsilon\|\dot\alpha_k^i\|_2^2=0.
\end{equation}
\end{lemma}

\begin{proof}
We extend $d$ and $\mathbb B$ to $(-\infty,0)$ as in~\eqref{eq:dB_ext}. We fix $k\in\N$ and $i=1,\dots,k$. Let $\beta\in H^1(\Omega)$ with $\beta\le 0$  and let $\lambda>0$. If we consider
$(\alpha_k^i+\lambda\beta)^+\in \mathcal D(\alpha_k^{i-1})$,  by~\eqref{eq:min_ki} we get
\begin{align*}
\mathcal E(\alpha_k^i,e_k^i,p_k^i)+\frac{\varepsilon}{2\tau_k}\|\alpha_k^i-\alpha_k^{i-1}\|_2^2&\le \mathcal E((\alpha_k^i+\lambda\beta)^+,e_k^i,p_k^i)+\frac{\varepsilon}{2\tau_k}\|(\alpha_k^i+\lambda\beta)^+-\alpha_k^{i-1}\|_2^2.
\end{align*}
Moreover, by~\eqref{eq:d},~\eqref{eq:B1}, and the fact that
$\alpha_k^{i-1}\ge 0$ we have
\begin{align*}
&D((\alpha_k^i+\lambda\beta)^+)\le D(\alpha_k^i+\lambda\beta)& &\|\nabla((\alpha_k^i+\lambda\beta)^+)\|^2_2\le \|\nabla(\alpha_k^i+\lambda\beta)\|_2^2,\\
&\tilde Q((\alpha_k^i+\lambda \beta)^+,p_k^i)\le \tilde Q(\alpha_k^i+\lambda \beta,p_k^i)  & &\|(\alpha_k^i+\lambda\beta)^+-\alpha_k^{i-1}\|_2^2\le \|\alpha_k^i+\lambda\beta-\alpha_k^{i-1}\|_2^2.
\end{align*}
Hence
\begin{align*}
\mathcal E(\alpha_k^i,e_k^i,p_k^i)+\frac{\varepsilon}{2\tau_k}\|\alpha_k^i-\alpha_k^{i-1}\|_2^2&\le \mathcal E(\alpha_k^i+\lambda\beta,e_k^i,p_k^i)+\frac{\varepsilon}{2\tau_k}\|\alpha_k^i+\lambda\beta-\alpha_k^{i-1}\|_2^2.
\end{align*}
By dividing for $\lambda>0$, sending $\lambda\to 0^+$, and using~\eqref{eq:par_der} we get~\eqref{eq:st_aki}.

If we use $\beta=\dot\alpha_k^i$ in~\eqref{eq:st_aki} we get the first inequality of~\eqref{eq:st_daki}. To get the other inequality, we fix $\lambda\in (0,\tau_k)$ and we consider $\alpha_k^i-\lambda\dot\alpha_k^i\in\mathcal D(\alpha_k^{i-1})$. Hence, thanks~\eqref{eq:min_ki} we have
\begin{equation*}\mathcal E(\alpha_k^i,e_k^i,p_k^i)+\frac{\varepsilon}{2\tau_k}\|\alpha_k^i-\alpha_k^{i-1}\|_2^2\le \mathcal E(\alpha_k^i-\lambda\dot\alpha_k^i,e_k^i,p_k^i)+\frac{\varepsilon}{2\tau_k}\|\alpha_k^i-\alpha_k^{i-1}-\lambda\dot\alpha_k^i\|_2^2.\end{equation*}
By dividing for $\lambda>0$, sending $\lambda\to 0^+$, and using~\eqref{eq:par_der} we derive~\eqref{eq:st_daki}.
\end{proof}

\begin{remark}
As in Section~\ref{sec:gl_qe}, the assumptions $\dot d(0)\le 0$ and $\dot{\B}(0)\xi:\xi\le 0$ for all $\xi\in\M$ in~\eqref{eq:d} and~\eqref{eq:B1} are needed to prove that~\eqref{eq:st_aki} holds.
\end{remark}

\begin{remark}
Arguing as before, by~\eqref{eq:ic} for all $\beta\in H^1(\Omega)$ with $\beta\le 0$ we deduce
\begin{equation}\label{eq:st_ak0}
\partial_\alpha\mathcal E(\alpha^0,e^0,p^0)[\beta]\ge 0.
\end{equation}
\end{remark}

In the same spirit of Theorem~\ref{thm:cont_time}, in the following lemma we show that we can estimate the norms of $\dot u_k$, $\dot e_k$, $\dot p_k$  by the norm of  $\dot \alpha_k$ and $\dot w_k$ times a constant independent on $k$ and $\varepsilon$. The proof follows the one of~\cite[Lemma 3.6]{CrLa}.

\begin{lemma}\label{lem:ki_est}
Assume~\eqref{eq:d}--\eqref{eq:B3},~\eqref{eq:K}, and~\eqref{eq:ic}. Define for all $k\in\N$ and $i=1\,\dots,k$
\begin{equation*}
\omega_k^i\coloneqq \|\alpha_k^i-\alpha_k^{i-1}\|_2+\|\EE w_k^i-\EE w_k^{i-1}\|_2.
\end{equation*}
There exists a constant $C>0$, independent on $\varepsilon,k,i$, such that for all $i=1,\dots,k$ and $k\in\N$
\begin{equation}\label{eq:ekikiwki}
\|e_k^i-e_k^{i-1}\|_2+\|p_k^i-p_k^{i-1}\|_{H^1}\le C\omega_k^i.
\end{equation}
In particular, for all $i=1,\dots,k$ and $k\in\N$ we have
\begin{equation}\label{eq:ukiwki}
\|u_k^i-u_k^{i-1}\|_{H^1}\le C(\|w_k^i-w_k^{i-1}\|_2+\omega_k^i).
\end{equation}
Moreover, for all $i=1,\dots,k$ and $k\in\N$
\begin{equation}\label{eq:Hdotpk}
\mathcal H(\dot p_k)\le (\C e_k^i,\dot p_k)_2-2(\B(\alpha_k^i)p_k^i,\dot p_k^i)-2(\nabla p_k^i,\nabla\dot p_k^i)_2+C\tau_k\left(\|\dot\alpha_k^i\|_2^2+\|\EE \dot w_k^i\|_2^2\right).
\end{equation}
\end{lemma}

\begin{proof}
We fix $k\in\N$ and $i=1,\dots,k$. If we test Lemma~\ref{lem:ki_stab}-(c) in $i-1$ with
\begin{equation*}
(u_k^i-(w_k^i-w_k^{i-1}),e_k^i-\EE(w_k^i-w_k^{i-1}),p_k^i)\in \mathcal A(w_k^{i-1}),
\end{equation*}
we get
\begin{align*}
&\mathcal Q(e_k^{i-1})+\mathcal Q(e_k^i-e_k^{i-1}-\EE(w_k^i-w_k^{i-1}))\\
&\quad+\tilde Q(\alpha_k^{i-1},p_k^{i-1})+\tilde Q(\alpha_k^{i-1},p_k^i-p_k^{i-1})+\|\nabla p_k^{i-1}\|_2^2+\|\nabla p_k^i-\nabla p_k^{i-1}\|_2^2\\
&\le \mathcal Q(e_k^i-\EE(w_k^i-w_k^{i-1}))+\tilde Q(\alpha_k^{i-1},p_k^i)+\|\nabla p_k^i\|_2^2+\mathcal H(p_k^i-p_k^{i-1}).
\end{align*}
In particular
\begin{align}
&\mathcal Q(e_k^i-e_k^{i-1})+\tilde Q(\alpha_k^{i-1},p_k^i-p_k^{i-1})+\|\nabla p_k^i-\nabla p_k^{i-1}\|_2^2\nonumber\\
&\le \mathcal Q(e_k^i)-\mathcal Q(e_k^{i-1})-(\C e_k^{i-1},\EE(w_k^i-w_k^{i-1}))\nonumber\\
&\quad+\tilde Q(\alpha_k^{i-1},p_k^i)-\tilde Q(\alpha_k^{i-1},p_k^{i-1})+\|\nabla p_k^i\|_2^2-\|\nabla p_k^{i-1}\|_2^2+\mathcal H(p_k^i-p_k^{i-1}).\label{eq:QQpk}
\end{align}
We now use the minimality of $(\alpha_k^i,u_k^i,e_k^i,p_k^i)\in \mathcal D(\alpha_k^{i-1})\times \mathcal A(w_k^i)$ with
\begin{equation*}
(\alpha_k^i,u_k^{i-1}+(w_k^i-w_k^{i-1}),e_k^{i-1}+\EE(w_k^i-w_k^{i-1}),p_k^{i-1})\in \mathcal D(\alpha_k^{i-1})\times \mathcal A(w_k^i),
\end{equation*}
and we get
\begin{align*}
&\mathcal Q(e_k^i)+\tilde Q(\alpha_k^i,p_k^i)+\|\nabla p_k^i\|_2^2+\mathcal H(p_k^i-p_k^{i-1})\\
&\le \mathcal Q(e_k^{i-1}+\EE(w_k^i-w_k^{i-1}))+\tilde Q(\alpha_k^i,p_k^{i-1})+\|\nabla p_k^{i-1}\|_2^2.
\end{align*}
Hence
\begin{align}
&\mathcal H(p_k^i-p_k^{i-1})\nonumber\\
&\le \mathcal Q(e_k^{i-1}+\EE(w_k^i-w_k^{i-1}))-\mathcal Q(e_k^i)+\tilde Q(\alpha_k^i,p_k^{i-1})-\tilde Q(\alpha_k^i,p_k^i)+\|\nabla p_k^{i-1}\|_2^2-\|\nabla p_k^i\|_2^2.\label{eq:Hpk}
\end{align}
By combining together~\eqref{eq:QQpk} and~\eqref{eq:Hpk}, we deduce the existence of a constant $C_1>0$, independent on $\varepsilon,k,i$, such that
\begin{align}
&\mathcal Q(e_k^i-e_k^{i-1})+\tilde Q(\alpha_k^{i-1},p_k^i-p_k^{i-1})+\|\nabla p_k^i-\nabla p_k^{i-1}\|_2^2\nonumber\\
&\le \mathcal Q(\EE(w_k^i-w_k^{i-1}))+\tilde Q(\alpha_k^{i-1},p_k^i)-\tilde Q(\alpha_k^i,p_k^i)-\tilde Q(\alpha_k^{i-1},p_k^{i-1})+\tilde Q(\alpha_k^i,p_k^{i-1})\nonumber\\
&\le \frac{\gamma_2}{2}\|\EE w_k^i-\EE w_k^{i-1}\|_2^2-([\B(\alpha_k^i)-\B(\alpha_k^{i-1})](p_k^i-p_k^{i-1}),p_k^i+p_k^{i-1})_2\nonumber\\
&\le C_1\left(\|\EE w_k^i-\EE w_k^{i-1}\|_2^2+\|\alpha_k^i-\alpha_k^{i-1}\|_2\|p_k^i-p_k^{i-1}\|_{H^1}\right).\label{eq:Qeps11}
\end{align}
Moreover, by Remarks~\ref{rem:BCS} and~\ref{rem:emb} and the uniform estimate~\eqref{eq:unif-est-eps}, there is a constant $C_2>0$, independent on $\varepsilon,k,i$, such that
\begin{align*}
r_H\|p_k^i-p_k^{i-1}\|_1&\le \mathcal H(p_k^i-p_k^{i-1})\\
&\le \mathcal Q(e_k^i)-\mathcal Q(e_k^{i-1})+\mathcal Q(\EE(w_k^i-w_k^{i-1}))+(\C e_k^{i-1},\EE(w_k^i-w_k^{i-1}))_2\\
&\quad+\tilde Q(\alpha_k^i,p_k^{i-1})-\tilde Q(\alpha_k^{i-1},p_k^{i-1})+\tilde Q(\alpha_k^{i-1},p_k^i)-\tilde Q(\alpha_k^i,p_k^i)\\
&\quad+\tilde Q(\alpha_k^{i-1},p_k^{i-1})-\tilde Q(\alpha_k^{i-1},p_k^i)+\|\nabla p_k^{i-1}\|_2^2-\|\nabla p_k^i\|_2^2\\
&\le C_2\left(\sqrt{\mathcal Q(e_k^i-e_k^{i-1})}+\sqrt{\tilde Q(\alpha_k^{i-1},p_k^i-p_k^{i-1})}+\|\nabla p_k^i-\nabla p_k^{i-1}\|_2\right)\\
&\quad+C_2(\|\EE(w_k^i-w_k^{i-1})\|_2+\|\alpha_k^i-\alpha_k^{i-1}\|_2).
\end{align*}
Therefore, we have
\begin{align}
&\|p_k^i-p_k^{i-1}\|_1^2\nonumber\\
&\le C_3\left(\mathcal Q(e_k^i-e_k^{i-1})+\tilde Q(\alpha_k^{i-1},p_k^i-p_k^{i-1})+\|\nabla p_k^i-\nabla p_k^{i-1}\|_2^2+\|\EE(w_k^i-w_k^{i-1})\|_2^2+ \|\alpha_k^i-\alpha_k^{i-1}\|^2_2\right)\nonumber\\
&\le C_4(\|\alpha_k^i-\alpha_k^{i-1}\|_2\|p_k^i-p_k^{i-1}\|_{H^1}+\|\EE w_k^i-\EE w_k^{i-1}\|_2^2+\|\alpha_k^i-\alpha_k^{i-1}\|_2^2),\label{eq:Qeps12} 
\end{align}
for two constants $C_3,C_4>0$, independent on $\varepsilon,k,i$.  If we sum~\eqref{eq:Qeps11} and~\eqref{eq:Qeps12}, and we use~\eqref{eq:C2} and Corollary~\ref{coro:H1norm}, we derive that
\begin{align*}
\|e_k^i-e_k^{i-1}\|_2^2+\|p_k^i-p_k^{i-1}\|_{H^1}^2&\le C_5(\|\alpha_k^i-\alpha_k^{i-1}\|_2\|p_k^i-p_k^{i-1}\|_{H^1}+\|\EE w_k^i-\EE w_k^{i-1}\|_2^2+\|\alpha_k^i-\alpha_k^{i-1}\|_2^2)\\
&\le \frac{1}{2}\|p_k^i-p_k^{i-1}\|_{H^1}^2+C_6(\|\EE w_k^i-\EE w_k^{i-1}\|_2^2+\|\alpha_k^i-\alpha_k^{i-1}\|_2^2),
\end{align*}
for two constants $C_5,C_6>0$, independent on $\varepsilon,k,i$. This gives~\eqref{eq:ekikiwki}. Moreover, to obtain~\eqref{eq:ukiwki} it is enough to combine~\eqref{eq:ekikiwki} with Korn Inequality. 

Finally, by combining~\eqref{eq:Hpk} and~\eqref{eq:int_for}, with~\eqref{eq:ekikiwki} and the identity
$$|a|^2-|b|^2=-2b\cdot(b-a)+|b-a|^2\quad\text{for all $a,b\in\R^d$},$$
we obtain
\begin{align*}
&\mathcal H(p_k^i-p_k^{i-1})\\
&\le (\C e_k^i,\EE(w_k^i-w_k^{i-1})-(e_k^i-e_k^{i-1}))+\mathcal Q(\EE(w_k^i-w_k^{i-1}))\\
&\quad-(\C (e_k^i-e_k^{i-1}),\EE(w_k^i-w_k^{i-1}))_2+\mathcal Q(e_k^i-e_k^{i-1})\\
&\quad -2(\mathbb B(\alpha_k^i)p_k^i,p_k^i-p_k^{i-1})_2+\tilde Q(p_k^i-p_k^{i-1})-2(\nabla p_k^i,\nabla p_k^i-\nabla p_k^{i-1})_2+\|\nabla p_k^i-\nabla p_k^{i-1}\|_2^2\\
&\le (\C e_k^i,p_k^i-p_k^{i-1})_2-2(\mathbb B(\alpha_k^i)p_k^i,p_k^i-p_k^{i-1})_2-2(\nabla p_k^i,\nabla p_k^i-\nabla p_k^{i-1})_2\\
&\quad+C\left(\|\alpha_k^i-\alpha_k^{i-1}\|_2^2+\|\EE w_k^i-\EE w_k^{i-1}\|_2^2\right),
\end{align*}
which gives~\eqref{eq:Hdotpk}.
\end{proof}

\begin{remark}
We can rephrase Lemma~\ref{lem:ki_est} in the following way: there is a constant $C$, independent on $\varepsilon,k,i$, such that for all $k\in\N$ and $i=1,\dots,k$ 
\begin{equation*}\|\dot e_k^i\|_2+\|\dot p_k^i\|_{H^1}\le C(\|\dot \alpha_k^i\|_{H^1}+\|\EE \dot w_k^i\|_2),\qquad \|\dot u_k^i\|_{H^1}\le C(\|\dot \alpha_k^i\|_{H^1}+\|\dot w_k^i\|_{H^1}).
\end{equation*}
\end{remark}

We now estimate, by means of the following two lemmas, the regularity in time of the viscously regularized evolutions, taking value in the target spaces of the internal variables, with respect to the $H^1(0,T)$ norm (for fixed $\varepsilon\in(0,1)$) and the $W^{1,1}(0,T)$ norm (uniform in $\varepsilon\in(0,1)$).

\begin{lemma}\label{lem:w11-est}
Assume~\eqref{eq:d}--\eqref{eq:B3},~\eqref{eq:K}, and~\eqref{eq:ic}. There exists a constant $C>0$, independent on $\varepsilon,k,i$, such that
\begin{align}\label{eq:eps_est}
\|\dot\alpha_k^i\|_2^2+\sum_{j=1}^i\tau_k\|\dot\alpha_k^j\|_{H^1}^2\le \frac{C}{\varepsilon}{\rm e}^{\frac{C}{\varepsilon}t_k^i}\quad\text{for all $k\ge \frac{C}{\varepsilon}$ and $i=1,\dots,k$.}
\end{align}
\end{lemma}

\begin{proof}
We fix $k\in\N$. By the uniform estimate~\eqref{eq:unif-est-eps}, the identity~\eqref{eq:st_daki} with $i=1$, the variational inequality~\eqref{eq:st_ak0} with $\beta=\dot\alpha_k^1$, the~\eqref{eq:ekikiwki}, and Remark~\ref{rem:emb} we have 
\begin{align}
\varepsilon\|\dot\alpha_k^1\|_2^2+2\tau_k\|\dot\alpha_k^1\|_{H^1}^2&\le\partial_\alpha\mathcal E(\alpha_k^0,e_k^0,p_k^0)[\dot\alpha_k^1]-\partial_\alpha\mathcal E(\alpha_k^1,e_k^1,p_k^1)[\dot\alpha_k^1]+2\tau_k\|\dot\alpha_k^1\|_{H^1}^2\nonumber\\
&=\int_\Omega [\dot{d}(\alpha_k^0(x))-\dot{d}(\alpha_k^1(x))]\dot\alpha_k^1(x) \,\de x+2\tau_k\|\dot\alpha_k^1\|_2^2\nonumber\\
&\quad+2\int_\Omega [\dot{\B}(\alpha_k^0(x))-\dot{\B}(\alpha_k^1(x))]\dot\alpha_k^1(x)p_k^0(x):p_k^0(x) \,\de x\nonumber\\
&\quad-2\int_\Omega \dot{\B}(\alpha_k^1(x))\dot\alpha_k^1(x)(p_k^1(x)-p_k^0(x)):(p_k^1(x)+p_k^0(x)) \,\de x\nonumber\\
&\le C_1\tau_k(\|\dot\alpha_k^1\|_2^2+\|\dot\alpha_k^1\|_4^2+\|\dot\alpha_k^1\|_2\|p_k^1-p_k^0\|_{H^1})\nonumber\\
&\le C_2\tau_k(\|\dot\alpha_k^1\|_4^2+\|\EE \dot w_k^1\|_2^2),\label{eq:ak1-est}
\end{align}
for two constants $C_1,C_2>0$, independent on $\varepsilon,k$. Similarly, for $j=2,\dots,k$ we use~\eqref{eq:st_daki} in $j$ and~\eqref{eq:st_aki} in $j-1$ with $\beta=\dot\alpha_k^j$, and we get
\begin{align}
\varepsilon(\dot\alpha_k^j,\dot\alpha_k^j-\dot\alpha_k^{j-1})_2+2\tau_k\|\dot\alpha_k^j\|_{H^1}^2&\le\partial_\alpha\mathcal E(\alpha_k^{j-1},e_k^{j-1},p_k^{j-1})[\dot\alpha_k^j]-\partial_\alpha\mathcal E(\alpha_k^j,e_k^j,p_k^j)[\dot\alpha_k^j]+2\tau_k\|\dot\alpha_k^j\|_{H^1}^2\nonumber\\
&=\int_\Omega [\dot{d}(\alpha_k^{j-1}(x))-\dot{d}(\alpha_k^j(x))]\dot\alpha_k^j(x) \,\de x+2\tau_k\|\dot\alpha_k^j\|_2^2\nonumber\\
&\quad+2\int_\Omega [\dot{\B}(\alpha_k^{j-1}(x))-\dot{\B}(\alpha_k^j(x))]\dot\alpha_k^j(x)p_k^{j-1}(x):p_k^{j-1}(x) \,\de x\nonumber\\
&\quad-2\int_\Omega \dot{\B}(\alpha_k^j(x))\dot\alpha_k^j(x)(p_k^j(x)-p_k^{j-1}(x)):(p_k^j(x)+p_k^{j-1}(x)) \,\de x\nonumber\\
&\le C_3\tau_k(\|\dot\alpha_k^j\|_2^2+\|\dot\alpha_k^j\|_4^2+\|\dot\alpha_k^j\|_2\|p_k^j-p_k^{j-1}\|_{H^1})\nonumber\\
&\le C_4\tau_k(\|\dot\alpha_k^j\|_4^2+\|\EE \dot w_k^j\|_2^2),\label{eq:dotak-est}
\end{align}
for two constants $C_3,C_4>0$, independent on $\varepsilon,k,j$. We now use the estimate
\begin{equation*}\varepsilon(\dot\alpha_k^j,\dot\alpha_k^j-\dot\alpha_k^{j-1})_2\ge \frac{\varepsilon}{2}\|\dot\alpha_k^j\|_2^2-\frac{\varepsilon}{2}\|\dot\alpha_k^{j-1}\|_2^2\quad\text{for all $i=2,\dots,k$},\end{equation*}
and we sum the two above inequalities over $j=1,\dots,i$ to derive
\begin{align*}
\varepsilon\|\dot\alpha_k^i\|_2^2+\sum_{j=1}^i\tau_k\|\dot\alpha_k^j\|_{H^1}^2\le C_5\left(1+\sum_{j=1}^i\tau_k\|\dot\alpha_k^j\|_4^2\right)\quad\text{ for all $i=1\,\dots,k$},
\end{align*}
for a constant $C_5>0$, independent on $\varepsilon,k,i$. By Lemma~\ref{lem:Cdelta} with $\theta=4$ for all $\lambda>0$ there exists a constant $C_\lambda>0$ such that
\begin{equation*}\|\alpha\|_4^2\le \lambda\|\nabla \alpha\|_2^2+C_\lambda\|\alpha\|_2^2\quad\text{for all $\alpha\in H^1(\Omega)$}.\end{equation*}
Hence, there exists a constant $C_6>0$, independent on $\varepsilon,k,i$, such that
\begin{align}\label{eq:unif-eps-est}
\varepsilon\|\dot\alpha_k^i\|_2^2+\sum_{j=1}^i\tau_k\|\dot\alpha_k^j\|_{H^1}^2\le C_6\left(1+\sum_{j=1}^i\tau_k\|\dot\alpha_k^j\|_2^2\right).
\end{align}
Therefore, by applying the discrete version of Gronwall Inequality of~\cite[Lemma~3.2.4]{AGS}, we derive that
$$\|\dot\alpha_k^i\|_2^2\le \frac{2C_6}{\varepsilon}{\rm e}^{\frac{2C_6}{\varepsilon}t_k^i}\quad\text{for all $k>\frac{2TC_6}{\varepsilon}$ and $i=1,\dots,k$}.$$
Finally, if we combine the above inequality with~\eqref{eq:unif-eps-est} we obtain~\eqref{eq:eps_est}.
\end{proof}

\begin{lemma}\label{lem:eps_est}
Assume~\eqref{eq:d}--\eqref{eq:B3},~\eqref{eq:K}, and~\eqref{eq:ic}. There exists a constant $C>0$, independent on $\varepsilon,k,i$, such that 
\begin{align}\label{eq:H1eps_est}
\sum_{i=1}^k\tau_k\|\dot\alpha_k^i\|_{H^1}\le C\quad\text{for all $k\in \N$}.
\end{align}
\end{lemma}

\begin{proof}
We fix $k\in\N$. If we set 
$\dot \alpha_k^0\coloneqq0$, by~\eqref{eq:ak1-est},~\eqref{eq:dotak-est}, and the inequality
\begin{equation*}
\varepsilon(\dot\alpha_k^j,\dot\alpha_k^j-\dot\alpha_k^{j-1})_2\ge \varepsilon\|\dot\alpha_k^j\|_2(\|\dot\alpha_k^j\|_2-\|\dot\alpha_k^{j-1}\|_2)\quad\text{for all $j=1,\dots,k$},
\end{equation*}
we can find a constant $C_1>0$, independent on $\varepsilon,k,j$, such that 
\begin{align*}
\varepsilon\|\dot\alpha_k^j\|_2(\|\dot\alpha_k^j\|_2-\|\dot\alpha_k^{j-1}\|_2)+2\tau_k\|\dot\alpha_k^j\|_{H^1}^2\le C_1\tau_k(\|\dot\alpha_k^j\|_4^2+\|\EE \dot w_k^j\|_2^2)\quad\text{for all $j=1,\dots,k$}.
\end{align*}
By Lemma~\ref{lem:Cdelta} with $\theta=4$, for all $\lambda>0$ there exist $C_\lambda,\tilde C_\lambda>0$ such that
\begin{equation*}
\|\alpha\|_4^2\le \lambda\|\nabla \alpha\|_2^2+ C_\lambda\|\alpha\|_1^2\le \lambda\|\nabla \alpha\|_2^2+\tilde C_\lambda\|\alpha\|_1\|\alpha\|_2\ \quad\text{for all $\alpha\in H^1(\Omega)$}.
\end{equation*}
Hence, we deduce that
\begin{align*}
\varepsilon\|\dot\alpha_k^j\|_2(\|\dot\alpha_k^j\|_2-\|\dot\alpha_k^{j-1}\|_2)+\tau_k\|\dot\alpha_k^j\|_{H^1}^2\le C_2\tau_k(\|\dot\alpha_k^j\|_1\|\dot\alpha_k^j\|_2+\|\EE \dot w_k^j\|_2^2)\quad\text{for all $j=1,\dots,k$},
\end{align*}
for a constant $C_2>0$, independent on $\varepsilon,k,j$. By multiplying the above inequality by $\frac{2}{\varepsilon}$ and taking into account that $\|\dot\alpha_k^j\|_{H^1}\ge \|\dot\alpha_k^j\|_2$, we derive that
\begin{align*}
2\|\dot\alpha_k^j\|_2(\|\dot\alpha_k^j\|_2-\|\dot\alpha_k^{j-1}\|_2)+\frac{\tau_k}{\varepsilon}\|\dot\alpha_k^j\|_2^2+\frac{\tau_k}{\varepsilon}\|\dot\alpha_k^j\|_{H^1}^2\le \frac{2C_2\tau_k}{\varepsilon}(\|\dot\alpha_k^j\|_1\|\dot\alpha_k^j\|_2+\|\EE \dot w_k^j\|_2^2)\quad\text{for all $j=1,\dots,k$}.
\end{align*}
From now on, the proof follows closely the proof of~\cite[Proposition 3.8]{CrLa}, employing in particular the discrete Gronwall Inequality with weights from~\cite[Lemma 4.1]{KRZ}. We detail all the passages below, for reader’s convenience.

Let us set $a_0\coloneqq\|\dot\alpha_k^0\|_2=0$ and for $j=1,\dots,k$
\begin{align*}
\zeta\coloneqq \frac{\tau_k}{2\varepsilon},\qquad a_j\coloneqq \|\dot\alpha_k^j\|_2,\qquad b_j\coloneqq \sqrt{\frac{\tau_k}{\varepsilon}}\|\dot\alpha_k^j\|_{H^1},\qquad c_j\coloneqq \sqrt{\frac{2C_2\tau_k}{\varepsilon}}\|\EE \dot w_k^j\|_2,\qquad d_j\coloneqq \frac{C_2\tau_k}{\varepsilon}\|\dot\alpha_k^j\|_1.
\end{align*}
Therefore, we can rephrase the above inequality as
\begin{align*}
2a_j(a_j-a_{j-1})+2\zeta a_j^2+b_j^2\le c_j^2+2a_jd_j\quad\text{for all $j=1,\dots,k$}.
\end{align*}
Hence, we can apply~\cite[Lemma 4.1]{KRZ} and we use that $a_0=0$ to deduce
\begin{align*}
\sum_{h=1}^j(1+\zeta)^{2(h-j)-1}b_h^2\le 2\sum_{h=1}^j(1+\zeta)^{2(h-j)-1}c_h^2+4\left(\sum_{h=1}^j(1+\zeta)^{h-j-1}d_h\right)^2\quad\text{for all $j=1,\dots,k$}.
\end{align*}
In particular, for all $j=1,\dots,k$ we derive
\begin{align*}
&2\zeta\sum_{h=1}^j(1+\zeta)^{2(h-j)-1}\|\dot\alpha_k^h\|_{H^1}^2\le8C_2\zeta\sum_{h=1}^j(1+\zeta)^{2(h-j)-1}\|\EE \dot w_k^h\|_2^2+\left(4C_2\zeta\sum_{h=1}^j(1+\zeta)^{h-j-1}\|\dot\alpha_k^h\|_1\right)^2.
\end{align*}
We now use the inequality
\begin{equation*}a^2+b^2\le (a+b)^2\le (1+a^2+b)^2\quad\text{for all $a,b\ge 0$},
\end{equation*}
to estimate from above the right-hand side by
\begin{align*}
\left(1+8C_2\zeta\sum_{h=1}^j(1+\zeta)^{2(h-j)-1}\|\EE \dot w_k^h\|_2^2+4C_2\zeta\sum_{h=1}^j(1+\zeta)^{h-j-1}\|\dot\alpha_k^h\|_1\right)^2.
\end{align*}
Moreover, for all $j=1,\dots,k$ we have
\begin{align*}
\zeta\sum_{h=1}^j(1+\zeta)^{2(h-j)-1}\|\dot\alpha_k^h\|_{H^1}&\le \left(\zeta\sum_{h=1}^j(1+\zeta)^{2(h-j)-1}\right)^{\frac{1}{2}}\left(\zeta\sum_{h=1}^j(1+\zeta)^{2(h-j)-1}\|\dot\alpha_k^h\|_{H^1}^2\right)^{\frac{1}{2}}\\
&\le \left(\zeta\sum_{h=1}^j(1+\zeta)^{2(h-j)-1}\|\dot\alpha_k^h\|_{H^1}^2\right)^{\frac{1}{2}},
\end{align*}
since
\begin{align}
\zeta\sum_{h=1}^j(1+\zeta)^{2(h-j)-1}=\frac{\zeta}{1+\zeta}\sum_{h=0}^{j-1}(1+\zeta)^{-2h}= \frac{\zeta(1+\zeta)(1-(1+\zeta)^{-2j})}{(1+\zeta)^2-1}\le\frac{(1+\zeta)}{2+\zeta}\le 1.\label{eq:sum-zeta}
\end{align}
Therefore, by combining the previous inequalities we deduce that for $j=1,\dots,k$
\begin{align}
&\zeta\sum_{h=1}^j(1+\zeta)^{2(h-j)-1}\|\dot\alpha_k^h\|_{H^1}\le C_3\left(1+\zeta\sum_{h=1}^j(1+\zeta)^{2(h-j)-1}\|\EE \dot w_k^h\|_2^2+\zeta\sum_{h=1}^j(1+\zeta)^{h-j-1}\|\dot\alpha_k^h\|_1\right)\label{eq:aabb3}
\end{align}
for a constant $C_3>0$ independent on $\varepsilon,k,j$. We now multiply both sides by $\tau_k$ and we sum over $j=1,\dots,k$, to obtain
\begin{align}
&\zeta\sum_{j=1}^k\sum_{h=1}^j\tau_k(1+\zeta)^{2(h-j)-1}\|\dot\alpha_k^h\|_{H^1}\nonumber\\
&\le C_3\left(T+\zeta\sum_{j=1}^k\sum_{h=1}^j\tau_k(1+\zeta)^{2(h-j)-1}\|\EE \dot w_k^h\|_2^2+\zeta\sum_{j=1}^k\sum_{h=1}^j\tau_k(1+\zeta)^{h-j-1}\|\dot\alpha_k^h\|_1\right).\label{eq:aabb1}
\end{align}
If we change the order of the sum and we argue as in~\eqref{eq:sum-zeta}, we derive
\begin{align}
&\zeta\sum_{j=1}^k\sum_{h=1}^j\tau_k(1+\zeta)^{2(h-j)-1}\|\EE \dot w_k^h\|_2^2\le\frac{1+\zeta}{2+\zeta}\sum_{h=1}^k\tau_k\|\EE \dot w_k^h\|_2^2\le\sum_{h=1}^k\tau_k\|\EE \dot w_k^h\|_2^2,\\
&\zeta\sum_{j=1}^k\sum_{h=1}^j\tau_k(1+\zeta)^{h-j-1}\|\dot\alpha_k^h\|_1\le\frac{1}{1+\zeta}\sum_{h=1}^k\tau_k\|\dot\alpha_k^h\|_1\le \sum_{h=1}^k\tau_k\|\dot\alpha_k^h\|_1,\\
&\zeta\sum_{j=1}^k\sum_{h=1}^j\tau_k(1+\zeta)^{2(h-j)-1}\|\dot\alpha_k^h\|_{H^1}=\frac{1+\zeta}{2+\zeta}\sum_{h=1}^k\tau_k(1-(1+\zeta)^{2(h-k-1)})\|\dot\alpha_k^h\|_{H^1}.\label{eq:aabb2}
\end{align}
Hence, by combining~\eqref{eq:aabb1}--\eqref{eq:aabb2} we derive
\begin{align*}
\frac{1}{2}\sum_{h=1}^k\tau_k\|\dot\alpha_k^h\|_{H^1}&\le C_3\left(T+\sum_{h=1}^k\tau_k\|\EE \dot w_k^h\|_2^2+\sum_{h=1}^k\tau_k\|\dot\alpha_k^h\|_1\right)+\sum_{h=1}^k\tau_k(1+\zeta)^{2(h-k-1)}\|\dot\alpha_k^h\|_{H^1}.
\end{align*}
Since $\tau_k=2\varepsilon\zeta$, by~\eqref{eq:aabb3} we have
\begin{align*}
\sum_{h=1}^k\tau_k(1+\zeta)^{2(h-k-1)}\|\dot\alpha_k^h\|_{H^1}&=\frac{2\varepsilon}{1+\zeta}\zeta\sum_{h=1}^k(1+\zeta)^{2(h-k)-1}\|\dot\alpha_k^h\|_{H^1}\\
&\le C_3\left(2\varepsilon+\sum_{h=1}^k\tau_k\|\EE \dot w_k^h\|_2^2+\sum_{h=1}^k\tau_k\|\dot\alpha_k^h\|_1\right),
\end{align*}
which gives
\begin{align*}
\sum_{h=1}^k\tau_k\|\dot\alpha_k^h\|_{H^1}&\le C_4\left(1+\sum_{h=1}^k\tau_k\|\EE \dot w_k^h\|_2^2+\sum_{h=1}^k\tau_k\|\dot\alpha_k^h\|_1\right)
\end{align*}
for a constant $C_4>0$ independent on $\varepsilon,i,k$, since $\varepsilon\in(0,1)$. Finally, we observe that 
$$\sum_{h=1}^k\tau_k\|\dot\alpha_k^h\|_1=\sum_{h=1}^k\int_\Omega(\alpha_k^{h-1}(x)-\alpha_k^h(x))\,\de x=\int_\Omega (\alpha_k^0(x)-\alpha_k^k(x))\,\de x\le \mathcal L^n(\Omega).$$
Hence, we obtain~\eqref{eq:H1eps_est}.
\end{proof}

As done in the previous section in~\eqref{eq:overline_ak} and~\eqref{eq:overline_tauk}, we introduce the {\it piecewise constant interpolants} $\overline \alpha_k\colon [0,T]\to H^1(\Omega;[0,1])$, $\overline u_k\colon [0,T]\to H^1(\Omega;\R^n)$, $\overline e_k\colon [0,T]\to L^2(\Omega;\M)$, $\overline p_k\colon [0,T]\to H^1(\Omega;\M)$, $\overline w_k\to H^1(\Omega;\R^n)$, and $\overline \tau_k\colon [0,T]\to [0,T]$. Moreover, we define the {\it piecewise affine interpolant}  $\alpha_k\in H^1((0,T);H^1(\Omega;[0,1]))$ as
\begin{align*}
\alpha_k(t)\coloneqq \alpha_k^{i-1}+(t-t_k^{i-1})\dot \alpha_k^i\quad\text{for $t\in [t_k^{i-1},t_k^i]$, $i=1,\dots,k$},
\end{align*}
and similarly $u_k\in H^1((0,T);H^1(\Omega;\R^n))$, $e_k\in H^1((0,T);L^2(\Omega;\M))$, and $p_k\in H^1((0,T);H^1(\Omega;\M))$. 

By using the piecewise affine interpolants, we can rephrase Lemmas~\ref{lem:w11-est} and~\ref{lem:eps_est} in the following way: there exists a constant $C$, independent on $k$ and $\varepsilon$, such that
\begin{align}
&\int_0^T\left(\|\dot \alpha_k(r)\|_{H^1}^2+\|\dot u_k(r)\|_{H^1}^2+\|\dot e_k(r)\|_2^2+\|\dot p_k(r)\|_ {H^1}^2\right) \,\de r\le \frac{C}{\varepsilon}e^\frac{C}{\varepsilon}& &\text{for all $k\ge \frac{C}{\varepsilon}$},\label{eq:k_unif_bound}\\
&\int_0^T\left(\|\dot \alpha_k(r)\|_{H^1}+\|\dot u_k(r)\|_{H^1}+\|\dot e_k(r)\|_2+\|\dot p_k(r)\|_ {H^1}\right) \,\de r\le C& &\text{for all $k\in\N$}. \label{eq:k_unif_bound2}
\end{align}

%---------------------------------------------
%   Approximate viscous solutions            
%---------------------------------------------

\subsection{Passage to the limit}

In this section we fix $\varepsilon\in (0,1)$ and we shall send $k\to\infty$ to find an $\varepsilon$-approximate viscous evolution according to the definition below.

\begin{definition}\label{def:ep_QSE}
Assume~\eqref{eq:d}--\eqref{eq:B3} and~\eqref{eq:K}. Let $\varepsilon\in(0,1)$ be fixed. A quadruple $(\alpha_\varepsilon,u_\varepsilon,e_\varepsilon,p_\varepsilon)$ from $[0,T]$ into $H^1(\Omega;[0,1])\times\mathcal A$ is an {\it $\varepsilon$-approximate viscous evolution} with datum $w$ if
\begin{align}
&\alpha_\varepsilon\in H^1(0,T;H^1(\Omega;[0,1])),& & u_\varepsilon \in H^1(0,T;H^1(\Omega;\R^n)),\label{eq:ve_2eg1}\\
&e_\varepsilon \in H^1(0,T;L^2(\Omega;\M)),& & p_\varepsilon \in H^1(0,T;H^1(\Omega;\M))\label{eq:ve_2eg2},
\end{align}
and the following conditions are satisfied:
\begin{itemize}
\item[$(ev0)_\varepsilon$] {\it irreversibility}: for a.e.\ $x\in\Omega$ the map
\begin{equation*}
t\mapsto \alpha_\varepsilon(t,x)\quad\text{is non increasing on $[0,T]$};
\end{equation*}
\item[$(ev1)_\varepsilon$] {\it kinematic condition and equilibrium}: for all $t\in[0,T]$ 
\begin{equation*} 
(u_\varepsilon(t),e_\varepsilon(t),p_\varepsilon(t))\in \mathcal A(w(t))\quad\text{and}\quad\div \C e_\varepsilon(t)=0\quad\text{in $\mathcal D'(\Omega;\R^n)$};
\end{equation*}
\item[$(ev2)_\varepsilon$] {\it stress constraint}: for all $t\in[0,T]$
\begin{equation*}
\mathcal H(q)\ge (\C e_\varepsilon(t),q)_2-2(\B(\alpha_\varepsilon(t))p_\varepsilon(t),q)_2-2(\nabla p_\varepsilon(t),\nabla q)_2\quad\text{for all $q\in H^1(\Omega;\M)$};
\end{equation*}
\item[$(ev3)_\varepsilon$] {\it Kuhn–Tucker Inequality}: for a.e.\ $t\in[0,T]$
\begin{equation*}\partial_\alpha\mathcal E(\alpha_\varepsilon(t),e_\varepsilon(t),p_\varepsilon(t))[\beta]+\varepsilon(\dot\alpha_\varepsilon(t),\beta)_2\ge0\quad\text{for all $\beta\in H^1(\Omega)$ with $\beta \le 0$}; \end{equation*}
\item[$(ev4)_\varepsilon$] {\it energy balance}: for all $t\in[0,T]$
\begin{equation*}\mathcal E(\alpha_\varepsilon(t),e_\varepsilon(t),p_\varepsilon(t))+\int_0^t\mathcal H(\dot p_\varepsilon(r)) \,\de r+\varepsilon\int_0^t\|\dot\alpha_\varepsilon(r)\|_2^2 \,\de r=\mathcal E(\alpha_\varepsilon(0),e_\varepsilon(0),p_\varepsilon(0))+\int_0^t(\C e_\varepsilon(r),\EE \dot w(r))_2 \,\de r.\end{equation*}
\end{itemize}
\end{definition}

Condition~$(ev2)_\epsilon$ is a generalization of the standard stress constraint in plasticity. Indeed, we have the following result.

\begin{lemma}
Assume~\eqref{eq:d}--\eqref{eq:B3} and~\eqref{eq:K}. Let $(\alpha,u,e,p)\in H^1(\Omega;[0,1])\times\mathcal A$ be satisfying
\begin{equation}\label{eq:qstress}
\mathcal H(q)\ge (\C e,q)_2-2(\B(\alpha)p,q)_2-2(\nabla p,\nabla q)_2\quad\text{for all $q\in H^1(\Omega;\M)$}.
\end{equation}
If $p\in H^2(\Omega;\M)$ and $K$ is compact, then $(\alpha,u,e,p)\in H^1(\Omega;[0,1])\times\mathcal A$ satisfies
\begin{align*}
\C e(x)-2\B(\alpha(x))p(x)+2 \Delta p(x)\in K\quad\text{for a.e.\ $x\in\Omega$}.
\end{align*}
\end{lemma}
\begin{proof}
Let $q\in C_c^\infty(\Omega;\M)$. By integrating by parts in~\eqref{eq:qstress} we get 
$$\mathcal H(q)\ge (\C e-2\B(\alpha) p+2 \Delta p,q)_2\quad\text{for all $q\in C_c^\infty(\Omega;\M)$}.$$
Since $K$ is bounded, there exists $R>0$ such that $K\subset B_R(0)$. Then 
$$\mathcal H(q)\le R\|q\|_1\quad\text{for all $q\in L^2(\Omega;\M)$},$$
which implies that $\mathcal H$ is continuous on $L^2(\Omega;\M)$. Therefore, by a density argument we derive
$$\mathcal H(q)\ge (\C e-2\B(\alpha) p+2 \Delta p,q)_2\quad\text{for all $q\in L^2(\Omega;\M)$}.$$

Let us fix $\xi\in\M$. Since for all measurable set $B\subset \Omega$ we can take $q\coloneqq \chi_B\xi\in L^2(\Omega;\M)$, we derive
$$H(\xi)\ge [\C e(x)-2\B(\alpha(x)) p(x)+2 \Delta p(x)]:\xi\quad\text{for all $\xi\in \M$ and for a.e.\ $x\in\Omega$}.$$
Therefore, $\C e(x)-2\B(\alpha(x)) p(x)+2 \Delta p(x)\in \partial H(0)=K$ for a.e.\ $x\in\Omega$.
\end{proof}

The following proposition is instrumental to prove existence of $\varepsilon$-approximate viscous evolutions.

\begin{proposition}\label{prop:2210231935}
Assume~\eqref{eq:d}--\eqref{eq:B3} and~\eqref{eq:K}. Let $\varepsilon\in(0,1)$ be fixed and let $(\alpha_\varepsilon,u_\varepsilon,e_\varepsilon,p_\varepsilon)$ be satisfying~\eqref{eq:ve_2eg1}--\eqref{eq:ve_2eg2},~$(ev0)_\varepsilon$,~$(ev1)_\varepsilon$,~$(ev2)_\varepsilon$, and~$(ev3)_\varepsilon$. Then $(\alpha_\varepsilon,u_\varepsilon,e_\varepsilon,p_\varepsilon)$ is an $\varepsilon$-approximate viscous evolution, i.e., it satisfies the energy balance~$(ev4)_\varepsilon$, if and only if any of the conditions
below holds true:
\begin{itemize}
 \item[$(ev4')_\varepsilon$] for a.e.\ $t\in[0,T]$ the following hold:\\
Kuhn–Tucker equality: 
\begin{equation*}
\partial_\alpha\mathcal E(\alpha_\varepsilon(t),e_\varepsilon(t),p_\varepsilon(t))[\dot\alpha_\varepsilon(t)]+\varepsilon\|\dot \alpha_\varepsilon(t)\|_2^2= 0;
\end{equation*}
Hill’s maximum plastic work principle:
\begin{equation*}
\mathcal H(\dot p_\varepsilon(t))=(\C e_\varepsilon(t),\dot p_\varepsilon(t))_2-2(\B(\alpha_\varepsilon(t))p_\varepsilon(t),\dot p_\varepsilon(t))_2-2(\nabla p_\varepsilon(t),\nabla \dot p_\varepsilon(t))_2;
\end{equation*}
 \item[$(ev4'')_\varepsilon$] energy inequality: for all $t\in[0,T]$
 \begin{align*}
 \mathcal E(\alpha_\varepsilon(t),e_\varepsilon(t),p_\varepsilon(t))&+\int_0^t\mathcal H(\dot p_\varepsilon(r)) \,\de r+\varepsilon\int_0^t\|\dot\alpha_\varepsilon(r)\|_2^2 \,\de r\le\mathcal E(\alpha_\varepsilon(0),e_\varepsilon(0),p_\varepsilon(0))+\int_0^t(\C e_\varepsilon(r),\EE \dot w(r))_2 \,\de r.
 \end{align*}
\end{itemize}
\end{proposition}

\begin{proof}
{\bf~$(ev4)_\varepsilon\Longleftrightarrow (ev4')_\varepsilon$.} Since $\alpha_\varepsilon$, $e_\varepsilon$, and $p_\varepsilon$ satisfy~\eqref{eq:ve_2eg1} and~\eqref{eq:ve_2eg2}, then $t\mapsto \mathcal E(\alpha_\varepsilon(t),e_\varepsilon(t),p_\varepsilon(t))$ is absolutely continuous on $[0,T]$, and for a.e.\ $t\in[0,T]$ we have
\begin{align*}
   \frac{\de}{\de t} \mathcal E(\alpha_\varepsilon(t),e_\varepsilon(t),p_\varepsilon(t))&=\partial_\alpha\mathcal E(\alpha_\varepsilon(t),e_\varepsilon(t),p_\varepsilon(t))[\dot\alpha_\varepsilon(t)]+(\C e_\varepsilon(t),\dot e_\varepsilon(t))_2\\&\quad+2(\nabla p_\varepsilon(t),\nabla\dot p_\varepsilon(t))_2+2(\B(\alpha_\varepsilon(t))p_\varepsilon(t),\dot p_\varepsilon(t))_2.
\end{align*}
By~$(ev1)_\varepsilon$ we deduce that $(\dot u_\varepsilon(t),\dot e_\varepsilon(t),\dot p_\varepsilon(t))\in\mathcal A(w(t))$ and
$$(\C e_\varepsilon(t),\dot e_\varepsilon(t))=(\C e_\varepsilon(t),\EE \dot w(t)-\dot p_\varepsilon(t))_2\quad\text{for a.e.\ $t\in[0,T]$}.$$
Hence, we derive that for a.e.\ $t\in[0,T]$
\begin{align}
   \frac{\de}{\de t} \mathcal E(\alpha_\varepsilon(t),e_\varepsilon(t),p_\varepsilon(t))&=\partial_\alpha\mathcal E(\alpha_\varepsilon(t),e_\varepsilon(t),p_\varepsilon(t))[\dot\alpha_\varepsilon(t)]+(\C e_\varepsilon(t),\EE \dot w(t))_2\nonumber\\
   &\quad-(\C e_\varepsilon(t),\dot p_\varepsilon(t))_2+2(\nabla p_\varepsilon(t),\nabla\dot p_\varepsilon(t))_2+2(\B(\alpha_\varepsilon(t))p_\varepsilon(t),\dot p_\varepsilon(t))_2.\label{eq:der_ene}
\end{align}
Notice that~$(ev4)_\varepsilon$ holds if and only if for a.e.\ $t\in[0,T]$
\begin{align*}
   \frac{\de}{\de t} \mathcal E(\alpha_\varepsilon(t),e_\varepsilon(t),p_\varepsilon(t))&=-\mathcal H(\dot p_\varepsilon(t))-\varepsilon\|\dot \alpha_\varepsilon(t)\|_2^2+(\C e_\varepsilon(t),\EE \dot w(t))_2.
\end{align*}
Therefore,~$(ev4)_\varepsilon$ is equivalent to the following identity: for a.e.\ $t\in[0,T]$
\begin{align*}
&\mathcal H(\dot p_\varepsilon(t))+\varepsilon\|\dot \alpha_\varepsilon(t)\|_2^2\\
&=-\partial_\alpha\mathcal E(\alpha_\varepsilon(t),e_\varepsilon(t),p_\varepsilon(t))[\dot\alpha_\varepsilon(t)]+(\C e_\varepsilon(t),\dot p_\varepsilon(t))_2-2(\nabla p_\varepsilon(t),\nabla\dot p_\varepsilon(t))_2-2(\B(\alpha_\varepsilon(t))p_\varepsilon(t),\dot p_\varepsilon(t))_2.
\end{align*}
By~$(ev2)_\varepsilon$ and~$(ev3)_\varepsilon$ we have
\begin{align}
&\mathcal H(\dot p_\varepsilon(t))\ge (\C e_\varepsilon(t),\dot p_\varepsilon(t))_2-2(\nabla p_\varepsilon(t),\nabla\dot p_\varepsilon(t))_2-2(\B(\alpha_\varepsilon(t))p_\varepsilon(t),\dot p_\varepsilon(t))_2& &\text{for a.e.\ $t\in[0,T]$},\label{eq:Hdotpe}\\
&\varepsilon\|\dot \alpha_\varepsilon(t)\|_2^2\ge -\partial_\alpha\mathcal E(\alpha_\varepsilon(t),e_\varepsilon(t),p_\varepsilon(t))[\dot\alpha_\varepsilon(t)]& &\text{for a.e.\ $t\in[0,T]$}.\label{eq:KTin}
\end{align}
Hence, we derive that~$(ev4)_\varepsilon$ is equivalent to~$(ev4')_\varepsilon$. 

{\bf~$(ev4)_\varepsilon\Longleftrightarrow (ev4'')_\varepsilon$.} Clearly~$(ev4)_\varepsilon$ implies~$(ev4'')_\varepsilon$. Let us prove the other implication. By~\eqref{eq:der_ene},~\eqref{eq:Hdotpe}, and~\eqref{eq:KTin} we deduce for a.e.\ $t\in[0,T]$  
\begin{align*}
   \frac{\de}{\de t} \mathcal E(\alpha_\varepsilon(t),e_\varepsilon(t),p_\varepsilon(t))&\ge(\C e_\varepsilon(t),\EE \dot w(t))_2-\mathcal H(\dot p_\varepsilon(t))-\varepsilon\|\dot \alpha_\varepsilon(t)\|_2^2.
\end{align*}
Integrating over the intervall $[0,t]$ for all $t\in[0,T]$ we get
 \begin{align*}
 &\mathcal E(\alpha_\varepsilon(t),e_\varepsilon(t),p_\varepsilon(t))+\int_0^t\mathcal H(\dot p_\varepsilon(r)) \,\de r+\varepsilon\int_0^t\|\dot\alpha_\varepsilon(r)\|_2^2 \,\de r\ge\mathcal E(\alpha_\varepsilon(0),e_\varepsilon(0),p_\varepsilon(0))+\int_0^t(\C e_\varepsilon(r),\EE \dot w(r))_2 \,\de r,
 \end{align*}
 which implies~$(ev4)_\varepsilon$ by~$(ev4'')_\varepsilon$.
\end{proof}

By $(ev3)_\varepsilon$ and $(ev4')_\varepsilon$ it follows that
\begin{equation*}
\varepsilon\|\dot \alpha_\varepsilon(t)\|_2=\sup_{\beta\in F}\langle -\partial_\alpha \mathcal E(\alpha_\varepsilon(t),e_\varepsilon(t),p_\varepsilon(t)),\beta\rangle\quad\text{for a.e.\ $t\in[0,T]$},
\end{equation*}
where 
$$F\coloneqq \{\beta\in H^1(\Omega)\,:\,\beta\le 0,\,\|\beta\|_2=1\}.$$
If we define
\begin{align}
&\Phi(f)\coloneqq \sup_{\beta\in F}\langle -f, \beta\rangle& &\text{for all $f\in (H^1(\Omega))'$},\label{eq:Phi}\\
&\Psi(\alpha,e,p)\coloneqq \Phi(\partial_\alpha \mathcal E(\alpha,e,p))& &\text{for all $(\alpha,e,p)\in H^1(\Omega;[0,1])\times L^2(\Omega;\M)\times H^1(\Omega;\M)$},\label{eq:Psi}
\end{align}
then we deduce
\begin{equation}\label{eq:Psi_eps}
\Psi(\alpha_\varepsilon(t),e_\varepsilon(t),p_\varepsilon(t))=\varepsilon\|\alpha_\varepsilon(t)\|_2.
\end{equation}
Hence, we can rewrite the energy balance $(ev4)_\varepsilon$ as: for all $t\in[0,T]$.
\begin{align*}
&\mathcal E(\alpha_\varepsilon(t),e_\varepsilon(t),p_\varepsilon(t))+\int_0^t\mathcal H(\dot p_\varepsilon(r)) \,\de r+\int_0^t\|\dot\alpha_\varepsilon(r)\|_2\Psi(\alpha_\varepsilon(r),e_\varepsilon(r),p_\varepsilon(r)) \,\de r\\
&=\mathcal E(\alpha_\varepsilon(0),e_\varepsilon(0),p_\varepsilon(0))+\int_0^t(\C e_\varepsilon(r),\EE \dot w(r))_2 \,\de r.
 \end{align*}

By arguing as in~\cite[Lemma 4.4]{CrLa} we can give another characterization for the operator $\Phi$ introduced in~\eqref{eq:Phi}.

\begin{lemma}\label{lem:Psi}
Define
$$G\coloneqq \{h\in (H^1(\Omega))'\,:\,\langle h,\beta\rangle\ge 0\quad\text{for all $\beta\in H^1(\Omega)$ with $\beta\le 0$}\}$$
and 
$$d_2(f,G)\coloneqq \min\{\|g\|_2\,:\,f+g\in G\}.$$
Then
$$\Phi(f)=d_2(f,G)\quad\text{for all $f\in (H^1(\Omega))'$}.$$
\end{lemma}

\begin{proof}
The proof follows the same lines of the one of Lemma 4.4 of~\cite{CrLa}. It is enough to replace $H^m(\Omega)$ with $H^1(\Omega)$ and to use the density of $C^1(\overline\Omega)$ in $H^1(\Omega)$. 
\end{proof}

We are now ready to prove the existence result of $\varepsilon$-approximate viscous evolution for fixed $\varepsilon\in(0,1)$. 

\begin{theorem}\label{thm:eap_vis_ev}
Assume~\eqref{eq:d}--\eqref{eq:B3},~\eqref{eq:K}, and~\eqref{eq:ic}. For all $\varepsilon\in(0,1)$ there exists an $\varepsilon$-approximate viscous evolution $(\alpha_\varepsilon,u_\varepsilon, e_\varepsilon, p_\varepsilon)$ satisfying the initial condition
\begin{equation}\label{eq:eps_ic}
(\alpha_\varepsilon(0),u_\varepsilon(0), e_\varepsilon(0), p_\varepsilon(0))=(\alpha^0,u^0,e^0,p^0),
\end{equation} 
and the uniform estimates
\begin{equation}\label{eq:0_est}
 \int_0^T\|\dot\alpha_\varepsilon(r)\|_{H^1} \,\de r+\int_0^T\|\dot u_\varepsilon(r)\|_{H^1} \,\de r+\int_0^T\|\dot e_\varepsilon(r)\|_2 \,\de r+\int_0^T\|\dot p_\varepsilon(r)\|_{H^1} \,\de r\le C,   
\end{equation}
for a constant $C>0$ independent on $\varepsilon\in(0,1)$. 
\end{theorem}

\begin{proof}
{\bf Step 1}. Let us fix $\varepsilon\in(0,1)$. Starting from $(\alpha^0,u^0,e^0,p^0)$ we consider the incremental problem~\eqref{eq:min_ki} and we obtain a sequence of piecewise affine interpolants $(\alpha_k,u_k,e_k,p_k)$ from $[0,T]$ into $H^1(\Omega;[0,1])\times\mathcal A$. By the uniform bounds~\eqref{eq:k_unif_bound} and Remark~\ref{rem:emb} there exist a subsequence (not relabeled) and a quadruple $(\alpha_\varepsilon,u_\varepsilon,e_\varepsilon,p_\varepsilon)$ of functions from $[0,T]$ into $H^1(\Omega;[0,1])\times\mathcal A$ satisfying as $k\to\infty$
\begin{align}
&\alpha_k\rightharpoonup \alpha_\varepsilon \quad\text{in $H^1((0,T);H^1(\Omega))$},& &\alpha_k\to \alpha_\varepsilon \quad\text{in $C^0([0,T];L^\theta(\Omega))$},\label{eq:kweak1}\\
& p_k\rightharpoonup p_\varepsilon \quad\text{in $H^1((0,T);H^1(\Omega;\M))$},& &p_k\to p_\varepsilon \quad\text{in $C^0([0,T];L^\theta(\Omega;\M))$},\label{eq:kweak2}\\
&e_k\rightharpoonup e \quad\text{in $H^1((0,T);L^2(\Omega;\M))$} & &u_k\rightharpoonup u \quad\text{in $H^1((0,T);H^1(\Omega;\R^n))$}\label{eq:kweak3}
\end{align}
for all $\theta\in[1,6)$. Moreover, for all $t\in[0,T]$ as $k\to\infty$
\begin{align}
&\overline\alpha_k(t) \rightharpoonup\alpha_\varepsilon(t) \quad\text{in $H^1(\Omega)$},& & \overline p_k(t)\rightharpoonup p_\varepsilon(t) \quad\text{in $H^1(\Omega;\M)$}, \label{eq:kweak4}\\
&\overline e_k(t)\rightharpoonup e(t) \quad\text{in $L^2(\Omega;\M)$} & &\overline u_k(t)\rightharpoonup u(t)\quad\text{in $H^1(\Omega;\R^n)$}\label{eq:kweak5}.
\end{align}
In particular, the initial conditions~\eqref{eq:eps_ic} are satisfied, and by passing to the limit as $k\to\infty$ in~\eqref{eq:k_unif_bound2}, we derive the uniform estimates~\eqref{eq:0_est}. Moreover, the weak convergences~\eqref{eq:kweak4} and~\eqref{eq:kweak5} imply that $(u_\varepsilon(t),e_\varepsilon(t),p_\varepsilon(t))\in\mathcal A(w(t))$ and $\alpha_\varepsilon(t)\in H^1(\Omega;[0,1])$ for all $t\in[0,T]$, and that for a.e.\ $x\in\Omega$ 
\begin{equation*}
t\mapsto \alpha_\varepsilon(t,x)\quad\text{is non increasing on $[0,T]$}.
\end{equation*}
By~\eqref{eq:divCeki} and~\eqref{eq:kweak4} we conclude that for all $t\in[0,T]$ 
$$\div \C e_\varepsilon(t)=0\quad\text{in $\mathcal D'(\Omega;\R^n)$}.$$
Furthermore, by using Lemma~\ref{lem:ki_stab} together with~\eqref{eq:kweak4}--\eqref{eq:kweak5}, we deduce that for all $t\in[0,T]$
\begin{equation*}
\mathcal H(q)\ge (\C e_\varepsilon(t),q)_2-2(\B(\alpha_\varepsilon(t))p_\varepsilon(t),q)_2-2(\nabla p_\varepsilon(t),\nabla q)_2\quad\text{for all $q\in H^1(\Omega;\M)$}.
\end{equation*}
To conclude, it remains to prove~$(ev3)_\varepsilon$ and~$(ev4'')_\varepsilon$.

{\bf Step 2}. Since the functions $\alpha_k$, $e_k$, and $p_k$ are absolutely continuous, we derive that the function $t\mapsto \mathcal E(\alpha_k(t),e_k(t),p_k(t))$ is absolutely continuous and for a.e.\ $t\in[0,T]$
\begin{align*}
\frac{\de}{\de t}\mathcal E(\alpha_k(t),e_k(t),p_k(t))&=(\C e_k(t),\dot e_k(t))_2+2(\mathbb B(\alpha_k(t))p_k(t),\dot p_k(t))_2+2(\nabla p_k(t),\nabla\dot p_k(t))_2\\&\quad+\partial_\alpha\mathcal E(\alpha_k(t),e_k(t),p_k(t))[\dot \alpha_k(t)].
\end{align*}
Hence, by using also~\eqref{eq:st_daki} we conclude that for a.e.\ $t\in[0,T]$
\begin{align}
&\frac{\de}{\de t}\mathcal E(\alpha_k(t),e_k(t),p_k(t))\nonumber\\
&=(\C e_k(t),\dot e_k(t))_2+2(\mathbb B(\alpha_k(t))p_k(t),\dot p_k(t))_2+2(\nabla p_k(t),\nabla\dot p_k(t))_2-\varepsilon\|\dot \alpha_k(t)\|_2^2.\label{eq:eps_A1}
\end{align}
Notice that 
\begin{align}
(\C e_k(t),\dot e_k(t))_2&=(\C \overline e_k(t),\dot e_k(t))_2+(\C (e_k(t)-\overline e_k(t)),\dot e_k(t))_2,\label{eq:eps_A3}\\
(\mathbb B(\alpha_k(t))p_k(t),\dot p_k(t))_2&=(\mathbb B(\overline\alpha_k(t))\overline p_k(t),\dot p_k(t))_2+(\mathbb B(\overline \alpha_k(t))[p_k(t)-\overline p_k(t)],\dot p_k(t))_2\nonumber\\
&\quad+(\mathbb [\mathbb B(\alpha_k(t))-\mathbb B(\overline \alpha_k(t))]p_k(t),\dot p_k(t))_2,\\
(\nabla p_k(t),\nabla\dot p_k(t))_2&=(\nabla \overline p_k(t),\nabla\dot p_k(t))_2+(\nabla \overline p_k(t)-\nabla p_k(t),\nabla\dot p_k(t))_2.
\end{align}
Thanks to~\eqref{eq:int_for}, since $(\dot u_k(t),\dot e_k(t),\dot p_k(t))\in \mathcal A(\dot w_k(t))$ for a.e.\ $t\in[0,T]$, we get
\begin{align}
(\C \overline e_k(t),\dot e_k(t))_2=(\C \overline e_k(t),\EE \dot w_k(t))_2-(\C \overline e_k(t),\dot p_k(t))_2\quad\text{for a.e.\ $t\in[0,T]$}.
\end{align}
By~\eqref{eq:Hdotpk}, for a.e.\ $t\in[0,T]$ we get
\begin{align}
&-(\C \overline e_k(t),\dot p_k(t))_2+2(\mathbb B(\overline\alpha_k(t))\overline p_k(t),\dot p_k(t))_2+2(\nabla \overline p_k(t),\nabla\dot p_k(t))_2\nonumber\\
&\le -\mathcal H(\dot p_k(t))+C\tau_k\left(\|\dot\alpha_k(t)\|_2^2+\|\EE \dot w_k(t)\|_2^2\right).\label{eq:eps_A2}
\end{align}
By integrating~\eqref{eq:eps_A1} over the interval $[0,t]$ for all $t\in[0,T]$, and using~\eqref{eq:eps_A3}--\eqref{eq:eps_A2},~\eqref{eq:k_unif_bound}, and Remark~\ref{rem:emb} we deduce 
\begin{align*}
&\mathcal E(\alpha_k(t),e_k(t),p_k(t))-\mathcal E(\alpha^0,e^0,p^0)\\&=\int_0^t\left((\C e_k(r),\dot e_k(r))_2+2(\mathbb B(\alpha_k(r))p_k(r),\dot p_k(r))_2+2(\nabla p_k(r),\nabla\dot p_k(r))_2-\varepsilon\|\dot \alpha_k(r)\|_2^2\right) \,\de r\\
&\le \int_0^t(\C \overline e_k(r),\EE \dot w_k(r))_2 \,\de r-\int_0^t\mathcal H(\dot p_k(r)) \,\de r-\varepsilon\int_0^t\|\dot \alpha_k(r)\|_2^2 \,\de r\\
&\quad+C\int_0^t\left(\|e_k(r)-\overline e_k(r)\|_2\|\dot e_k(r)\|_2+\|p_k(r)-\overline p_k(r)\|_2\|\dot p_k(r)\|_2\right) \,\de r\\
&\quad+C\int_0^t\left(\|\alpha_k(r)-\overline\alpha_k(r)\|_2\|p_k(r)\|_{H^1}\|\dot p_k(r)\|_{H^1}+\|\nabla \overline p_k(r)-\nabla p_k(r)\|_2\|\nabla\dot p_k(r)\|_2\right) \,\de r\\
&\quad+C\tau_k\int_0^t\left(\|\dot\alpha_k(r)\|_2^2+\|\EE \dot w_k(r)\|_2^2\right)\,\de r\\
&\le \int_0^t(\C \overline e_k(r),\EE \dot w_k(r))_2 \,\de r-\int_0^t\mathcal H(\dot p_k(r)) \,\de r-\varepsilon\int_0^t\|\dot \alpha_k(r)\|_2^2 \,\de r\\
&\quad+C\tau_k\int_0^t\left(\|\dot e_k(r)\|_2^2+\|\dot p_k(r)\|_{H^1}^2+\|\dot \alpha_k(r)\|_2^2+\|\EE \dot w_k(r)\|_2^2\right)\,\de r\\
&\le \int_0^t(\C \overline e_k(r),\EE \dot w_k(r))_2 \,\de r-\int_0^t\mathcal H(\dot p_k(r)) \,\de r-\varepsilon\int_0^t\|\dot \alpha_k(r)\|_2^2 \,\de r+C\tau_k
\end{align*}
for a constant $C>0$ independent on $k$. Hence, for every $t\in[0,T]$ we have
\begin{align*}
&\mathcal E(\alpha_k(t),e_k(t),p_k(t))+\int_0^t\mathcal H(\dot p_k(r)) \,\de r+\varepsilon\int_0^t\|\dot\alpha_k(r)\|_2^2 \,\de r\le \mathcal E(\alpha^0,e^0,p^0)+\int_0^t(\C \overline e_k(r),\EE \dot w_k(r))_2 \,\de r+C\tau_k.
\end{align*}
By sending $k\to\infty$ and using the weak convergences~\eqref{eq:kweak1}--\eqref{eq:kweak5} together with the lower semicontuinty of the left-hand side (see Remark~\ref{rem:lsc}), we get~$(ev_4'')_\varepsilon$.

{\bf Step 3}. By~\eqref{eq:st_aki}, for all $\beta\in L^\infty((0,T);H^1(\Omega))$ with $\beta\le 0$ we have
\begin{align*}
0&\le (\dot{d}(\alpha_\varepsilon(t)),\beta(t))_2+2(\nabla\overline \alpha_k(t), \nabla\beta(t))_2+(\dot{\B}(\alpha_\varepsilon(t))\beta(t)\overline p_k(t),\overline p_k(t))_2+\varepsilon(\dot\alpha_k(t),\beta(t))_2\\
&\quad + (\dot{d}(\overline\alpha_k(t))-\dot{d}(\alpha_\varepsilon(t)),\beta(t))_2+([\dot{\B}(\overline\alpha_k(t))-\dot{\B}(\alpha_\varepsilon(t))]\beta(t)\overline p_k(t),\overline p_k(t))_2.
\end{align*}
By Remark~\ref{rem:emb}, as $k\to \infty$
\begin{align*}
&|(\dot{d}(\overline\alpha_k(t))-\dot{d}(\alpha_\varepsilon(t)),\beta(t))_2|\le \|\ddot d\|_\infty\|\overline\alpha_k(t)-\alpha_\varepsilon(t)\|_2\|\beta(t)\|_2\to 0,\\
&|([\dot{\B}(\overline\alpha_k(t))-\dot{\B}(\alpha_\varepsilon(t))]\beta(t)\overline p_k(t),\overline p_k(t))_2|\le  \|\ddot{\B}\|_\infty\|\overline\alpha_k(t)-\alpha_\varepsilon(t)\|_4\|\beta(t)\|_4\|\overline p_k(t)\|_4^2\to 0
\end{align*}
in view of the strong convergences in~\eqref{eq:kweak1}, see also~\eqref{eq:kweak4}. Moreover, thanks to the strong convergences in~\eqref{eq:kweak2} we get
$$(\dot{\B}(\alpha_\varepsilon(t))\beta(t)p_\varepsilon(t),p_\varepsilon(t))_2= \lim_{k\to\infty}(\dot{\B}(\alpha_\varepsilon(t))\beta(t)\overline p_k(t),\overline p_k(t))_2.$$
Hence, we can apply the Dominated Convergence Theorem to derive that
\begin{align*}
0&\le \lim_{k\to\infty}\int_0^T[(\dot{d}(\alpha_\varepsilon(r)),\beta(r))_2+2(\nabla\overline \alpha_k(r), \nabla\beta(r))_2+(\dot{\B}(\alpha_\varepsilon(r))\beta(r)\overline p_k(r),\overline p_k(r))_2+\varepsilon(\dot\alpha_k(r),\beta(r))_2]\, \de r\\
&=\int_0^T[(\dot{d}(\alpha_\varepsilon(r)),\beta(r))_2+2(\nabla\alpha_\varepsilon(r), \nabla\beta(r))_2+(\dot{\B}(\alpha_\varepsilon(r))\beta(r) p_\varepsilon(r), p_\varepsilon(r))_2+\varepsilon(\dot\alpha_\varepsilon(r),\beta(r))_2] \,\de r.
\end{align*}
If we fix $\beta\in H^1(\Omega)$ with $\beta\le 0$ a.e.\ in $\Omega$, by choosing $\beta(t)=\beta\chi_A(t)$ with $A\subset[0,T]$ arbitrary set, we get
\begin{align*}
(\dot{d}(\alpha_\varepsilon(t)),\beta)_2+(\dot{\B}(\alpha_\varepsilon(t))\beta p_\varepsilon(t),p_\varepsilon(t))_2+2(\nabla\alpha_\varepsilon(t), \nabla\beta)_2+\varepsilon(\dot\alpha_\varepsilon(t),\beta)_2\ge 0\quad\text{for all $t\in[0,T]\setminus E_\beta$},
\end{align*}
where $E_\beta$ is a  negligibile set depending on $\beta$. Since $\{\alpha\in H^1(\Omega)\,:\,\alpha\le 0\}$ is separable, we get~$(ev_3)_\varepsilon$.
\end{proof}

\section{Balanced viscosity (BV) solutions}\label{sec:res_qvs}

In the previous section for all $\varepsilon\in(0,1)$ we found an $\varepsilon$-approximate viscous evolution $(\alpha_\varepsilon,u_\varepsilon,e_\varepsilon,p_\varepsilon)$ from $[0,T]$ into $H^1(\Omega;[0,1])\times\mathcal A$ which satisfies~\eqref{eq:0_est}. We now introduce a ``slow'' time scale $s$ and we pass to the limit as $\varepsilon\to 0$, to obtain a Balanced Viscosity quasistatic evolution in the spirit of~\cite{CrLa}, see also~\cite{DMDSSo,EfMi,MieRosSav12}. We first introduce the following definition.

\begin{definition}\label{def:BV_sol}
Assume~\eqref{eq:d}--\eqref{eq:B3} and~\eqref{eq:K}. A quintuplet of Lipschitz functions $(\alpha_0,u_0,e_0,p_0,t_0)$ from $[0,S]$ into $H^1(\Omega;[0,1])\times\mathcal A\times [0,T]$ is \textit{Balanced Viscosity quasistatic evolution} in the time interval $[0,S]$ with datum $w$, if setting for all $s\in[0,S]$
\begin{align*}
w_0(s)\coloneqq w(t_0(s))\quad\text{and}\quad U_0\coloneqq \{s\in[0,S]\,:\,\text{$t_0$ is constant in a neighborhood of $s$}\},
\end{align*}
the following conditions are satisfied:
\begin{itemize}
\item[$(ev0)$] {\it irreversibility}: $t_0$ is non decreasing and surjective, and for a.e.\ $x\in\Omega$ the map
\begin{equation*}
s\mapsto \alpha_0(s,x)\quad\text{is non increasing on $[0,S]$};
\end{equation*}
\item[$(ev1)$] {\it kinematic condition and equilibrium}: for all $s\in[0,S]$ 
\begin{equation*} 
(u_0(s),e_0(s),p_0(s))\in \mathcal A(w_0(s))\quad\text{and}\quad \div \C e_0(s)=0\quad\text{in $\mathcal D'(\Omega;\R^n)$};
\end{equation*}
\item[$(ev2)$] {\it stress constraint}: for all $s\in[0,S]$
\begin{equation*}
\mathcal H(q)\ge (\C e_0(s),q)_2-2(\B(\alpha_0(s))p_0(s),q)_2-2(\nabla p_0(s),\nabla q)_2\quad\text{for all $q\in H^1(\Omega;\M)$;}
\end{equation*}
\item[$(ev3)$] {\it Kuhn–Tucker Inequality in $[0,S]\setminus U_0$}: for all $s\in[0,S]\setminus U_0$
\begin{equation*}
\partial_\alpha\mathcal E(\alpha_0(s),e_0(s),p_0(s))[\beta]\ge 0\quad\text{for all $\beta\in H^1(\Omega)$ with $\beta \le 0$;}
\end{equation*}
\item[$(ev4)$] {\it energy balance}: for all $s\in[0,S]$
\begin{align*}
&\mathcal E(\alpha_0(s),e_0(s),p_0(s))+\int_0^s\mathcal H(\dot p_0(r)) \,\de r+\int_0^s\|\dot\alpha_0(r)\|_2\Psi(\alpha_0(r),e_0(r),p_0(r)) \,\de r\\
&=\mathcal E(\alpha_0(0),e_0(0),p_0(0))+\int_0^s(\C e_0(r),\EE \dot w_0(r))_2 \,\de r,
\end{align*}
where $\Psi$ is the function defined in~\eqref{eq:Psi}, with the convention $0\cdot\infty=0$.
\end{itemize}
\end{definition}

\begin{remark} 
A similar definition for a different model with incomplete damage can be found in~\cite{CrLa}, where it is called rescaled quasistatic viscosity evolution. By Lemma~\ref{lem:Psi} and~$(ev3)$, for every $s\in [0,S]\setminus U_0$ we have
\begin{align*}
 0\le d(\partial_{\alpha}\mathcal E(\alpha_0(s),e_0(s),p_0(s)),G)&=\Psi(\alpha_0(s),e_0(s),p_0(s))=\sup_{\beta\in F}\langle-\partial_{\alpha}\mathcal E(\alpha_0(s),e_0(s),p_0(s)),\beta\rangle\le 0, 
\end{align*}
which gives
$$\Psi(\alpha_0(s),e_0(s),p_0(s))=0\quad\text{for all $s\in [0,S]\setminus U_0$}.$$
\end{remark}

Similarly to Proposition~\ref{prop:2210231935}, the following is a technical proposition instrumental to prove existence of Balanced Viscosity quasistatic solutions.

\begin{proposition}\label{prop:eq0}
Assume~\eqref{eq:d}--\eqref{eq:B3} and~\eqref{eq:K}. Let $(\alpha_0,u_0,e_0,p_0,t_0)$ be a quintuplet of Lipschitz functions from $[0,S]$ into $H^1(\Omega;[0,1])\times \mathcal A\times [0,T]$ satisfying~$(ev0)$,~$(ev1)$,~$(ev2)$, and~$(ev3)$. Then $(\alpha_0,u_0,e_0,p_0,t_0)$ is a Balanced Viscosity quasistatic evolution, i.e., it satisfies the energy balance~$(ev4)$, if and only if any of the conditions
below holds true:
\begin{itemize}
 \item[$(ev4')$] for a.e.\ $s\in[0,S]$ the following hold:\\
generalized Kuhn–Tucker equality: 
\begin{equation*}
-\partial_\alpha\mathcal E(\alpha_0(s),e_0(s),p_0(s))[\dot\alpha_0(s)]=\|\dot \alpha_0(s)\|_2\Psi(\alpha_0(s),e_0(s),p_0(s));
\end{equation*}
Hill’s maximum plastic work principle:
\begin{equation*}
\mathcal H(\dot p_0(s))=(\C e_0(t),\dot p_0(s))_2-2(\B(\alpha_0(s))p_0(s),\dot p_0(s))_2-2(\nabla p_0(s),\nabla \dot p_0(s))_2;
\end{equation*}
 \item[$(ev4'')$] energy inequality: for all $s\in[0,S]$
 \begin{align*}
 &\mathcal E(\alpha_0(s),e_0(s),p_0(s))+\int_0^s\mathcal H(\dot p_0(r)) \,\de r+\int_0^s\|\dot\alpha_0(r)\|_2\Psi(\alpha_0(r),e_0(r),p_0(r)) \,\de r\nonumber\\
 &\le\mathcal E(\alpha_0(0),e_0(0),p_0(0))+\int_0^s(\C e_0(r),\EE \dot w_0(r))_2 \,\de r.
 \end{align*}
\end{itemize}
\end{proposition}

\begin{proof}
The proof is analogous to the one of Proposition~\ref{prop:2210231935}.
\end{proof}

We eventually prove existence of Balanced Viscosity quasistatic solutions. 

\begin{theorem}\label{thm:BV_evol}
Assume~\eqref{eq:d}--\eqref{eq:B3},~\eqref{eq:K}, and~\eqref{eq:ic}. Then there exist $S>0$ and a Balanced Viscosity quasistatic solution $(\alpha_0,u_0,e_0,p_0,t_0)$ in the time interval $[0,S]$ satisfying the initial condition
\begin{equation}\label{eq:0_in-con}
(\alpha_0(0),u_0(0),e_0(0),p_0(0),t_0(0))=(\alpha^0,u^0,e^0,p^0,0).
\end{equation}
\end{theorem}

\begin{proof}

{\bf Step 1}. Let $\varepsilon\in(0,1)$ be fixed and let $(\alpha_\varepsilon,u_\varepsilon,e_\varepsilon,p_\varepsilon)$ be the $\varepsilon$-approximate viscous evolution given by Theorem~\ref{thm:eap_vis_ev}. For all $t\in[0,T]$ we define the function
$$s_\varepsilon^0(t)\coloneqq t+\int_0^t\|\dot\alpha_\varepsilon(r)\|_{H^1}\,\de r+\int_0^t\|\dot e_\varepsilon(r)\|_2\,\de r+\int_0^t\|\dot p_\varepsilon(r)\|_{H^1}\,\de r.$$
The function $s_\varepsilon^0$ is absolutely continuous, increasing, and bijective on $[0,T]$ and we have
\begin{equation}\label{eq:unif_Lip}
s_\varepsilon^0(t_2)-s_\varepsilon^0(t_1)\ge t_2-t_1\quad\text{for all $0\le t_1\le t_2\le T$}.   
\end{equation}
Let $S_\varepsilon\coloneqq s_\varepsilon^0(T)\ge T$ and let $t_\varepsilon^0\colon [0,S_\varepsilon]\to [0,T]$ be the inverse of $s_\varepsilon^0$. By~\eqref{eq:0_est} we can find a constant $C>0$ independent on $\varepsilon$ such that $S_\varepsilon\le C$ for all $\varepsilon\in(0,1)$. Hence, up to a subsequence, $S_\varepsilon\to S$ as $\varepsilon\to 0$, with $S\ge T$. 

We set $\overline S_=\sup_{\varepsilon\in(0,1)}S_\varepsilon\ge S$ and we extend every function $t_\varepsilon^0$ to the interval $[0,\overline S]$ by defining $t_\varepsilon^0(s)\coloneqq T$ for all $s\in[S_\varepsilon,\overline S]$. For all $\varepsilon\in(0,1)$ we define the following rescaled functions on $[0,\overline S]$: for all $s\in[0,\overline S]$
\begin{align*}
&\alpha_\varepsilon^0(s)\coloneqq \alpha_\varepsilon(t_\varepsilon^0(s)), & &w_\varepsilon^0(s)\coloneqq w(t_\varepsilon^0(s)),& & u_\varepsilon^0(s)\coloneqq  u_\varepsilon(t_\varepsilon^0(s)),\\
&e_\varepsilon^0(s)\coloneqq  e_\varepsilon(t_\varepsilon^0(s)),& &p_\varepsilon^0(s)\coloneqq p_\varepsilon(t_\varepsilon^0(s)). & &
\end{align*}
Notice that the functions $t_\varepsilon^0\colon [0,\overline S]\to [0,T]$ are uniformly Lipschitz in $\varepsilon\in(0,1)$ by~\eqref{eq:unif_Lip}, since
$$0\le t_\varepsilon^0(s_2)-t_\varepsilon^0(s_1)\le s_2-s_1\quad\text{for all $0\le s_1\le s_2\le \overline S$}.$$
Hence, there exists a subsequence and a Lipschitz map $t_0\colon [0,\overline S]\to [0,T]$ such that $t_\varepsilon^0\to t_0$ uniformly on $[0, \overline S]$ and weakly* in $W^{1,\infty}((0,\overline S);(0,T))$ as $\varepsilon\to 0$. In particular $t_0(0)=0$, $t_0(s)=T$ for all $s\in[S,\overline S]$, and 
$$0\le t_0(s_2)-t_0(s_1)\le s_2-s_1\quad\text{for all $0\le s_1\le s_2\le \overline S$}.$$
In particular, the map $t_0\colon [0,S]\to [0,T]$ is non decreasing and surjective. Moreover, as $\varepsilon\to 0$ we derive
$$w_\varepsilon^0(s)\to w_0(s)\coloneqq w(t_0(s))\quad\text{in $H^1(\Omega;\R^n)$ for all $s\in[0,S]$}.$$
By using the definitions of $s_\varepsilon^0$ and $t_\varepsilon^0$ we obtain
\begin{equation}\label{eq:1lip}
 \|\alpha_\varepsilon^0(s_2)-\alpha_\varepsilon^0(s_1)\|_{H^1}+\|e_\varepsilon^0(s_2)-e_\varepsilon^0(s_1)\|_2+\|p_\varepsilon^0(s_2)-p_\varepsilon^0(s_1)\|_{H^1}\le s_2-s_1\quad\text{for all $0\le s_1\le s_2\le S$}.
\end{equation}
By Ascoli-Arzela Theorem there exists a quadruple $(\alpha_0,u_0,e_0,p_0)$ of Lipschitz functions from $[0,S]$ into $H^1(\Omega;[0,1])\times \mathcal A$ satifying: for all $s\in[0,S]$ as $\varepsilon\to 0$ 
\begin{align}
& \alpha_\varepsilon^0(s)\rightharpoonup \alpha_0(s)\quad\text{in $H^1(\Omega)$}, & & u_\varepsilon^0(s)\rightharpoonup u_0(s)\quad\text{in $H^1(\Omega;\R^n)$},\label{eq:0con1}\\
& e_\varepsilon^0(s)\rightharpoonup e_0(s)\quad\text{in $L^2(\Omega;\M)$}, & & p_\varepsilon^0(s)\rightharpoonup p_0(s)\quad\text{in $H^1(\Omega;\M)$}.\label{eq:0con2}
\end{align}
In particular, by Remark~\ref{rem:emb}, for all $s\in[0,S]$  as $\varepsilon \to 0$ we have
\begin{align}
\alpha_\varepsilon^0\to \alpha_0\quad\text{in $C^0([0,S];L^\theta(\Omega))$},\qquad p_\varepsilon^0\to p_0\quad\text{in $C^0([0,S];L^\theta(\Omega;\M))$}\label{eq:0Scon}
\end{align}
for all $\theta\in[1,6)$.  Furthermore, if we  combine the uniform Lipschitz estimate~\eqref{eq:1lip} with the weak convergences~\eqref{eq:0con1}--\eqref{eq:0con2}, we get
$$\|\alpha_0(s_2)-\alpha_0(s_1)\|_{H^1}+\|e_0(s_2)-e_0(s_1)\|_2+\|p_0(s_2)-p_0(s_1)\|_{H^1}\le s_2-s_1\quad\text{for all $0\le s_1\le s_2\le S$},$$
which yields
$$\|\dot\alpha_0(s)\|_{H^1}+\|\dot e_0(s)\|_2+\|\dot p_0(s)\|_{H^1}\le 1\quad\text{for a.e.\ $s\in[0,S]$}.$$
Moreover, the initial conditions~\eqref{eq:0_in-con} hold by construction.

{\bf Step 2.} Let us prove that $(\alpha_0,u_0,e_0,p_0,t_0)$ is a Balanced Viscosity quasistatic solution on $[0,S]$. By~\eqref{eq:0con1} and~\eqref{eq:0con2} we derive that $(u_0(s),e_0(s),p_0(s))\in\mathcal A(w_0(s))$ and $\div\C e_0(s)=0$ in $\mathcal D'(\Omega;\R^n)$ for all $s\in[0,S]$, and that for a.e.\ $x\in\Omega$ the map $s\mapsto\alpha_0(s,x)$ in non increasing on $[0,S]$. Hence $(\alpha_0,u_0,e_0,p_0,t_0)$ satisfies~$(ev0)$ and~$(ev1)$.  
Let us fix $s\in[0,S]$ and $q\in H^1(\Omega;\M)$. Then $t_\varepsilon(s)\in [0,T]$ for all $s\in[0,S]$ and $\varepsilon\in(0,1)$. Hence, thanks to $(ev2)_\varepsilon$, for all $s\in[0,S]$ we get
\begin{equation*}
\mathcal H(q)\ge (\C e_\varepsilon^0(s),q)_2-2(\B(\alpha_\varepsilon^0(s))p_\varepsilon^0(s),q)_2-2(\nabla p_\varepsilon^0(s),\nabla q)_2\quad\text{for all $q\in H^1(\Omega;\M)$}.
\end{equation*}
By sending $\varepsilon\to 0$ and exploiting the convergences~\eqref{eq:0con1}--\eqref{eq:0Scon}, for all $s\in[0,S]$ we derive
\begin{equation*}
\mathcal H(q)\ge (\C e_0(s),q)_2-2(\B(\alpha_0(s))p_0(s),q)_2-2(\nabla p_0(s),\nabla q)_2\quad\text{for all $q\in H^1(\Omega;\M)$}.
\end{equation*}
Therefore, also condition~$(ev2)$ is satisfied.

{\bf Step 3.} We introduce the functions $s_0^+,s_0^-\colon [0,T]\to [0,S]$, defined by
\begin{align*}
&s_0^+(t)\coloneqq \inf\{s\in[0,S]:t_0(s)>t\}\quad\text{for $t\in[0,T)$},& &s_0^+(T)=S,\\
&s_0^-(t)\coloneqq \sup\{s\in[0,S]:t_0(s)<t\}\quad\text{for $t\in(0,T]$},& &s_0^+(0)=0.
\end{align*}
Then
\begin{equation}\label{eq:se0}
s_0^-(t)\le \liminf_{\varepsilon\to 0}s_\varepsilon^0(t)\le \limsup_{\varepsilon\to 0}s_\varepsilon^0(t)\le s_0^+(t)\quad\text{for all $t\in[0,T]$},
\end{equation}
and
$$t_0(s_0^-(t))=t=t_0(s_0^+(t))\quad\text{for all $t\in[0,T]$},\qquad s_0^-(t_0(s))\le s\le s_0^+(t_0(s))\quad\text{for all $s\in[0,S]$}.$$
In particular, the set
$$S_0\coloneqq \{t\in[0,T]\,:\,s_0^-(t)<s_0^+(t)\}$$
is at most countable and
$$U_0\coloneqq \{s\in[0,S]\,:\,\text{$t_0$ is constant in a neighborhood of $s$}\}=\bigcup_{t\in S_0}(s_0^-(t),s_0^+(t)).$$

Thanks to~\eqref{eq:se0}, for all $t\in[0,T]\setminus S_0$ we have that $s_\varepsilon^0(t)\to s_0^-(t)=s_0^+(t)$ as $\varepsilon\to 0$. Hence, by exploting~\eqref{eq:1lip}--\eqref{eq:0con2} we deduce that for all $t\in[0,T]\setminus S_0$ as $\varepsilon \to 0$ 
\begin{align}
& \alpha_\varepsilon(t)\rightharpoonup \alpha_0(s_0^-(t))\quad\text{in $H^1(\Omega)$}, & & u_\varepsilon(t)\rightharpoonup u_0(s_0^-(t))\quad\text{in $H^1(\Omega;\R^n)$},\label{eq:weakcon1}\\
& e_\varepsilon(t)\rightharpoonup e_0(s_0^-(t))\quad\text{in $L^2(\Omega;\M)$}, & & p_\varepsilon(t)\rightharpoonup p_0(s_0^-(t))\quad\text{in $H^1(\Omega;\M)$},\label{eq:weakcon2}
\end{align}
and by Remark~\ref{rem:emb} for all $t\in[0,T]\setminus S_0$ as $\varepsilon \to 0$ 
\begin{align*}
\alpha_\varepsilon(t)\to \alpha_0(s_0^-(t))\quad\text{in $L^\theta(\Omega)$}, & & p_\varepsilon(t)\to p_0(s_0^-(t))\quad\text{in $L^\theta(\Omega;\M)$}
\end{align*}
for all $\theta\in[1,6)$.

We want to show~$(ev3)$ and we define
$$A_0\coloneqq \{s\in[0,S]\,:\, \Psi(\alpha_0(s),e_0(s),p_0(s))>0\}.$$
Therefore, in order to prove~$(ev3)$ it is enough to show that $A_0\subset U_0$. Since for all $\beta\in H^1(\Omega)$ the map $s\mapsto\langle -\partial_\alpha\mathcal E(\alpha_0(s),e_0(s),p_0(s)),\beta\rangle$ is continuous on $[0,S]$, we derive that the map $s\mapsto \Psi(\alpha_0(s),e_0(s),p_0(s))$ is lower semicontinous on $[0,T]$. In particular, the set $A_0$ is open on $[0,T]$. Moreover, for all $\beta\in H^1(\Omega)$ with $\beta\le 0$ we have
$$\langle -\partial_\alpha\mathcal E(\alpha_0(s),e_0(s),p_0(s)),\beta\rangle\le \liminf_{\varepsilon\to 0}\langle -\partial_\alpha\mathcal E(\alpha_\varepsilon^0(s),e_\varepsilon^0(s),p_\varepsilon^0(s)),\beta\rangle\le \liminf_{\varepsilon\to 0}\Psi(\alpha_\varepsilon^0(s),e_\varepsilon^0(s),p_\varepsilon^0(s)),$$
which gives
\begin{equation}\label{eq:Psiliminf}
\Psi(\alpha_0(s),e_0(s),p_0(s))\le \liminf_{\varepsilon\to 0}\Psi(\alpha_\varepsilon^0(s),e_\varepsilon^0(s),p_\varepsilon^0(s)).    
\end{equation}

Let us define
$$D_0\coloneqq \{s\in (0,S)\,:\, \dot t_0(s)=0\},$$
and let us show that
\begin{equation}\label{eq:limsupte}
\limsup_{\varepsilon\to 0}\dot t_\varepsilon^0(s)>0\quad\text{for a.e.\ $s\in[0,S]\setminus D_0$}.    
\end{equation}
On the contrary, since $t_\varepsilon^0$ is non decreasing, we could find a measurable set $A\subset (0,S)\setminus D_0$ with positive measure such that
$$\lim_{\varepsilon\to 0}\dot t_\varepsilon^0(s)=0\quad\text{for all $s\in A$}.$$
Since every function $t_\varepsilon^0$ is 1-Lipschitz on $[0,S]$, by applying the Dominated Convergence Theorem we conclude that
$$\lim_{\varepsilon\to 0}\int_A \dot t_\varepsilon^0(r)\,\de r=0.$$
On the other hand, we have that $t_\varepsilon^0\rightharpoonup t_0$ weakly* in $W^{1,\infty}((0,S);(0,T))$, which gives
$$\lim_{\varepsilon\to 0}\int_A \dot t_\varepsilon^0(r)\,\de r=\int_A \dot t_0(r)\,\de r>0,$$
since $\dot t_0(s)>0$ for a.e.\ $s\in [0,S]\setminus D_0$. This leads to a contradiction and proves~\eqref{eq:limsupte}.

For a.e.\ $s\in[0,S]\setminus D_0$ we derive
$$0\le \Psi(\alpha_0(s),e_0(s),p_0(s))\le \liminf_{\varepsilon\to 0}\Psi(\alpha_\varepsilon^0(s),e_\varepsilon^0(s),p_\varepsilon^0(s))=\liminf_{\varepsilon\to 0}\varepsilon\|\dot\alpha_\varepsilon(t_\varepsilon^0(s))\|_2=\liminf_{\varepsilon\to 0}\varepsilon\frac{\|\dot\alpha_\varepsilon^0(s)\|_2}{\dot t_\varepsilon^0(s)}=0$$
by~\eqref{eq:Psi_eps},~\eqref{eq:1lip}, and~\eqref{eq:limsupte}. Therefore, we have
$$\Psi(\alpha_0(s),e_0(s),p_0(s))=0\quad\text{for a.e.\ $s\in[0,S]\setminus D_0$},$$
which gives that $A_0\subset D_0$, i.e., 
$$\dot t_0(s)=0\quad\text{for a.e.\ $s\in A_0$}.$$
Since $A_0$ is open on $[0,T]$, every $s\in A_0$ has an open neighborhood where $\dot t_0=0$. Since $t_0$ is Lipschitz, we conclude that $A_0\subset U_0$. This gives~$(ev3)$.

{\bf Step 4.} By Proposition~\ref{prop:eq0}, in order to show that $(\alpha_0,u_0,e_0,p_0,t_0)$ is a Balanced Viscosity quasistatic evolutio, it is enough to prove the energy inequality~$(ev4'')$. By using~$(ev4)_\varepsilon$ in $t=t_\varepsilon^0(s)$ together with~\eqref{eq:Psi} and the change of variable formula, for all $\varepsilon\in(0,1)$ and $s\in[0,S]$ we get
\begin{align}
&\mathcal E(\alpha_\varepsilon^0(s),e_\varepsilon^0(s),p_\varepsilon^0(s))+\int_0^s\mathcal H(\dot p_\varepsilon^0(r))\,\de r+\int_0^s\|\dot\alpha_\varepsilon^0(r)\|_2\Psi(\alpha_\varepsilon^0(r),e_\varepsilon^0(r),p_\varepsilon^0(r))\,\de r\nonumber\\
 &=\mathcal E(\alpha^0,e^0,p^0)+\int_0^{t_\varepsilon^0(s)}(\C e_\varepsilon(r),\EE \dot w(r))_2\,\de r.\label{eq:eneqe}
 \end{align}
 By the weak convergences in~\eqref{eq:0con1} and~\eqref{eq:0con2} we get
 \begin{align*}
&\mathcal E(\alpha_0(s),e_0(s),p_0(s))\le\liminf_{\varepsilon\to 0}\mathcal E(\alpha_\varepsilon^0(s),e_\varepsilon^0(s),p_\varepsilon^0(s)).
 \end{align*}
Moreover, by Remark~\ref{rem:lsc} and the fact that $\dot p_\varepsilon^0\rightharpoonup \dot p_0$ weakly* in $L^\infty((0,S);H^1(\Omega;\M))$ as $\varepsilon\to 0$, we have
\begin{align*}
\int_0^s\mathcal H(\dot p_0(r))\,\de r\le\liminf_{\varepsilon\to 0}\int_0^s\mathcal H(\dot p_\varepsilon^0(r))\,\de r,
 \end{align*}

We want to show that
\begin{equation}\label{eq:Psilsc}
\int_{A_0}\|\dot\alpha_0(r)\|_2\Psi(\alpha_0(r),e_0(r),p_0(r))\,\de r\le\liminf_{\varepsilon\to 0}\int_{A_0}\|\dot\alpha_\varepsilon^0(r)\|_2\Psi(\alpha_\varepsilon^0(r),e_\varepsilon^0(r),p_\varepsilon^0(r))\,\de r.   
\end{equation}
Let $C\subset A_0$ be a compact set and let $\psi\colon C\to [0,\infty)$ be a continuous function such that
$$\Psi(\alpha_0(r),e_0(r),p_0(r))>\psi(r)\quad\text{for all $r\in C$}.$$
By~\eqref{eq:Psiliminf} and the compactness of $C$, we deduce that for all $\varepsilon>0$ small enough
$$\Psi(\alpha_\varepsilon^0(r),e_\varepsilon^0(r),p_\varepsilon^0(r))>\psi(r)\quad\text{for all $r\in C$}.$$
By a standard approximation argument from below, since $s\mapsto \Psi(\alpha_0(r),e_0(r),p_0(r)) is lower semicontinuous$, in order to prove~\eqref{eq:Psilsc} it suffices to show that
\begin{equation}\label{eq:psiliminf}
\int_C\|\dot\alpha_0(r)\|_2\psi(r)\,\de r\le\liminf_{\varepsilon\to 0}\int_C\|\dot\alpha_\varepsilon^0(r)\|_2\psi(r)\,\de r
\end{equation}
for all compact set $C\subset A_0$ and every  continuous function $\psi\colon C\to [0,\infty)$. To this end, let
$(\varphi_i)_i\subset C_c^\infty(\Omega)$ be a dense sequence in the unit ball of $L^2(\Omega)$. Since
$$\|\dot\alpha_\varepsilon^0(r)\|_2=\sup_i(\varphi_i,\dot\alpha_\varepsilon^0(r))_2\quad\text{for all $r\in C$},$$
by the Localization Lemma (see, e.g.,~\cite[Lemma 2.3.2]{Bu}) we have
$$\int_C\|\dot\alpha_\varepsilon^0(r)\|_2\psi(r)\,\de r=\sup\sum_{i=1}^k\int_{C_i}(\varphi_i,\dot\alpha_\varepsilon^0(r))_2\psi(r)\,\de r,$$
where the supremum is taken over all integers $k$ and over all finite Borel partitions $C_1,\dots,C_k$ of $C$. For all $i$ the real-valued functions $r\mapsto (\varphi_i,\alpha_\varepsilon^0(r))_2$ are equi-Lipschitz on $[0,S]$ and converge to $r\mapsto (\varphi_i,\alpha_0(r))_2$ for all $r\in[0,S]$ as $\varepsilon\to 0$. Hence the functions $r\mapsto (\varphi_i,\dot\alpha_\varepsilon^0(r))_2$ converge to $r\mapsto (\varphi_i,\dot\alpha_0(r))_2$ weakly* in $L^\infty(0,S)$. It follows that
$$\sum_{i=1}^k\int_{C_i}(\varphi_i,\dot\alpha_0(r))_2\psi(r)\,\de r=\lim_{\varepsilon\to 0}\sum_{i=1}^k\int_{C_i}(\varphi_i,\dot\alpha_\varepsilon^0(r))_2\psi(r)\,\de r\le\liminf_{\varepsilon\to 0}\int_C\|\dot\alpha_\varepsilon^0(r)\|_2\psi(r)\,\de r.$$
This implies~\eqref{eq:psiliminf}, which gives~\eqref{eq:Psilsc}. 

It remains to study the limit as $\varepsilon\to 0$ of the last term in the right hand side of~\eqref{eq:eneqe}. By the Dominated Convergence Theorem and the weak convergences of~\eqref{eq:weakcon1} and~\eqref{eq:weakcon2} we have
$$\lim_{\varepsilon\to 0}\int_0^{t_\varepsilon^0(s)}(\C e_\varepsilon(r),\EE \dot w(r))_2\,\de r=\int_0^{t_0(s)}(\C e_0(s_0^-(r)),\EE \dot w(r))_2\,\de r.$$
Since $t_0$ is Lipschitz and $w$ satisfies~\eqref{eq:w}, we get that $w_0\in AC([0,S];H^1(\Omega;\R^n)$ and
$$\EE \dot w_0(r)=\EE \dot w(t_0(r))\dot t_0(r)\quad\text{for a.e.\ $r\in[0,S]$}.$$
Hence, by the area formula we derive
$$\int_0^{t_0(s)}(\C e_0(s_0^-(r)),\EE \dot w(r))_2\,\de r=\int_0^s(\C e_0(s_0^-(t_0(r))),\EE \dot w(t_0(r)))_2\dot t_0(r)\,\de r=\int_0^s(\C e_0(r),\EE \dot w_0(r))_2\,\de r,$$
because $s_0^-(t_0(r))=r$ for all $r\in[0,S]\setminus U_0$ and $\dot t_0(r)=0$ for all $r\in U_0$. This implies the energy inequality~$(ev4'')$, which concludes the proof.
\end{proof}

\begin{acknowledgements}
M.C. is a member of {\em Gruppo Nazionale per l'Analisi Matematica, la Probabilità e le loro Applicazioni} (GNAMPA) of {\em Istituto Nazionale di Alta Matematica} (INdAM). M.C.\ acknowledges the support of the project STAR PLUS 2020 - Linea 1 (21-UNINA-EPIG-172) ``New perspectives in the Variational modeling of Continuum Mechanics'', and of the INdAM - GNAMPA Project ``Equazioni differenziali alle derivate parziali di tipo misto o dipendenti da campi di vettori'' (Project number CUP\_E53C22001930001). V.C.\ has been funded by the MIUR - PRIN project 2017BTM7SN \emph{Variational Methods for stationary and evolution problems with singularities and interfaces} and by SEED PNR Project VarADeRo. 
\end{acknowledgements}

\end{document}